\documentclass[a4paper]{amsart}

\usepackage{amscd,amsthm,amsfonts,amssymb,amsmath,esint}
\usepackage[colorlinks=true]{hyperref}
\usepackage{fullpage}
\usepackage[all]{xy}

\usepackage{mathrsfs}

\newcommand{\msb}{\mathscr{B}}\newcommand{\msc}{\mathscr{C}}

\newcommand{\msi}{\mathscr{I}}\newcommand{\msj}{\mathscr{J}}\newcommand{\msl}{\mathscr{L}}
\newcommand{\msp}{\mathscr{P}}
\newcommand{\msr}{\mathscr{R}}\newcommand{\mst}{\mathscr{T}}

    \newcommand{\BA}{{\mathbb {A}}} 
    \newcommand{\BC}{{\mathbb {C}}} 
     \newcommand{\BF}{{\mathbb {F}}}
    \newcommand{\BG}{{\mathbb {G}}}

    \newcommand{\BM}{{\mathbb {M}}} 
     
    \newcommand{\BQ}{{\mathbb {Q}}}

     \newcommand{\BZ}{{\mathbb {Z}}}

     \newcommand{\CB}{{\mathcal {B}}}
     \renewcommand{\CD}{{\mathcal {D}}}
    \newcommand{\CE}{{\mathcal {E}}} \newcommand{\CF}{{\mathcal {F}}}
    \newcommand{\CG}{{\mathcal {G}}} 
     
     \newcommand{\CL}{{\mathcal {L}}}
     
    \newcommand{\CO}{{\mathcal {O}}}

    \newcommand{\CW}{{\mathcal {W}}}

    \newcommand{\fa}{{\mathfrak{a}}} \newcommand{\fb}{{\mathfrak{b}}}
    \newcommand{\fc}{{\mathfrak{c}}} \newcommand{\fd}{{\mathfrak{d}}}
    \newcommand{\fe}{{\mathfrak{e}}} \newcommand{\ff}{{\mathfrak{f}}}
    \newcommand{\fg}{{\mathfrak{g}}} \newcommand{\fh}{{\mathfrak{h}}}
    \newcommand{\fii}{{\mathfrak{i}}} 
     
    \newcommand{\fm}{{\mathfrak{m}}} 
     \newcommand{\fp}{{\mathfrak{p}}}
    \newcommand{\fq}{{\mathfrak{q}}} \newcommand{\fr}{{\mathfrak{r}}}
    \newcommand{\fs}{{\mathfrak{s}}} 
    \newcommand{\fu}{{\mathfrak{u}}}

     \newcommand{\fA}{{\mathfrak{A}}} \newcommand{\fB}{{\mathfrak{B}}}
    \newcommand{\fC}{{\mathfrak{C}}} \newcommand{\fD}{{\mathfrak{D}}}
    \newcommand{\fE}{{\mathfrak{E}}} \newcommand{\fF}{{\mathfrak{F}}}
    \newcommand{\fG}{{\mathfrak{G}}} \newcommand{\fH}{{\mathfrak{H}}}
    \newcommand{\fI}{{\mathfrak{I}}} \newcommand{\fJ}{{\mathfrak{J}}}
    \newcommand{\fK}{{\mathfrak{K}}} \newcommand{\fL}{{\mathfrak{L}}}
    \newcommand{\fM}{{\mathfrak{M}}} 
     \newcommand{\fP}{{\mathfrak{P}}}
    \newcommand{\fQ}{{\mathfrak{Q}}} \newcommand{\fR}{{\mathfrak{R}}}
    \newcommand{\fS}{{\mathfrak{S}}} 
     
    \newcommand{\fW}{{\mathfrak{W}}}

    \newcommand{\Gal}{{\mathrm{Gal}}} 
    
    \newcommand{\Hom}{{\mathrm{Hom}}}

    \newcommand{\Ker}{{\mathrm{Ker}}}

    \newcommand{\ord}{{\mathrm{ord}}}

    \renewcommand{\mod}{\ \mathrm{mod}\ }

    \newcommand{\Sel}{{\mathrm{Sel}}}

    \font\cyr=wncyr10

    \newcommand{\Sha}{\hbox{\cyr X}}
    \newcommand{\wh}{\widehat}

    \newcommand{\ov}{\overline}

    \newcommand{\ra}{\rightarrow} 
    
    \newcommand{\nequiv}{\equiv\hspace{-10pt}/\ }

\newcommand*\xbar[1]{%
   \hbox{%
     \vbox{%
       \hrule height 0.5pt 
       \kern0.3ex
       \hbox{%
         \kern -0.05em
         \ensuremath{#1}%
         \kern -0.05em
       }%
     }%
   }%
}
    \theoremstyle{plain}
    \newtheorem{thm}{Theorem}[section] \newtheorem{cor}[thm]{Corollary}
    \newtheorem{lem}[thm]{Lemma}  \newtheorem{prop}[thm]{Proposition}
    \newtheorem {conj}[thm]{Conjecture}

\theoremstyle{remark} 
\theoremstyle{remark} 
\theoremstyle{remark} 

    \newcommand{\Neron}{N\'{e}ron~}

    \numberwithin{equation}{section}

\begin{document}

\title{ Non-vanishing theorems for central $L$-values of some elliptic curves with complex multiplication.}

\author{John Coates, Yongxiong Li}

\begin{abstract}
The paper uses Iwasawa theory at the prime $p=2$ to prove non-vanishing theorems for the value at $s=1$ of the complex $L$-series
of certain quadratic twists of the Gross family of elliptic curves  with complex multiplication by the field $K = \BQ(\sqrt{-q})$, where $q$ is any prime $\equiv 7 \mod 8$. Our results establish some broad generalizations of the non-vanishing theorem first proven by
D. Rohrlich using complex analytic methods. Such non-vanishing theorems are important because it is known that they imply the finiteness of the Mordell-Weil group and the Tate-Shafarevich group of the corresponding elliptic curves over the Hilbert class field of $K$. It is essential for the proofs to study the Iwasawa theory of the higher dimensional
abelian variety with complex multiplication which is obtained by taking the restriction of scalars to $K$ of the particular elliptic curve with complex multiplication introduced by Gross.
\end{abstract}
\maketitle

\section{Introduction}

Let $K=\BQ(\sqrt{-q})$ be an imaginary quadratic field, where $q$ is any prime number with $q \equiv 7 \mod 8$. We fix throughout an embedding of $K$ into $\BC$. Let $\CO_K$ be the ring of integers of $K$, and write $h$ for the class number of $K$. Note that, since $q \equiv 7 \mod 8$, the prime 2 splits in $K$, say $2\CO_K=\fp\fp^*$, a fact which will underly all of our subsequent arguments with Iwasawa theory. We fix one of these primes $\fp$, and we assume from now on that we have chosen the sign of $\sqrt{-q}$ so that $\ord_\fp((1-\sqrt{-q})/2) > 0$. Let $H=K(j(\CO_K))$ denote the Hilbert class field of $K$, where $j$ denotes the classical modular function. Gross  (\cite{Gross0}, Theorem 12.2.1) has proven that there exists a unique elliptic curve $A$ defined over $\BQ(j(\CO_K))$, with complex multiplication by $\CO_K$, minimal discriminant $(-q^3)$, and which is a $\BQ$-curve in the sense that it is isogenous over $H$ to all of its conjugates.  An explicit  equation for $A$ over $H$ is given by
\begin{equation}\label{mg}
y^2 = x^3 + 2^{-4}3^{-1}mqx - 2^{-5}3^{-3}rq^2,
\end{equation}
where $m^3 = j(\CO_K)$, and $r^2 = ((12)^3 - j(\CO_K))/q$ with $r > 0$ (see \cite{Gross2}). Let $L(A/H, s)$ be the complex $L$-series of $A/H$, and write $L(A/H, \eta, s)$ for the twist of this $L$-series by any finite order abelian character $\eta$ of $H$. For $1 \leq n \leq \infty$,  let $A_{\fp^n}$ be the Galois module of $\fp^n$-division points on $A$, and define $\fF_\infty = H(A_{\fp^\infty})$. Let $\CG$ denote the Galois group of $\fF_\infty$ over $H$. The action of $\CG$ on $A_{\fp^\infty}$ defines an isomorphism $\rho_\fp: \CG \simeq \CO_\fp^{\times} = \BZ_2^{\times}$, where $\CO_\fp$
denotes the ring of integers of the completion of $\CO_K$ at $\fp$.

\begin{thm}\label{t1} Assume that $q$ is any prime such that $q \equiv 7 \mod 16.$ Then, for all characters $\nu$ of finite order of $\CG = Gal(\fF_\infty/H)$, we have $L(A/H, \nu, 1) \neq 0.$
\end{thm}

\noindent We remark that, in the special case when $\nu$ is the trivial character, D. Rohrlich \cite{Ro1} has proven that $L(A/H, 1) \neq 0$ for all primes $q \equiv 7 \mod 8$, by a completely different method using complex analytic arguments. However, it does not seem that his method can easily be extended to proving the non-vanishing of the $L(A/H, \nu, 1)$ for all characters $\nu$ of finite order of $\CG$ when $q \equiv 7 \mod 16$, and we do not know if this stronger assertion is even true for the primes $q \equiv 15 \mod 16$ (but see \cite{Ro2}, where it is shown, in particular, that $L(A/H, \nu, 1) \neq 0$ for all but a finite of characters $\nu$ of finite order of $\CG$ for all primes $q \equiv 7 \mod 8$).
If $J$ is any extension of $H$, we write, as usual, $A(J)$ for the group of $J$-rational points on $A$, and $\Sha(A/J)$ for the Tate-Shafarevich group of $A/J$. In \cite{Gross0}, Gross proved that
$A(H)$ is always finite. Parallel to Theorem \ref{t1}, we prove the following result.

\begin{thm}\label{tl} Assume that $q$ is any prime with $q \equiv 7 \mod 16.$ Then, for all finite extensions
$J$ of $H$ contained in $\fF_\infty$, both $A(J)$ and the $\fp$-primary subgroup of
$\Sha(A/J)$ are finite.
\end{thm}

\noindent We remark, however, that,  when  $q \equiv 7 \mod 16$, and in addition there is more than one prime of $H$ lying above $\fp$, then we show later (see Theorem \ref{8.3}) that we always have  $\Sha(A/\fF_\infty)(\fp) = (K_\fp/\CO_\fp)^{m_q}$ for some integer $m_q > 0.$ On the other hand it is already well known by an Euler system argument due to Kolyvagin-Gross-Zagier that $L(A/H, 1) \neq 0$ implies that both $A(H)$ and $\Sha(A/H)$ are finite.

\medskip

We further prove a non-vanishing theorem for the values at $s=1$ of the complex $L$-series of a large class of quadratic twists of $A$. Let $\msr$ denote the set of all square free positive integers $R$ of the form $R = r_1...r_k$ , where $k\geq 0$,  and $r_1, \dots, r_k$ are distinct primes such that (i) $r_i \equiv 1 \mod 4$, and  (ii) $r_i$ is inert in $K$, for $i=1,..., k$. For $R \neq 1 \in \msr$, let $A^{(R)}$ be the twist of $A$ by the quadratic extension $H(\sqrt{R})/H$. Note that such an extension is non-trivial since $K$ has odd class number. We write $L(A^{(R)}/H, s)$ for the complex $L$-series of $A^{(R)}/H$. By Deuring's theorem, $L(A^{(R)}/H, s)$ is a product of Hecke $L$-series with Grossencharacter. We shall prove the following theorem.

\begin{thm}\label{main}
Assume that  $q\equiv 7\mod 16$. Then, for all $R \in \msr$, we have $L(A^{(R)}/H, 1) \neq 0$.
\end{thm}
\begin{cor}\label{maincor}
Assume that  $q\equiv 7\mod 16$. Then, for all $R \in \msr$, both $A^{(R)}(H)$ and $\Sha(A^{(R)}/H)$ are finite.
\end{cor}
\noindent Indeed, it is well known (see \cite{KK}, \S5) that if $E$ is any twist of $A$ by a quadratic extension of $K$, then $L(E/H, 1) \neq 0$ implies that both $E(H)$ and $\Sha(E/H)$ are finite.  However, at a much more elementary level exploiting the fact that the  torsion subgroup of $A^{(R)}(H)$ is equal to $\CO_K/2\CO_K$, it has been shown independently by several authors (J. Choi \cite{JC}, and K. Li and Y. Ren \cite{LR}) that a classical 2-descent argument on $A^{(R)}/H$ proves that $A^{(R)}(H)$ is in fact finite for all $R \in \msr$. However, such an argument tells us nothing about the finiteness of even  the 2-primary subgroup of the Tate-Shafarevich group $\Sha(A^{(R)})$ for $R \in \msr$.  We remark that for the special case $q=7$, the curve $A$ is the modular elliptic curve $X_0(49)$, and the above theorem is already proved in \cite{CLTZ} by two rather different methods. In fact, we use a modification of one of these methods, due originally to C. Zhao \cite{Zhao1}, \cite{Zhao2}, to prove the above theorem.   However, the proof is considerably more delicate when $q > 7$, and makes essential use of the restriction of scalars abelian variety and its quadratic twists.

\medskip

Let $\beta = \sqrt{-q}$, and let $A^{(-\beta)}$ be the twist of $A$ by the quadratic extension $H(\sqrt{-\beta})/H$. Thus $A^{(-\beta)}$ is also an elliptic curve defined over $H$ with complex multiplication by $\CO_K$.  These elliptic curves $A^{(-\beta)}$ seem, in some sense, to be even simpler than the Gross curves $A$, and they appear not to have been
discussed in earlier literature. From the equation \eqref{mg} for $A$, one easily shows that $A^{(-\beta)}$ has the nice explicit equation
\begin{equation}\label{mgm}
y^2 = x^3 - 2^{-4}3^{-1}(j(\CO_K))^{1/3}x + 2^{-5}3^{-3}(j(\CO_K) - (12)^3)^{1/2},
\end{equation}
where it is understood that, in this equation, we take the real cube root of $j(\CO_K)$, and the square root of $j(\CO_K) - (12)^3$ lying in the upper half complex plane.  Note that this curve is defined over $H$, but it is not a $\BQ$-curve in the sense of \cite{Gross0}. Of course, this equation is not a good one at the primes of $H$ above lying above 2 and 3, but it has nevertheless the rather striking property that its discriminant is equal to 1. In fact, we shall
see later that $A^{(-\beta)}$ has good reduction outside the set of primes of $H$ lying above $\fp$ for all primes $q \equiv 7 \mod 8$ (see Corollary \ref{op}).  More generally, we point out that, for any imaginary quadratic field $K$ the equation \eqref{mgm} defines an elliptic curve with complex multiplication by the full ring of integers of this imaginary quadratic field, which is defined over an extension of degree at most 6 of the Hilbert class field of $K$, and which has good reduction outside the set of primes dividing 2 and 3. In fact, classical results due to Weber and Sohngen (see \cite{BB}) show that the curve \eqref{mgm} is defined over $H$ itself whenever the discriminant of $K$ is prime to 6.   We hope to study systematically the arithmetic properties of this elliptic curve in a subsequent paper. In the present case,  when $K = \BQ(\sqrt{-q})$ with $q$ any prime  $\equiv 7 \mod 8$, we prove the following two results about the arithmetic  of our curve \eqref{mgm}. By class field theory, there is a unique $\BZ_2$-extension $K_\infty/K$ which is unramified outside $\fp$, and, for each integer $n \geq 0$, we define $K_n$ to be the unique intermediate field with $[K_n:K] = 2^n$. Put $H_n = HK_n$, and note that $[H_n:H] = 2^n$, since the class number of $K$ is odd. Write $L(A^{(-\beta)}/H_n, s)$ for the complex $L$-series of $A^{(-\beta)}/H_n$.

\begin{thm}\label{tm}  For all primes $q \equiv 7 \mod 8$, we have $L(A^{(-\beta)}/H_n, 1) \neq 0$ for all $n \geq 0$.
\end{thm}

\begin{cor}\label{tmc} For all primes $q \equiv 7 \mod 8$, $\Sha(A^{(-\beta)}/H)$ is finite, and also $A^{(-\beta)}(H_n)$ is finite for all $n \geq 0$.
\end{cor}

\noindent We now explain a curious consequence of this theorem. Write $\fM_K$ for the set of all non-zero integers $M$ in $\CO_K$, which are prime to $q$, satisfy $M \equiv 1 \mod 4$, and are not squares in $K$.  For each $M \in \fM_K$, define
\begin{equation}\label{f1}
E = A^{(M)}
\end{equation}
to be the twist of $A$ by the extension $H(\sqrt{M})/H$, which is non-trivial because the class number of $K$ is odd. Define
$$
\CF = H(E_{\fp^2}), \, \CF_n = \CF H_n = H(E_{\fp^{n+2}}).
$$
Write $g_{E/H_n}$ and $g_{E/\CF_n}$ for the respective ranks of the Mordell-Weil groups of $E/H_n$ and $E/\CF_n$, and let $L(E/H_n, s)$ and $L(E/\CF_n,s)$ denote their respective complex $L$-series.

\begin{thm}\label{t2}  Assume that $q \equiv 7 \mod 8$, and that $M \in \fM_K$. Then, for all $n \geq 0$, we have that  $g_{E/H_n} = g_{E/\CF_n}$, and $\ord_{s=1}L(E/H_n, s) = \ord_{s=1}L(E/\CF_n, s)$.
\end{thm}

\noindent We remark that Theorems \ref{t1} and \ref{t2} do not seem to have been known before even in the special case $q=7$, when $A = X_0(49)$ (see, however, the remarks at the end of \cite{CC}). Of course, we have no way at present of hoping to prove that the four integers $g_{E/H_n},  g_{E/\CF_n}, \ord_{s=1}L(E/H_n, s)$ , and $\ord_{s=1}L(E/\CF_n, s)$ are all equal, as is predicted by the conjecture of Birch and Swinnerton-Dyer.

\medskip

Finally, we make a conjecture about the elliptic curve \eqref{mgm} in the case when $K = \BQ(\sqrt{-q})$, with $q$ now a prime satisfying $q \equiv 3 \mod 8$, so that 2 is inert in $K$.  Then, provided we now take the square root of $j(\CO_K) - (12)^3$ lying in the lower half complex plane, and assume that $q >3$, the curve \eqref{mgm} is still the twist by the quadratic extension $H(\sqrt{-\beta})/H$ of the Gross curve $A /H$, whose existence is proven in \cite{Gross0}. Thus, since $A$ has good reduction outside the set of primes of $H$ above $q$, and \eqref{mgm} has discriminant 1, the curve $A^{(-\beta)}$ will always have good reduction outside the set of primes of $H$ lying above the prime $p=2$, assuming $q >3$.

\begin{conj}\label{um} For all primes $q$ with $q \equiv 3 \mod 8$, we have $L(A^{(-q)}/H, 1) \neq 0.$
\end{conj}

\noindent However, we hasten to say that, in contrast to the proof of Theorem \ref{tm} above, we see absolutely no way at present for attacking such a conjecture using Iwasawa theory. Nevertheless, some remarkable numerical evidence in support of this conjecture has been found very recently by Dabrowski, Jedrzejak, and Szymaszkiewicz (see \cite{DJS}), who have shown that $L(A^{(-q)}/H, 1) \neq 0$ for all primes $q \equiv 3 \mod 8$ with $q \leq 10163$.
\medskip

The reader will quickly realize that the novelty of the arguments in the present paper is to show that it turns out to be much simpler to deal with some of the algebraic aspects of the Iwasawa theory of the $h$-dimensional abelian abelian variety $B/K$, which is the restriction of scalars from $H$ to $K$ of the elliptic curve $A/H$, rather than for the elliptic curve $A/H$ itself.  This technique has not been exploited widely in the literature, although it is used earlier in \cite{KK}. The key arguments used in this paper make essential use throughout of the fact that we are working with the prime $p=2$, notably in the use of Nakayama's lemma and the fact that both $K$ and the whole tower of fields $F_n \, (n\geq 0)$ all have odd class number in \S3, and in Zhao's induction method in \S9. We establish all of our analytic results for the  Iwasawa theory in sections 4 -7  to by employing the Euler system of elliptic units as defined in the Appendix of \cite{coates2}, showing that this Euler system works beautifully even for the prime $p=2$. Finally, \S9 of the paper explains
how to use Zhao's induction method for the abelian variety $B/K$.

\medskip

In further joint work in preparation with Y. Kezuka and Y. Tian \cite{CKLT}, we hope to use Iwasawa theory to prove the exact Birch-Swinnerton-Dyer formula for the order of the Tate-Shafarevich group of all the elliptic curves with complex multiplication appearing in Theorems \ref{main} and \ref{tm}. The two papers of Dabrowski, Jedrzejak, and Szymaszkiewicz \cite{DJS} and \cite{DJS2} contain some remarkable tables of numerical values of the orders of these Tate-Shafarevich groups, assuming the exact Birch-Swinnerton-Dyer formula for their order.

\medskip

The present paper grew out of \cite{CC}, which discussed only the case $q = 7$. We would like to thank the organizers of the conference "Iwasawa 2017" held at Tokyo University in July 2017 for providing us with an excellent opportunity to initially discuss the ideas developed here.  The first author would also like to thank Tsinghua University and the Morningside Center of the Chinese Academy of Sciences for generous hospitality while much of the subsequent detailed work was being developed. Finally, we thank Zhibin Liang for making  some numerical computations related to Theorem \ref{8.11}, and Jianing Li for informing us of his extremely ingenious elementary proof of Corollary \ref{8.13}.

\section{Notation and preliminaries}

We shall use the following notation throughout the rest of this paper. We write
\begin{equation}\label{1.1}
B ={\rm Res}_{H/K}(A),
\end{equation}
for the abelian variety which is obtained from $A$ by restriction of scalars from $H$ to $K$ (see \cite{Gross0}, \S 15). We define
$$
\msb = End_K(B), \, \, \mst = \msb \otimes \BQ.
$$
Then $\mst$ is a CM field of degree $h$ of $K$.  Let
$$
\psi_{A/H}: \BA^\times_H \to K^\times, \, \, \phi: \BA^\times_K \to \mst^\times
$$
be the Serre-Tate characters (see \cite{ST}, Theorem 10) attached to $A/H$ and $B/K$, respectively, where $\BA^\times_H$ (resp. $\BA^\times_K$) denotes the idele group of $H$ (resp. of $K$). Then we have
$$
\psi_{A/H}=\phi\circ N_{H/K}
$$
where $N_{H/K}$ is the norm map from  the idele group of $H$ to the idele group of $K$. The conductor of  $\phi$ is equal to $\fq = \sqrt{-q}\CO_K$. For a more detailed discussion of the following facts, see \cite{BG}, $\S 1$, and \cite{Gross0}, $\S 13$. The endomorphism
ring $\msb$ is generated over $\CO_K$ by the values $\phi(\fc)$ for $\fc$ running over all integral ideals of $K$ prime to $\fq$, and has discriminant $h^h\CO_K$ as an $\CO_K$-module. Thus  $\msb$ is an order in the field $\mst$, which is not necessarily maximal. However, it is always maximal when localized away from $h$. Also, the field $\mst$ contains only 2 roots of unity, and we have $H\cap\mst = K.$ We fix any embedding of $\mst$ in $\BC$
\begin{equation}\label{em}
i: \mst \to \BC.
\end{equation}
which extends our given embedding of $K$ in $\BC$. When there is no danger of confusion, we will omit  the embedding $i$ from the notation.

\begin{lem}\label{t3} The prime 2 is unramified in the field $\mst$, and there is a prime $\fP$ of $\mst$ above $\fp$ with residue field of order 2.
\end{lem}

\begin{proof} Since the class number $h$ of $K$ is odd, it follows that from the above remarks that 2 will be unramified in the field $\mst$. Furthermore, the finitely generated abelian group $A(H) = B(K)$ is a module for the algebra $\msb$, and the torsion subgroup of this abelian group is $\CO_K/2\CO_K$ (see \cite{Gross0}, $\S 13$). This action of $\msb$ stabilizes the torsion, and thus gives an $\CO_K$-algebra surjection
$$
\msb \to \CO_K/2\CO_K.
$$
The kernel of this homomorphism is the product of two conjugate primes $\fP, \fP^*$ in $\mst$, and $\fP$ has the desired property.
\end{proof}

\medskip

Gross \cite{Gross2} has proven the existence of a global minimal Weierstrass equation for $A$ over $H$, and we fix one such minimal equation
\begin{equation}\label{2.6}
y^2 + a_1xy +a_3y = x^3 + a_2x^2 + a_4x + a_6,
\end{equation}
whose coefficients are algebraic integers in $H$. Since $H = K(j(\CO_K))$, our fixed embedding of $K$ in $\BC$ induces an embedding of $H$ into $\BC$. The \Neron differential $\omega = dx/(2y + a_1x +a_3)$ then has a complex period lattice of the form $\CL = \Omega_\infty(A)\CO_K$, where $\Omega_\infty(A)$ is uniquely determined up to sign. Further, we have the Weierstrass isomorphism $\CW(z, \CL) : \BC/\CL \simeq A(\BC)$ given by
\begin{equation}\label{2.11}
\CW(z, \CL) = (\wp(z, \CL) - b_2/12, \frac{1}{2}(\wp'(z, \CL) - a_1(\wp(z, \CL) - b_2/12) - a_3),
\end{equation}
where $b_2 = a_1^2 + 4a_2$, and $\wp(z, \CL)$ is the classical Weierstrass $\wp$-function of the lattice $\CL$.  Put $\fG = \Gal(H/K)$. If $\fa$ is an integral ideal of $K$ with $(\fa, \fq) = 1$, we write $\sigma_\fa$ for the Artin symbol of $\fa$ in $\fG$. Given such an ideal $\fa$, $A^\fa$ will denote the elliptic curve over $H$ which is obtained by applying $\sigma_\fa$ to the coefficients of the equation \eqref{2.6} for $A$. We write $(x_\fa, y_\fa)$ for a generic point on $A^\fa$. The endomorphism $\phi(\fa)$ of the abelian variety $B$ defines a canonical $H$-isogeny
\begin{equation}\label{2.16}
\eta_{A}(\fa) : A \to A^\fa,
\end{equation}
with kernel precisely the Galois module $A_\fa$ of $\fa$-division points on $A$. Then, as is explained in $\S4$ of \cite{GS}, the pull back by $\eta_A(\fa)$ of the \Neron differential $\omega_\fa$ on $A^\fa$ must be of the form $\xi(\fa)\omega$, where $\xi(\fa)$ is a uniquely determined non-zero element of $H$.
Further, always with our fixed embedding of $H$ into $\BC$, the complex period lattice $\CL_\fa$ of $\omega_\fa$ is then equal to $\xi(\fa)\Omega_\infty(A)\fa^{-1}$, with Weierstrass isomorphism $\CW(z, \CL_\fa): \BC/\CL_\fa \simeq A^\fa(\BC)$.

\medskip

For $\alpha \in \msb$, let $B_\alpha$ to be the kernel of the endomorphism $\alpha$ on $B(\bar{K})$, and, for each integer $n \geq 1$, define $B_{\fP^n} = \cap_{\alpha \in \fP^n} B_\alpha$. We shall mainly be studying the Iwasawa theory of $B$ over the tower of fields $F_\infty/K$, where
\begin{equation}\label{1.3}
F = K(B_{\fP^2}), \, \, F_n = K(B_{\fP^{n+2}})   \, \, (n \geq 0), \, \, F_\infty = K(B_{\fP^\infty}),
\end{equation}
and $B_{\fP^\infty} = \cup_{n \geq 1}B_{\fP^n}$. Lurking in the background, we will also consider the Iwasawa theory of $A$ over the tower of fields
\begin{equation}\label{y1.3}
\fF = H(A_{\fp^2}), \, \, \fF_n = H(A_{\fp^{n+2}})   \, \, (n \geq 0), \, \, \fF_\infty = H(A_{\fp^\infty}).
\end{equation}
Moreover, as in the Introduction, let $K_\infty$ be the unique $\BZ_2$-extension of $K$ which is unramified outside $\fp$, and write $K_n$ for the unique intermediate field with $[K_n:K] =2^n$. Put
$$
G = \Gal(F_\infty/K), \, \, \, \CG = \Gal(\fF_\infty/H),
$$
and let
\begin{equation}\label{y1.4}
\rho_{\fP}: G \to \BZ_2^\times, \rho_{\fp}: \CG \to \BZ_2^\times
\end{equation}
be the characters giving the action of these two Galois groups on $B_{\fP^\infty}$ and $A_{\fp^\infty}$, respectively. In fact, for reasons which we explain in the next paragraph,  both of these characters are isomorphisms. We write $\msi$ for the ring of integers of the completion of the maximal unramified extension of $K_\fp$. For any $p$-adic Lie group
$\fH$, $\Lambda_\msi(\fH)$ will denote the Iwasawa algebra of $\fH$ with coefficients in $\msi$.

\medskip

Throughout the paper, $w$ will denote a place of $H$ lying above $\fp$. We write $H_w$ for the completion of $H$ at $w$, and $\CO_{H, w}$ for the ring of integers of $H_w$. For $\fc$ any integral ideal of $K$, with $(\fc, \fq) = 1$, let  $\wh{A^\fc_w}$ be the formal group of $A^\fc$ at $w$. It is a formal group defined over $\CO_{H, w}$ with local parameter $t_{\fc,w}  = - x_\fc/y_\fc$.

\begin{lem}\label{lt} For each integral ideal $\fc$ of $K$ prime to $\fq$, $\wh{A^\fc_w}$ is a relative
Lubin-Tate formal group in the sense of \cite{DS2} for the unramified extension $H_w/K_\fp$.
\end{lem}
\begin{proof} The canonical isogeny $\eta_{A^\fc}(\fp) : A^\fc \to A^{\fc\fp}$ induces a map of formal groups defined over  $\CO_{H, w}$
\begin{equation}\label{u3.1}
\wh{\eta_{A^\fc,\fp,w}}: \wh{A^\fc_w} \to \wh{A^{\fc\fp}_w}
 \end{equation}
 which can be realized by a formal power series $\wh{\eta_{A^\fc,\fp,w}}(t_{\fc,w})$ lying in
 $\CO_{H, w}[[t_{\fc,w}]]$. Moreover, since $\phi$ is the Serre-Tate character of $B$, it is not difficult to see that
 \begin{equation}\label{u3.2}
 \wh{\eta_{A^\fc,\fp,w}}(t_{\fc,w}) \equiv t_{\fc,w}^2 \mod w, \, \, \, \wh{\eta_{A^\fc,\fp,w}}(t_{\fc,w}) \equiv m_\fc t_{\fc,w} \mod t_{\fc,w}^2\CO_{H, w}[[t_{\fc,w}]],
 \end{equation}
 where $N_{H_w/K_\fp}(m_\fc)$ has $\fp$-order equal to $[H_w:K_\fp]$. These are precisely the
 conditions required to define a Lubin-Tate formal group relative to $H_w/K_\fp$.
 \end{proof}

 \begin{cor} The prime $w$ is totally ramified in the extension $H_w(A_{\fp^\infty})/H_w$, and the Galois group of this extension is isomorphic to $\CO_\fp^\times$.
 \end{cor}

\begin{thm}\label{ge} The abelian variety $B$ has good reduction everywhere over the field $F = K(B_{\fP^2})$, and the elliptic curve $A$ has good reduction everywhere over the field $\fF = H(A_{\fp^2})$. \end{thm}

\begin{proof}  We give the proof for $B$, and the proof for $A$ is entirely similar. The conductor of $\phi$ is prime to 2, and
thus the abelian variety $B$ has good reduction at the primes of $K$ lying above 2. Let $\phi_{ F}$ be the Serre-Tate homomorphism attached to $B$ over the field $F$, so that
$\phi_{F}= \phi \circ N_{F/K}$, where $N_{F/K}$ denotes the norm map from the idele group of $F$ to the idele group of $K$. Let $v$ be any place of $F$ which does not lie above 2,
and let $U_v$ be the group of units of the completion $F_v$ of $F$ at $v$. As is shown in \cite{ST}, $B$ will have good reduction at $v$ if and only if $\phi_{F}(U_v) = 1.$ But
$\phi_{F}(U_v)=\phi(J_{v'})$, where $v'$ denotes the place of $K$ below $v$, and $J_{v'}$ denotes the image of $U_v$ under the local norm from $F_{v}$ to $K_{v'}$. Let $\xi_K:  \BA^\times_K \to G_K^{ab}$,
where $G_K^{ab}$ denotes the Galois group of the maximal abelian extension of $K$, be Artin's global reciprocity map. Note that $\xi_K(J_{v'})$ fixes $F.$ We write $\nu_{\fP} : G_K^{ab} \to \CO_\fp^\times$ for the character giving the action of $G_K^{ab}$ on $B_{\fP^\infty}$. Now $B$ has potential good reduction everywhere, since the same assertion is true for the elliptic curve $A$ and its conjugate curves over $H$. Hence, since $v'$ does not lie above 2, by the criterion of \Neron-Ogg-Shafarevich,  we must have that $\nu_{\fP}(\xi_K(x))$ is a root of unity for each $x$ in the local units at $v'$. Moreover, by another basic property of the Serre-Tate homomorphism (see \cite{ST},  Theorem 11), we have $\nu_{\fP}(\xi_K(x)) = \phi(x)$ for every $x$ in the local units at $v'$. Now assume that $x$ lies in $J_{v'}$, so that $\nu_{\fP}(\xi_K(x))$ must belong to the subgroup $1 + \fP^2$ of $\CO_{\fp}^\times$. But this subgroup contains no non-trivial roots of unity, whence we must have that $\nu_{\fP}(\xi_K(x)) = 1$, and so $\phi(x) = 1$. Hence $B$ has good reduction everywhere over $F$, as claimed.
\end{proof}

\begin{lem} We have strict inclusions $K_\infty \subset F_\infty \subset \fF_\infty$, and $F_\infty = FK_\infty, \fF_\infty = HF_\infty$. Moreover, the two characters \eqref{y1.4} are both isomorphisms, the prime $\fp$ of $K$ is totally ramified in $F_\infty$, and all primes of $H$ above $\fp$ are totally ramified in $\fF_\infty$.
\end{lem}

\begin{proof}  We first remark that the classical theory of complex multiplication shows that
$K_\infty \subset \fF_\infty$. Let $w$ be any prime of $H$ above $\fp$. Then, by Lemma \ref{lt}, $w$ is totally ramified in $\fF_\infty$, and the character $\rho_\fp$ is an isomorphism. Thus we must have $\fF_\infty = \fF K_\infty$. Also, we must have $F_\infty \subset \fF_\infty$ since $B$ is isomorphic over $H$ to the product of the $h$ curves conjugate to $A$ under the action of $\Gal(H/K)$. It then follows that $\fp$ must be totally ramified in $F_\infty$. Furthermore, since $B$ has good reduction everywhere over $F$,  we have $F \neq K$,
and thus $\rho_\fP$ must be an isomorphism. This completes the proof.
\end{proof}

If $\fb$ is any ideal of $K$ prime to $\fp\fq$, we shall write $\tau_\fb$ for the Artin symbol of $\fb$ in
$\Gal(\fF_\infty/K)$. Note that, since $\phi$ is the Serre-Tate character of $B$, the Artin symbol $\tau_\fb$ will fix the field $F_n$ if and only if
\begin{equation}\label{y1.6}
\phi(\fb) \equiv 1 \, \mod \, \fP^{n+2}.
\end{equation}
For each $n \geq 0$, we fix a set $\fC_n$ of integral ideals of $K$, prime to $\fp\fq$ such that
\begin{equation}\label{y1.7}
\Gal(\fF_n/F_n) = \{ \tau_\fc|\fF_n \, : \fc \in \fC_n \}.
\end{equation}
Thus the elements of $\fC_n$ satisfy \eqref{y1.6}, and also give a complete set of representatives
of the ideal class group of $K$ since the restriction map from $\Gal(\fF_n/F_n)$ to $\Gal(H/K)$ is an isomorphism.

\medskip

The restriction map defines an isomorphism from  $\Gal(\fF_\infty/F_\infty)$ to $\Gal(H/K)$, and we define $\delta$ to be the unique element of $\Gal(\fF_\infty/F_\infty)$ whose restriction to $H$ is the Artin symbol $\sigma_\fp$ of $\fp$. Finally, we fix a set $\{V_n\, : \, n \geq 0\}$ of primitive $\fp^{n+2}$-division points on $A$, which are compatible in the  sense that
  \begin{equation}\label{y1.9}
\eta_A(\fp)(V_{n+1}) = V_n^\delta \, \, \, (n \geq 0).
\end{equation}
Note that $V_n^\delta$ is a primitive $\fp^{n+2}$-division point on $A^\fp$.

\section{Iwasawa theory for the abelian variety $B$ over the field $F_\infty = K(B_{\fP^\infty})$}

The aim of this section is to use some very elementary arguments from Iwasawa theory to study
descent theory on $B$ over $F_\infty = K(B_{\fP^\infty})$. Note that the proof of Theorem \ref{t4} depends crucially on the fact that $p=2$. We define $M(F_\infty)$ to be the maximal abelian $2$-extension of $F_\infty = K(B_{\fP^\infty})$, which is unramified outside the primes lying above $\fp$, and put
$$
X(F_\infty)  = Gal(M(F_\infty)/F_\infty).
$$
We recall that we have chosen the sign of $\sqrt{-q}$ so that $\ord_\fp(\sqrt{-q} - 1)/2) > 0$.

 \begin{thm}\label{t4} For all primes $q$ with $q \equiv 7 \mod 8$,  $X(F_\infty)$
 is a free finitely generated $\BZ_2$-module of rank at most $2^{k-2} - 1$, where $k = ord_\fp(\sqrt{-q} - 1).$ In particular, $X(F_\infty) = 0$ when $q \equiv 7 \mod 16$.
 \end{thm}
\noindent As is explained in more detail at the very end of \S8, the recent ingenious elementary work of Li \cite{JL} does in fact imply that $X(F_\infty) \neq 0$ for all primes $q \equiv 15 \mod 16$. We now closely follow the arguments of elementary Iwasawa theory given in \cite{CC}, \S 2 to prove Theorem \ref{t4}. Of course, $M(F_\infty)$ is Galois over $K$ by maximality, and thus $G=\Gal(F_\infty/K)$ has the usual natural continuous action of Iwasawa theory on $X(F_\infty)$. We also remark that, since $\rho_\fP$ is an isomorphism, the Galois group
$G$ is of the form $G = \Delta \times \Gamma$, where $\Delta$ is cyclic of order 2 and $\Gamma$ is isomorphic to $\BZ_2$, and all of our arguments will be based on Nakayama's lemma for either of the natural $\Delta$-actions or $\Gamma$-actions. Let $\fR=\BZ_2[\Delta]$ be the group ring of $\Delta$ over $\BZ_2$. If $V$ is any $\fR$-module, we write
as usual $V_\Delta$ for the largest quotient of $V$ on which $\Delta$ acts trivially. Similarly, if
$V$ is a compact $\Gamma$-module which is a $\BZ_2$-module, $(V)_\Gamma$ will be the largest quotient of $V$ on which $\Gamma$ acts trivially.

\begin{lem}\label{b}
The field $K_\infty$ has no non-trivial abelian $2$-extension, which is unramified outside the unique prime of $K_\infty$ above $\fp$.
\end{lem}
\begin{proof}
Denote by $M(K_\infty)$ the maximal abelian $2$-extension over $K_\infty$ which is unramified outside the unique prime above $\fp$, and let $X(K_\infty)$ be the Galois group $\Gal(M(K_\infty)/K_\infty)$. Since $M(K_\infty)$ is Galois over $K$ by maximality,  the Galois group $\Gamma=\Gal(K_\infty/K)$ acts on it continuously by lifting inner automorphisms. We claim that
$(X(K_\infty))_\Gamma=0$, which will suffice to prove what we want by Nakayama's lemma.  Denote by $\fJ$ the maximal abelian extension of $K$ in $M(K_\infty)$, so that
$Gal(\fJ/K_\infty)=(X(K_\infty))_\Gamma.$ Now, since the class number of $K$ is odd, class field theory shows immediately that $K_\infty$ itself is the maximal abelian 2-extension of $K$ which is unramified outside $\fp$. Hence $\fJ = K_\infty$, and the proof is complete.
\end{proof}

\begin{lem} Let $r_q$  denote the number of primes of $K_\infty$ lying above the prime $\fq = \sqrt{-q}\CO_K$ of $K$.  Then $r_q = 2^{k-2}$, where $k = ord_\fp(\sqrt{-q} - 1).$
\end{lem}
\begin{proof} It follows from the definition of $k$ that $\sqrt{-q}\in 1+\fp^k$ but $\sqrt{-q} \notin 1+\fp^{k+1}$, where $k \geq 2$. Noting that $K_{n}$ is the 2-part of the ray class field of $K$ modulo $\fp^{n+2}$ for all $n \geq 0$, we then conclude easily from class field theory that $\fq = \sqrt{-q}\CO_K$ splits completely in the extension $K_{k-2}$, and that each prime of $K_{k-2}$ above $\fq$ is inert in  $K_\infty$.  Thus there are precisely $r_q$ primes of $K_\infty$ above $\fq$ ,
and the proof is complete.
\end{proof}

\begin{lem}\label{t5}
We have $(X(F_\infty))_\Delta$ is an $\BF_2$-vector space of dimension at most $r_q-1$, where
$\BF_2$ denotes the field with 2 elements.
\end{lem}

\begin{proof} By the definition of the $\Delta$ action, we have $\Gal(\fI/F_\infty)=(X(F_\infty))_\Delta$, where $\fI$ is the maximal abelian extension of $K_\infty$ contained in  $M(F_\infty).$ But the only primes of $K_\infty$ which ramify in $\fI$ are the unique prime above $\fp$, and the $r_q$ primes above $\fq$. Moreover, the ramification index of each of these primes above $\fq$ in $\fI$ is precisely 2, because this is the ramification index of $\fq$ in $F$. Let
$\CD$ denote the subgroup of $\Gal(\fI/K_\infty)$ generated by the inertial subgroups of these primes above $\fq$. Thus $\CD$ is a vector space of dimension at most $r_q$ over $\BF_2$.
Now the fixed field of $\CD$ is an abelian 2-extension of $K_\infty$ unramified outside $\fp$.
Thus, by Lemma \ref{b}, this fixed field must be equal to $K_\infty$. Hence $\CD = \Gal(\fI/K_\infty)$, and the assertion of the lemma follows because $[F_\infty:K_\infty] =2$
\end{proof}

\begin{cor}\label{t6} $X(F_\infty)$ is a finitely generated $\fR$-module, which is generated by at most $2^{k-2}-1$ elements over $\fR$. In particular, $X(F_\infty))$ is a finitely generated $\BZ_2$-module.
\end{cor}

\begin{proof}
As $X(F_\infty)$ is a compact $\fR$-module, the corollary follows immediately from Lemma \ref{t5} and the  Nakayama lemma.
\end{proof}

\begin{lem}\label{t7}
$X(F_\infty)$ is a free $\BZ_2$-module, and $X(F_\infty)^\Delta=0$.
\end{lem}

\begin{proof} We have the exact sequence of finitely generated $\BZ_2$-modules
\[0\ra (X(F_\infty))^\Delta\ra X(F_\infty)\ra X(F_\infty)\ra (X(F_\infty))_\Delta\ra 0,\]
where the middle map is given by multiplication $1-\epsilon$, where $\epsilon$ denotes the non-trivial element of $\Delta$. Since $(X(F_\infty))_\Delta$ is finite by Lemma \ref{t5}, it follows that
 $X(F_\infty)^\Delta$ is also finite. But  $X(F_\infty)^\Delta$ is also a $\Gamma$-module,  whence by the theorem of Greenberg \cite{Greenberg0}, p. 94,  asserting that $X(F_\infty) $ has no nonzero finite $\Gamma$-submodule in our case  even when$p =2$,  we conclude that $(X(F_\infty))^\Delta=0$,  and also that the torsion subgroup of $X(F_\infty)$ must be zero. This completes the proof.
\end{proof}

We omit the proof (see \cite{CC}, Lemma 2.8) of the following simple algebraic lemma, whose proof was pointed out to one of us by Romyar Sharifi.

\begin{lem}\label{t8}
Let $Y$ be a free $\BZ_2$-module of finite rank, which is also a $\Delta$-module, and assume that $(Y)_\Delta=\left(\BZ/2\BZ\right)^r(r\geq 0)$.  Then $Y$ is a free $\BZ_2$-module of rank $r$.
\end{lem}

\noindent Combining Corollary \ref{t6}, and Lemmas \ref{t7} and \ref{t8}, the proof of Theorem \ref{t4} is complete.

\medskip

Remarkably, the following result is valid for all primes $q$ with $q \equiv 7 \mod 8$.

\begin{thm}\label{odd}  For all primes $q$ with $q \equiv 7 \mod 8$, the field $F_n  = K(B_{\fP^{n+2}})$ has odd class number for all $n \geq 0$,
and $F_\infty$ has no unramified abelian $2$-extension.
\end{thm}

\begin{proof} We first show that $F$ has odd class number. Let $L(F)$ be the $2$-Hilbert class field of $F$, and put $Y(F) = \Gal(L(F)/F)$.  By maximality, $L(F)$ is Galois over $K$, and so $\Delta$ acts on $Y(F)$ in the usual fashion.  Thus $Y(F)_\Delta = \Gal(J/F)$, where $J$ is the maximal abelian extension of $K$ contained in $L(F)$. Now the only primes of $K$ which are ramified in $J$ are $\fp$ and
$\fq$.  Let $\Phi$ be the inertial subgroup of $\fq$ in $\Gal(J/K)$. Then the fixed field $J^\Phi$
of $\Phi$ must be an abelian 2-extension of $K$ which is unramified outside of $\fp$. But, since $K$ has odd class number, the only abelian 2-extensions of $K$ unramified outside of $\fp$ are the fields $K_n (n \geq 0)$, and so we must have $J^\Phi = K_m$ for some $m$. But then it follows that $K_mF \subset J$, and so the extension $K_mF/F$ is unramified. But the unique prime above $\fp$ is totally ramified in the extension $K_mF/F$. Thus we must have $K_m \subset F$. However, $K_m \neq F$ because $\fq$ is ramified in $F$, and so we conclude that $m =0$.   Hence we have shown that $K$ is the fixed field of
$\Phi$, and so $\Gal(J/K) = \Phi$. But $\fq$ has ramification index 2 in the extension $J/K$, so that $\Phi$ has order 2, whence also $\Gal(J/K)$ has order 2.  It follows that necessarily $J=F$, and  so $Y(F)_\Delta = 0$. But then by Nakayama's lemma, we must have $Y(F) = 0$, proving that $F$ has odd class number.

Now let $L(F_\infty)$ be the maximal unramified abelian 2-extension of $F_\infty$, and put
$Y(F_\infty) = \Gal(L(F_\infty)/F_\infty)$. Recall that $\Gamma = \Gal(F_\infty/F)$.
Now the $\BZ_2$-extension $F_\infty/F$ is totally ramified at the unique prime of $F$ above $\fp$,
and no other prime of $F$ is ramified in it. A classical argument in Iwasawa theory then proves that
$Y(F_\infty)_\Gamma \simeq \Gal(L(F)/F)$, where again $L(F)$ denotes the 2-Hilbert class field of $F$. But we have just shown that $L(F) = F$. Hence $Y(F_\infty)_\Gamma = 0$, and so by the topological Nakayama's lemma, we must have $Y(F_\infty) = 0$. But, if we write $L(F_n)$ for the 2-Hilbert class field of $F_n$ for any $n \geq 0$, the same classical argument in Iwasawa theory
shows that $Y(F_\infty)_{\Gamma_n} \simeq \Gal(L(F_n)/F_n)$, where $\Gamma_n$ denotes the unique closed subgroup of $\Gamma$ of index $2^n$. Hence $L(F_n) = F_n$,
and so $F_n$ has odd class number for all $n \geq 0$. This completes the proof.
\end{proof}

\medskip

The above arguments make essential use of the fact that we are working with Galois groups. However, for the arithmetic applications, it is important that we translate all into assertions about Selmer groups, as is done in \cite{CC} in the special case $q = 7$. We make use of the standard notation for the Galois cohomology of Galois modules and abelian varieties. Recall that $\msb$ is the ring of $K$-endomorphisms of
the abelian variety $B$. We fix any non-zero element $\pi$ of $\msb$ such that the ideal factorization of $\pi$ in the ring of integers of $\mst$ is $\fP^r$ for some integer $r \geq 1$. Now let $L$ be any algebraic extension of $K$. As usual, we define, for each integer $n \geq 1$, the Selmer group $\Sel_{\pi^n}(B/L)$ by the exact sequence
$$
\Sel_{\pi^n}(B/L)=\Ker\left(H^1(L, B_{\pi^n})\ra \prod_v H^1(L_v, B)_{\pi^n}\right),
$$
where $v$ runs over all finite places of $L$, and $L_v$ is the compositum of the completions at $v$ of all finite extensions of $K$ contained in $L$. Passing to the inductive limit over all $n \geq 1$, and noting that $B_{\pi^\infty} = B_{\fP^\infty}$, we then define the Selmer group $\Sel_{\fP^\infty}(B/L)$
to be the inductive limit of the Selmer groups  $\Sel_{\pi^n}(B/L)$, so that we have
$$
\Sel_{\fP^\infty}(B/L)=\Ker\left(H^1(L, B_{\fP^\infty})\ra \prod_{v}H^1(L_v, B)(\fP)\right);
$$
here, for any $\msb$-module V, we write $V(\fP)$ for the submodule of elements which are annihilated by some power of $\pi$. In an entirely similar manner, the modified Selmer group  $\Sel'_{\fP^\infty}(B/L)$ is defined by
$$
\Sel'_{\fP^\infty}(B/L)=\Ker\left(H^1(L,B_{\fP^\infty})\ra \prod_{v\nmid \fp}H^1(L_v, B)(\fP)\right),
$$
where now the product is taken over all primes $v$ of $L$ which do not lie above the prime $\fp$ of $K$.

\begin{thm}\label{t9}
We have
$$
\Sel_{\fP^\infty}(B/F_\infty)= \Sel'_{\fP^\infty}(B/F_\infty) = \Hom(X(F_\infty), B_{\fP^\infty}).
$$
\end{thm}

\begin{proof} Since $B$ has good reduction everywhere over $F$, the  $G_{F}$-module $B_{\fP^\infty}$ is unramified outside the set of primes of $F$ lying above $\fp$.
Combining this with the fact that  $B_{\fP^\infty}$ is fixed by $\Gal(\bar{F}/F_\infty)$, an entirely similar argument to that given in the proof of Theorem 12 of \cite{coates1}  shows that
$$
 \Sel'_{\fP^\infty}(B/F_\infty) = \Hom(X(F_\infty), B_{\fP^\infty}).
$$
Hence the assertion of the theorem will follow once we have shown that, for the unique place $v$ of $F_\infty$ above $\fp$, we have
\begin{equation}\label{t10}
H^1(F_{\infty,v}, B)(\fP)=0.
\end{equation}
Since $F_\infty/K$ is totally ramified at the unique prime above $\fp$, it follows that the residue field of $v$ restricted to $F_n$ is always equal to  $k_{v} = \BZ/2\BZ$. Let
$B'$ denote the dual abelian variety of $B$ over $K$, so that $B'$ also has good reduction everywhere over $F$. We write $\fB'_v$ for the reduction of $B'$ modulo $v$.
Fix at first the integer $n \geq 1$. Then Tate local duality at $v$ shows that, for all integers
$m \geq 1$, $H^1(F_{n,v}, B)_{\pi^n}$ is dual to $B'(F_{n,v})/{\pi^*}^mB'(F_{n,v})$,
where $\pi^*$ denotes the complex conjugate of $\pi$ for the CM field $\mst$. Now
$v$ lies above $\fp$, and so ${\pi^*}^m$ is an automorphism of the formal group of $B'$ at $v$,  whence
$$
B'(F_{n,v})/{\pi^*}^mB'(F_{n,v}) = \fB'_v(k_v)/{\pi^*}^m\fB'_v(k_v).
$$
Passing to the inductive limit, we conclude that $H^1(F_{n,v}, B)(\fP)$ is dual to
$\fB'_v(k_v)(\fP^*)$, and so $H^1(F_{\infty,v}, B)(\fP)$ will be dual to the projective limit of the
$\fB'_v(k_v)(\fP^*)$ with respect to the norm maps up the tower $F_\infty/F$.  But the Galois group of $F_\infty/F$ acts trivially on the finite group $\fB'_v(k_v)(\fP^*)$, and so the projective limit of these groups with respect to the norm map is clearly zero. Thus completes the proof of
\eqref{t10} and the theorem.
\end{proof}

\begin{prop}\label{t11} Recalling that $\Delta = \Gal(F_\infty/K_\infty)$, we have
$$
\Sel'_{\fP^\infty}(B/F_\infty) = \Sel'_{\fP^\infty}(B/F_\infty)^\Delta.
$$
\end{prop}
\begin{proof} As before, let $\epsilon$ be the non-trivial element of $\Delta$, so that $\epsilon$ acts on $B_{\fP^\infty}$ by $-1$. Thus, if $f$ belongs to $\Hom(X(F_\infty), B_{\fP^\infty})$, we have
$(\epsilon f)(x) = -f(\epsilon x)$ for all $x$ in $X(F_\infty)$. Hence
$$
\Sel'_{\fP^\infty}(B/F_\infty)^\Delta = \Hom(X(F_\infty)/(1+\epsilon)X(F_\infty), B_{\fP^\infty}).
$$
But $(1+\epsilon)X(F_\infty) \subset X(F_\infty)^\Delta$, and this latter group is zero by Lemma \ref{t7}, whence the assertion of the proposition follows.
\end{proof}

\begin{prop}\label{t12}
For all $n\geq 0$, the restriction map yields an isomorphism
\[\Sel'_{\fP^\infty}(B/F_n)\simeq \Sel'_{\fP^\infty}(B/F_\infty)^{\Gamma_n},\]
where $\Gamma_n=\Gal(F_\infty/F_n)$.
\end{prop}

\begin{proof}
By the definition of the Selmer group, the restriction maps gives rise to the following commutative diagram with exact rows:
\[\xymatrix{
  0  \ar[r]^{} & \Sel'_{\fP^\infty}(B/F_n) \ar[d]_{} \ar[r]^{} & H^1(F_n, B_{\fP^\infty}) \ar[d]_{} \ar[r]^{} & \prod_{v\nmid\fp}H^1(F_{n,v},B)(\fP)\ar[d]_{} \\
  0 \ar[r]^{} & \Sel'_{\fP^\infty}(B/F_\infty)^{\Gamma_n} \ar[r]^{} & H^1(F_\infty, B_{\fP^\infty})^{\Gamma_n} \ar[r]^{} & \left(\prod_{w\nmid\fp}H^1(F_{\infty, w}, B)(\fP)\right)^{\Gamma_n}.}\]
Now the middle vertical map is an isomorphism. Indeed $H^2(\Gamma_n, B_{\fP^\infty}) = 0$ because $\Gamma_n$ has 2-cohomological dimension equal to 1, and a well known argument (see the proof of Lemma 2.11 in \cite{CC}) shows that also $H^1(\Gamma_n, B_{\fP^\infty}) = 0$.  Moreover, the right vertical map is injective because $B$ has good reduction at all primes $v$, and the extension $F_{\infty, w}/F_{n,v}$ is unramified when $w$ does not lie above the prime $\fp$ of $K$. The assertion of the lemma now follows.
\end{proof}

\begin{thm}\label{t13} For all $n \geq 0$, we have
${\rm Rank}_{\BZ}(B(F_n))={\rm Rank}_{\BZ}(B(K_n))$,
and the $\BZ_2$-corank of $\Sha(B/F_n)(\fP)$ is equal to the $\BZ_2$-corank of $\Sha(B/K_n)(\fP)$, for all $n \geq 0$.
\end{thm}

\begin{proof} Note first that it follows immediately from Propositions \ref{t11} and \ref{t12} that, for all $n \geq 0$, we have
\begin{equation}\label{t14}
\Sel'_{\fP^\infty}(B/F_n)^\Delta=\Sel'_{\fP^\infty}(B/F_n).
\end{equation}
Now since $\Delta$ is of order 2, one sees easily that the kernel and cokernel of the restriction map from $\Sel'_{\fP^\infty}(B/K_n)$
to $\Sel'_{\fP^\infty}(B/F_n)^\Delta$ are annihilated by 2, and so it follows from \eqref{t14} that, for all $n \geq 0$, we have
\begin{equation}\label{t15}
\BZ_2-{\rm corank \, of}  \,  \Sel'_{\fP^\infty}(B/K_n) = \BZ_2-{\rm corank \,  of} \,  \Sel'_{\fP^\infty}(B/F_n).
\end{equation}
Define the modified Shafarevich-Tate group via the exactness of the sequence
\[0\ra \Sha'(B/K_n)\ra H^1(K_n, B)\ra \prod_{v\nmid \fp}H^1(K_{n,v}, B),\]
where $v$ runs over all the finite places of $K_n$ distinct from $\fp$. Note that we then have the exact sequence
\begin{equation}\label{t16}
  0\ra B(K_n)\otimes_{\msb}(\mst_{ \fP}/\msb_{ \fP})\ra \Sel_{\fP^\infty}'(B/K_n)\ra \Sha'(B/K_n)({\fP})\ra 0.
\end{equation}
We have an entirely similar exact sequence for the field $F_n$. Denote by $g_{K_n}, t_{K_n}$ the $\BZ_2$-corank of $B(K_n)\otimes_{\msb}(\mst_\fP/\msb_{\fP})$, and of $\Sha'(B/K_n)({\fP})$, respectively. Define $g_{F_n},t_{F_n}$ in an entirely analogous fashion for the field $F_n$. It follows immediately from \eqref{t15} that we have
\begin{equation}\label{t17}
g_{K_n}+t_{K_n}=g_{F_n}+t_{K_n}.
\end{equation}
Note further that $t_{K_n}\leq t_{F_n}$ and $g_{K_n}\leq g_{F_n}$, because the restriction maps
\[\Sha'(B/K_n)(\fP) \ra \Sha'(B/F_n)(\fP) \,\, {\rm and} \,\, B(K_n)\otimes_{\msb}(\mst_{\fP}/\msb_{\fP}) \ra B(F_n)\otimes_{\msb}(\mst_{\fP}/\msb_{ \fP}) \]
have finite kernels. Therefore, we conclude from \eqref{t17} that $g_{K_n}=g_{F_n}, t_{K_n}=t_{F_n}$. The first assertion of the theorem now follows easily.  For the second assertion, note that we have the exact sequence
\[0\ra \Sha(B/K_n)(\fP) \ra \Sha'(B/K_n)(\fP) \ra H^1(K_{n,w}, B)(\fP),\]
where $w$ denotes the unique prime of $K_n$ above $\fp$. But an entirely similar argument with Tate local duality to that used above in the proof of Theorem \ref{t9} shows that the group on the extreme right of this last exact sequence is finite. Hence $t_{K_n}$ is equal to the $\BZ_2$-corank of $\Sha(B/K_n)(\fP)$. Similarly, we find that $t_{F_n}$ is the  $\BZ_2$-corank of $\Sha(B/F_n)(\fP)$, and the proof of the theorem is now complete.
\end{proof}

\section{Ellitpic units for the field $F_\infty = K(B_{\fP^\infty}).$}

The aim of the present section is to define, for every $n \geq 0$, a suitable group of elliptic units for the field $F_n = K(B_{\fP^{n+2}})$, which we will denote by $C(F_n)$, and to prove the existence of a suitable interpolating power series for them. We use a variant of the method first pioneered in \cite{CW3}.

\medskip

\medskip

We first  determine the conductors of some of the abelian extensions of $K$ which arise in the rest of the section.

\begin{lem}\label{g2.6} Let $\ff$ be any integral ideal of $K$ which is divisible by the conductor $\fq$ of $\phi$. Then $H(A_\ff)$ is equal to the ray class field of $K$ modulo $\ff$.
\end{lem}

\begin{proof} The classical theory of complex multiplication shows that the ray class field of $K$ modulo $\ff$ is contained in the field $H(A_\ff)$. Conversely, suppose that $\alpha$ is an element of $K$
with $\ord_v(\alpha - 1) \geq \ord_v(\ff)$ for all places $v$ of $K$ dividing $\ff$, and write $(\alpha)_K$ and $(\alpha)_H$ for the respective principal ideals of $K$ and $H$ generated by $\alpha$.
Note that the abelian extension $H(A_\ff)/K$ is unramified outside the set of primes of $K$ dividing $\ff$, because all primes of bad reduction for $A/H$ must divide the ideal of $H$ generated by $\ff$.
Write $m_K$ for the Artin symbol of $(\alpha)_K$ in $\Gal(H(A_\ff)/K)$. We must show that $m_K$ fixes $A_\ff$.  To do this, let us consider the Artin symbol $m_H$ of $(\alpha)_H$ for the abelian extension
$H(A_\ff)/H$. By the definition of the Grossencharacter, the Artin symbol $m_H$ acts on $A_\ff$ by multiplication by the endomorphism
$$
\psi_{A/H}((\alpha)_H) = \phi((\alpha)_K^h) = \alpha^h,
$$
whence it is clear that $m_H$ fixes $A_\ff$.  But by the functorality of the Artin symbol, we know that $m_K = m_H^h$, and so it follows that $m_K$ also fixes $A_\ff$, as required.
\end{proof}

Recall that $\fF_n = H(A_{\fp^{n+2}}) = HF_n$.

\begin{lem}\label{2.7} For all $n \geq 0$, the conductors of the abelian extensions $\fF_n/K$ and  $F_n/K$ are both equal to $\ff_n = \fq\fp^{n+2}$.
\end{lem}

\begin{proof} By the previous lemma, we know that $H(A_{\ff_n})$ is equal to the ray class field of $K$ modulo $\ff_n$. But
$$
F_n \subset \fF_n \subset H(A_{\ff_n}),
$$
and so the conductors of $F_n/K$  and $\fF_n/K$ must both divide $\ff_n$. But $K_n/K$ has conductor equal to $\fp^{n+2}$, since otherwise it would be contained in the field $HK_{n-1}$, and this is impossible because $\fp$ has ramification index $2^n$ in $K_n$.  Since $K_n \subset F_n$, it follows that the conductors of $F_n/K$ and $\fF_n/K$ must both be divisible by $\fp^{n+2}$.
Moreover, $B$ has good reduction everywhere over the field $F_n$, and thus its Grossencharacter over this field, which is $\phi \circ N_{F_n/K}$, must have trivial conductor.
Hence the conductor $\fq$ of $\phi$ must divide the conductor of $F_n/K$.  Similarly, $A$ has good reduction everywhere over $\fF_n$, whence again $\fq$ must divided the conductor of $\fF_n/K$. As $\fp$ and $\fq$ are relatively prime, this completes the proof.
\end{proof}

While our aim is to define a group of elliptic units for each of the fields $F_n = K(B_{\fP^{n+2}})$, we need first to discuss the appropriate group of elliptic units for the fields $\fF_n = H(A_{\fp^{n+2}})$.   Let us introduce the index set
\begin{equation}\label{ni}
\msj = \{\alpha: \, \alpha \in \CO_K, \alpha \neq 0, \pm 1, (\alpha, 6\fq) = 1, \, \alpha \equiv 1 \mod \fp^2 \}.
\end{equation}
The congruence $\alpha \equiv 1 \mod \fp^2$ imposed on the elements of $\msj$ is not strictly necessary, but we will use it to avoid some technical complications in later arguments.
For each $\alpha \in \msj$, we defined the rational function $g_{\alpha,A}(P)$ on $A/H$ by
\begin{equation}\label{2.8}
g_{\alpha,A}(P) = \prod_{V}(x(P) - x(V))^{-1},
\end{equation}
where $V$ runs over any set of representatives of the set of non-zero elements of the Galois module $A_\alpha$ modulo $\pm 1$. Here $P = (x,y)$ is a generic point on \eqref{2.6}. As is shown by a very elementary argument in the Appendix of \cite{coates2},
there exists a unique non-zero $c_\alpha(A) \in H$ such that the normalized function
\begin{equation}\label{2.9}
\fg_{\alpha,A}(P) = c_\alpha(A)g_{\alpha,A}(P)
\end{equation}
satisfies
\begin{equation}\label{2.10}
\fg_{\alpha,A}(\beta(P)) = \prod_{W \in A_\beta}\fg_{\alpha,A}(P\oplus W)
\end{equation}
for all non-zero $\beta$ in $\CO_K$ with $(\beta, \alpha) = 1$; here the symbol $\oplus$ denotes the group law on the elliptic curve $A$. Recalling  that $\CL = \Omega_\infty(A)\CO_K$ is the period lattice of the \Neron differential on our generalized Weierstrass equation \eqref{2.6} for $A/H$, we define a primitive $\sqrt{-q}$-division point on $A$ by
\begin{equation}\label{2.12}
Q = \CW(\Omega_\infty(A)/\sqrt{-q}, \CL).
\end{equation}
We then define
\begin{equation}\label{2.13}
\fR_{\alpha, A}(P) = \prod_{\tau \in \Gal(H(A_\fq)/H)} \fg_{\alpha,A}(P \oplus Q^\tau),
\end{equation}
which is thus a rational function on $A$ with coefficients in $H$. Plainly $ \fR_{\alpha, A}(P)$  depends only on the orbit of $Q$ under the action of $\Gal(H(A_\fq)/H)$, but note that, in view of Lemma \ref{g2.6}, this Galois group does not act transitively on the set of all primitive $\fq$-division points of $A$. It is this which guarantees the all important fact that $\fR_{\alpha, A}(P) \neq
\fR_{\alpha, A}(\ominus P)$, where $\ominus$ denotes the subtraction on the elliptic curve.
It is also important to define intrinsic analogues of the rational function $\fR_{\alpha, A}(P)$
for every conjugate curve  of $A/H$ under the action of the Galois group $\fG = \Gal(H/K)$.   Let $Q(\fa)$ be the $\fq$-division point on $A^\fa$  defined by
\begin{equation}\label{2.14}
Q(\fa) = \CW(\xi(\fa)\Omega_\infty(A)/\sqrt{-q}, \CL_{\fa}),
\end{equation}
and define the rational function on $A^\fa/H$ by
\begin{equation}\label{2.15}
\fR_{\alpha, A^\fa}(P) = \prod_{\tau \in \Gal(H(A^\fa_\fq)/H)} \fg_{\alpha,A^\fa}(P \oplus Q(\fa)^\tau).
\end{equation}
Thanks to Lemma \ref{g2.6}, the fields $H(A^\fa_\fq)$ are all equal to $H(A_\fq)$, and then (see
$\S4$ of \cite{GS}) the Artin symbol of $\fa$ in $\Gal(H(A_\fq)/K)$ maps $Q$ to $Q(\fa)$. It follows easily that, whenever $\fa$ and $\fb$ are integral ideals of $K$,  which are prime to $\fq$ and lie in the same ideal class, then necessarily $\fR_{\alpha, A^\fa}(P) =  \fR_{\alpha, A^\fb}(P) $.

\medskip

We first define the group of elliptic units for the field $\fF_n = H(A_{\fp^{n+2}})$ for any integer $n \geq 0$.  Recall that $V_n$ denotes a primitive
$\fp^{n+2}$-division point on the curve $A$. We then define $C(\fF_n)$ to be the subgroup of
$\fF_n^\times$ which is generated by all conjugates of $\fR_{\alpha, A}(V_n)$ under the action
of the $\Gal(\fF_n/K)$ and all $ \alpha \in \msj$. We shall show below that the elements of $C(\fF_n)$ are indeed global units. We recall that, if $\fc$ is any integral ideal of $K$ prime to $\fp\fq$, $\tau_\fc$ denotes the Artin symbol of $\fc$ in $\Gal(\fF_\infty/K)$. Since $\tau_\fc(A) = A^\fc$ and $\tau_c(V_n) = \eta_A(\fc)(V_n)$, the action of $\Gal(\fF_n/K)$ on $\fR_{\alpha, A}(V_n)$ is given by the following lemma.

\begin{lem}\label{nc0} For all $n \geq 0$, and all integral ideals $\fc$ of $K$ prime to $\fp\fq$, we have
\begin{equation}\label{2.16}
\tau_\fc(\fR_{\alpha, A}(V_n)) = \fR_{\alpha, A^\fc}(\eta_A(\fc)(V_n)).
\end{equation}
\end{lem}
\noindent We next discuss the behaviour of $C(\fF_n)$ under the norm map $N_{\fF_n/\fF_{n-1}}$
from $\fF_n$ to $\fF_{n-1}$.  Now the Artin symbol of $\fp$ cannot be defined in $\Gal(\fF_\infty/K)$. However, we recall that
$\Gal(\fF_\infty/F_\infty)$ is isomorphic to $\fG$ under restriction, and $\delta$
is the unique element of $\Gal(\fF_\infty/F_\infty)$  whose restriction to $\fG$ is the Artin symbol of $\fp$. Now we have the $H$-isogeny
$$
\eta_A(\fp): A \to A^\fp,
$$
whose kernel is precisely $A_\fp$. Thus $\eta_A(\fp)(V_n)$ must be a primitive $\fp^{n+1}$-division
point on the curve $A^\fp$. On the other hand, it is shown in Theorem 4 of the Appendix of \cite{coates2} that the rational functions $\fg_{\alpha, A}(P)$ behave nicely with respect to isogenies of degree prime to $\alpha$. Since $(\fp, \alpha) = 1$ because $(\alpha, 6) = 1$, the following lemma then follows easily.

\begin{lem} We have
\begin{equation}\label{nc8}
 {\fR}_{\alpha, A^\fp}(\eta_A(\fp)(P)) = \prod_{V \in A_\fp} \fR_{\alpha, A}(P \oplus V).
\end{equation}
\end{lem}
\noindent Now for $n \geq 1$, we have $[\fF_n:\fF_{n-1}] = 2$ and the conjugates of $V_n$ under the action of $\Gal(\fF_n/\fF_{n-1})$ are precisely the $V_n \oplus V$ with $V \in A_\fp$. Hence
we immediately obtain

\begin{cor} For all $n \geq 1$, we have $N_{\fF_n/\fF_{n-1}}(\fR_{\alpha, A}(V_n)) =  {\fR}_{\alpha, A^\fp}(\eta_A(\fp)(V_n))$.
\end{cor}
\noindent On the other hand, we have already fixed in section 2 (see \eqref{y1.9}) the convention that we have chosen the primitive $\fp^{n+1}$-division point $V_{n-1}$ so that $\delta(V_{n-1}) = \eta_A(\fp)(V_n)$. Hence the above corollary can be rewritten as giving
$$
N_{\fF_n/\fF_{n-1}}(\fR_{\alpha, A}(V_n)) = \delta(\fR_{\alpha, A}(V_{n-1})) \, \, \, (n \geq 1).
$$
Hence we immediately obtain the theorem.

\begin{thm}\label{nc1} For all $n \geq 1$, we have
$$
N_{\fF_n/\fF_{n-1}}(\delta^{-n}(\fR_{\alpha, A}(V_n))) = \delta^{-(n-1)}(\fR_{\alpha, A}(V_{n-1})).
$$
\end{thm}
\begin{cor} Every element of $C(\fF_n)$ is a global unit.
\end{cor}
\noindent Indeed, since every prime of $\fF_n$ which does not lie above $\fp$ is unramified in $\fF_\infty$
and its decomposition group is of finite index, it follows easily from the universal norm property
of Theorem \ref{nc1} (see Lemma 5 of \cite{coates2}) that the only primes which can possibly divide $\fR_{\alpha, A}(V_n)$ must divide
$\fp$. On the other hand, the following classical lemma shows that  $\fR_{\alpha, A}(V_n)$
is a unit at each prime $w$ of $H$ above $\fp$ since the power series in Lemma \ref{u}, in the special case when $\fb = \CO_K$, will certainly converge at $t_w(V_n)$ to a unit at $w$. Let $\fb$ be any integral ideal of $K$ prime to $\fq$.  Recall that  $P_\fb = (x_\fb, y_\fb)$ is the generic point on the equation
for $A^\fb$, and that $t_{\fb, w} = -x_\fb/y_\fb$ is a parameter for the formal group of $A^\fb$ at $w$, where $w$ is any prime of $H$ above $\fp$.

\begin{lem}\label{u} For each integral ideal $\fb$ of $K$ prime to $\fq$, the $t_{\fb, w}$-expansion of $\fR_{\alpha, A^\fb}(P_\fb)$ is a unit in the power series ring $\CO_{H,w}[[t_{\fb, w}]]$.
\end{lem}
\begin{proof} This is a well known classical argument (see the proof of
Lemma 8 of \cite{coates2}) and we omit the details. Note that we also use the fact that $c_{A^\fb}(\lambda)$ is a unit in $\CO_{H, w}$ because it is shown in the Appendix to \cite{coates2}
that $c_{\alpha}(A^\fb)^{12} = \Delta(A^\fb)^{(N\alpha - 1)}/\alpha^{12}$, where $\Delta(A^\fb)$
is the discriminant of our global minimal equation for $A^\fb$, and we have $(\fp, \Delta(A^\fb)\alpha) = 1$.
\end{proof}

\medskip

We now turn to the fields $F_n = K(B_{\fP^{n+2}})$, where, of course, the only natural thing to do is to define the group of elliptic units $C(F_n)$ of $F_n$ by
\begin{equation}\label{nc2}
C(F_n) = N_{\fF_n/F_n}(C(\fF_n)),
\end{equation}
where $N_{\fF_n/F_n}$ denotes the norm map from $\fF_n$ to $F_n$. Thus, writing
\begin{equation}\label{nc3}
u_{\alpha, n} = N_{\fF_n/F_n}(\fR_{\alpha, A}(V_n)),
\end{equation}
these $u_{\alpha, n}$ are global units in $F_n$, and, since the restriction of $\delta$
to $\fF_n$ lies in $\Gal(\fF_n/F_n)$, we conclude immediately from Theorem \ref{nc1} that
\begin{equation}\label{nc4}
N_{F_n/F_{n-1}}(u_{\alpha, n}) = u_{\alpha, n-1} \, \, \, (n \geq 1).
\end{equation}
We write
\begin{equation}\label{eu}
u_{\alpha, \infty} = (u_{\alpha, n})
\end{equation}
for this norm compatible system of elliptic units. However, unlike the situation for the field $\fF_\infty$, it seems that these units $u_{\alpha, n}$ cannot be obtained for all $n \geq 0$ by evaluating a single rational function on $A$ at the point $V_n$. All we can do in this direction is the following. Recall that $\fC_n$ denotes a set of integral
ideals of $K$ prime to $\fp\fq$ whose Artin symbols in $\Gal(\fF_n/K)$ give precisely $\Gal(\fF_n/F_n)$, and define the rational function $D_{\alpha, n}(P)$ on $A/H$ by
\begin{equation}\label{nc5}
D_{\alpha, n}(P) = \prod_{\fc \in \fC_n}\fR_{\alpha, A^\fc}(\eta_A(\fc)(P)).
\end{equation}

\begin{lem}\label{nc6} For all integers $n$ and $k$ with $n \geq 0$ and $0 \leq k \leq n$, we have
$D_{\alpha, n}(V_k) = u_{\alpha, k}.$
\end{lem}
\begin{proof} This clear from Lemma \ref{nc0} and the fact that $\Gal(\fF_n/F_n)$ is isomorphic
under restriction to $\Gal(\fF_k/F_k)$ for $0 \leq k \leq n$. \end{proof}
\begin{lem}\label{nc7} For each prime $w$ of $H$ above $\fp$, the rational function
$D_{\alpha, n}(P)$ has a $t_w$-expansion $d_{\alpha,n}(t_w)$ which lies in $\CO_{H,w}[[t_w]]$,
and is a unit in this ring.
\end{lem}
\begin{proof} By Lemma \ref{u}, for each $\fc \in \fC_n$, the rational function $\fR_{\alpha, A^\fc}(P_\fc)$ has a $t_{\fc,w}$-expansion which is a unit in $\CO_{H,w}[[t_{\fc, w}]]$. Now the isogeny $\eta_{A}(\fc) : A \to A^\fc$ induces a homomorphism of formal groups
$$
\widehat{\eta_{A,\fc,w}}: \widehat{A_w} \to \widehat{A^\fc_w},
$$
where $\widehat{A_w}$ and $\widehat{A^\fc_w}$ denote the respective formal groups of $A$ and $A^\fc$ at $w$. In particular, such a homomorphism is realized by a formal power series $t_{\fc,w} = \widehat{\eta_{A,\fc,w}}(t_w)$ in $\CO_{H, w}[[t_w]]$ with $\widehat{\eta_{A,\fc,w}}(0) = 0$. Substituting this formal power series in the unit power series expansion in $\CO_{H, w}[[t_{\fc,w}]]$ of $\fR_{\alpha, A^\fc}(P_\fc)$, and taking the product over all $\fc \in \fC_n$, the assertion of the lemma follows.
\end{proof}

Now the power series ring $\CO_{H, w}[[t_w]]$ is a regular local ring of dimension  2, which is complete for the topology defined by the powers of its maximal ideal $\fm.$

\begin{prop}\label{g3.3}  For each prime $w$ of $H$ above $\fp$, the sequence of unit power series $d_{\alpha,m}(t_w) \, (m=0, 1, \cdots)$ in $\CO_{H, w}[[t_w]]$ converges to a unique unit power series $d_{\alpha, \infty}(t_w)$ in
$\CO_{H, w}[[t_w]]$, satisfying $d_{\alpha, \infty}(t_w(V_n)) = u_{\alpha, n}$ for all $n \geq 0$.
\end{prop}
\begin{proof} For all $m_2 \geq m_1$, it follows from Lemma \ref{nc6} that the power series
$$
d_{\alpha, m_2}(t_w) - d_{\alpha, m_1}(t_w)
$$
vanishes at the points $t_w(V_k) \, (k=0, \cdots, m_1)$ and all conjugates of these points over $H_w$. Thus $d_{\alpha, m_2}(t_w) - d_{\alpha, m_1}(t_w)$ is divisible in $\CO_{H,w}[[t_w]]$ by the monic polynomial $P_{m_1}(t_w)$ whose roots are given by  all conjugates of the $t_w(V_k) (k = 0, \cdots, m_1)$ over $H_w$. Since it is easily seen that $P_{m_1}(t_w)$ tends to zero in the $\fm$-adic topology as $m_1 \to \infty$, it follows by completeness that the limit power series $d_{\alpha, \infty}(t_w) = \lim_{m \to \infty} d_{\alpha,m}(t_w)$ exists, and satisfies $d_{\alpha, \infty}(t_w(V_n)) = u_{\alpha, n}$ for all $n \geq 0$. This completes the proof.

\end{proof}

\section {Canonical measures attached to elliptic units.}

The aim of this section is to show how to relate norm compatible systems of elliptic units in the
tower $F_\infty /F$ to the the complex $L$-values $L(\bar{\phi}^k, k)$ for all integers  $k \geq 1$, again broadly following the method introduced in \cite{CW3}.  As always, $\delta$ always denotes the unique element of $\Gal(\fF_\infty/F_\infty)$ whose restriction to $\fG$ is the
Artin symbol of $\fp$. Thus, for any place $w$ of $H$ above $\fp$, we can view $\delta$ as a generator of $\Gal(H_w/K_\fp)$. If $J(P)$
is any rational function on $A/H$, we shall write $J^\delta(P_\fp)$ for the rational function on
$A^\fp/H$ obtained by applying $\delta$ to the coefficients of $J(P)$.

We first establish the following analogue of \eqref{nc8}.

\begin{lem} For all $n \geq 0$, we have
\begin{equation}\label{5.1}
D_{\alpha,n}^\delta(\eta_A(\fp)(P)) = \prod_{V \in A_\fp} D_{\alpha, n}(P \oplus V).
\end{equation}
\end{lem}
\begin{proof}  For each $\fc \in \fC_n$, we have the equality of isogenies
$$
\eta_{A^\fp}(\fc) \circ \eta_A(\fp) = \eta_{A^\fc}(\fp) \circ \eta_A(\fc) = \eta_A(\fc\fp),
$$
whence it follows easily that
\begin{equation}\label{5.2}
D_{\alpha,n}^\delta(\eta_A(\fp)(P)) = \prod_{\fc \in \fC_n}\fR_{\alpha, A^{\fc\fp}}(\eta_{A^\fc}(\fp)(\eta_A(\fc)(P)).
\end{equation}
Since $\fp$ is prime to $\alpha$, the good behaviour of the $R$-functions with respect to isogeny (see Theorem 4 of \cite{coates2}) shows that
$$
 \fR_{\alpha, A^{\fc\fp}}(\eta_{A^\fc}(\fp)(P_\fc)) = \prod_{W \in A_\fp^\fc}\fR_{\alpha, A^\fc}(P_\fc \oplus W).
 $$
 The assertion of the lemma then follows on noting that the isogeny $\eta_A(\fc)$ maps $A_\fp$ isomorphically to $A^\fc_\fp$
because of our assumption that $(\fc, \fp) = 1$ for $\fc \in \fC_n$. This completes the proof.
\end{proof}

As in the previous section, we write $d_{\alpha,n}(t_w)$ for the formal power series expansion of the rational function $D_{\alpha,n}(P)$ on $A/H$ in terms of the parameter $t_w$ of the formal group of $A$ at $w$, say $d_{\alpha,n}(t_w) = \sum_{k\geq0}e_{\alpha,n}(k)t_w^k.$  It  is then clear that the rational function $D_{\alpha,n}^\delta(P_\fp )$ on $A^\fp$ has the expansion in terms of the parameter $t_{\fp,w}$ of the formal group of $A^\fp$ at $w$ given by $d_{\alpha,n}^\delta(t_{\fp,w}) = \sum_{k\geq0}\delta(e_{\alpha,n}(k))t_{\fp,w}^k.$ Moreover, since $d_{\alpha, \infty}(t_w) =
\lim_{n \to \infty} d_{\alpha,n}(t_w)$, we see that $\lim_{n \to \infty} d_{\alpha,n}^\delta(t_{\fp,w})$
also exists, and we denote this limit by  $d_{\alpha, \infty}^\delta(t_{\fp,w})$. All of these formal power series are, of course, units. Then the key first step for constructing our measures is to define, for $0 \leq n \leq \infty$,  the power series
\begin{equation}\label{5.3}
\fD_{\alpha, n, w}(t_w) = d_{\alpha, n}(t_w)^2/d_{\alpha, n}^\delta(\widehat{\eta_{A,\fp,w}}(t_w)),
\end{equation}
where, as before,  $\widehat{\eta_{A,\fp,w}} : \widehat{A}_w \to \widehat{A^\fp_w}$ is the map
between formal groups induced by the isogeny $\eta_A(\fp)$.

\begin{lem}\label{5.4} For $0 \leq n \leq \infty$, the power series $\fD_{\alpha, n, w}(t_w)$ lies in $1+\fm_{H, w}[[t_w]]$, where $\fm_{H, w}$ denotes the maximal ideal of $\CO_{H,w}$.
\end{lem}
\begin{proof} Since $\phi$ is the Serre-Tate character of $B/K$, the reduction of the endomorphism $\phi(\fp)$ of $B$ must be the Frobenius endomorphism of the reduction of $B$ modulo $\fp$, from which it follows easily that
\begin{equation}\label{5.5}
\widehat{\eta_{A, \fp, w}}(t_w) \equiv t_w^2 \mod \, \fm_{H, w}.
\end{equation}
Hence
$$
d_{\alpha, n}^\delta(\widehat{\eta_{A,\fp,w}}(t_w)) = \sum_{k=0}^{\infty}\delta(e_{\alpha,n}(k))(\widehat{\eta_{A,\fp,w}}(t_w))^k \equiv \sum_{k=0}^{\infty}e_{\alpha,n}(k)^2t_w^{2k} \mod \, \fm_{H, w}.
$$
On the other hand, $d_{\alpha, n}(t_w)^2$ has $t_w$-expansion
$$
(\sum_{k=0}^{\infty}e_{\alpha, n}(k)t_w^k)^2 \equiv \sum_{k=0}^{\infty}e_{\alpha,n}(k)^2t_w^{2k} \mod \, \fm_{H, w},
$$
whence the power series \eqref{5.3} does indeed lie in $1+\fm_{H, w}[[t_w]]$, as claimed.
\end{proof}

\begin{lem}\label{5.6} For $0 \leq n \leq \infty$, the power series $J_{\alpha,n,w}(t_w) = \frac{1}{2}log(\fD_{\alpha, n, w}(t_w))$ lies in $\CO_{H, w}[[t_w]]$, and satisfies the identity
\begin{equation}\label{5.7}
\sum_{V \in A_\fp}J_{\alpha,n,w}(t_w[+]t_w(V)) = 0,
\end{equation}
where $[+]$ denotes the group law of the formal group of $\widehat{A_w}$.
\end{lem}

\begin{proof} The first assertion follows immediately from Lemma \ref{5.4} because $H_w$ is an unramified extension of $\BQ_2$. Moreover, for $0 \leq n \leq \infty$, Lemma \ref{5.1} shows that
\begin{equation}\label{5.8}
\prod_{V \in A_\fp}\fD_{\alpha, n}(t_w [+] t_w(V)) = 1.
\end{equation}
But then we conclude that this identity must also hold for $n = \infty$ by passage to the limit as $n \to \infty$. The second assertion of the lemma now follows on taking logarithms, and noting that the points in $A_\fp$ do indeed lie on the formal group $\widehat{A_w}$.
\end{proof}

We now fix an embedding of $H$ into the maximal unramified extension of $K_\fp$ which induces the prime $w$ of $H$. We recall that $\msi$ denotes the ring of integers of the completion of the maximal
unramified extension of $K_\fp$. Since the formal group $\widehat{A_w}$ has height 1, a classical theorem asserts that it is isomorphic over $\msi$ to the formal multiplicative group $\widehat{\BG}_m,$ and we fix such an isomorphism
\begin{equation}\label{5.9}
j_w: \widehat{\BG}_m \simeq \widehat{A_w}.
\end{equation}
Writing $W$ for the parameter of the formal
multiplicative group, the isomorphism $j_w$ can be viewed as being given by a formal power series $t_w = j_w(W)$ in $W$ with coefficients in $\msi$ of the form $j_w(W) = \Omega_w(A) W + \cdots$, where $\Omega_w(A)$ is a unit in $\msi$.  For $0 \leq n \leq \infty$, we can then define the power series
\begin{equation}\label{5.10}
\fJ_{\alpha,n,w}(W) = J_{\alpha,n,w}(j_w(W)),
\end{equation}
which, in view of \eqref{5.7}, satisfies
\begin{equation}\label{5.11}
\sum_{\zeta \in \mu_2}\fJ_{\alpha, n,w} (\zeta(1+W)-1) = 0,
\end{equation}
where $\mu_2 = \{\pm1\}.$ Recall that $\Lambda_{\msi}(\CO_\fp)$ (resp. $\Lambda_{\msi}(\CO_\fp^\times)$)
denotes the ring of $\msi$-valued measures on $\CO_\fp$ (resp. $\CO_\fp^\times$). Thanks to Mahler's beautiful theorem on the characterization of continuous $2$-adic valued functions on
$\CO_\fp = \BZ_2$, we have the topological ring isomorphism
\begin{equation}\label{5.12}
\BM : \Lambda_{\msi}(\CO_\fp) \simeq \msi[[W]],
\end{equation}
which is defined by $\BM(\mu) = \sum_{n \geq 0}a_n(\mu)W^n$, where $a_n(\mu) = \int_{\CO_\fp}\binom{x}{n}d\mu$ for all $n \geq 0.$ Now we have the inclusion $i: \Lambda_{\msi}(\CO_\fp^\times) \to
\Lambda_{\msi}(\CO_\fp)$ given by extending a measure on $\CO_\fp^\times$ to $\CO_\fp$ by zero.
Moreover, it is well known that a measure $\mu$ belongs to $i(\Lambda_{\msi}(\CO_\fp^\times))$ if and only if $\BM(\mu)$ satisfies the equation
\begin{equation}\label{5.13}
\sum_{\zeta \in \mu_2}\BM(\mu)(\zeta(1+W)-1) = 0.
\end{equation}
In particular, we conclude from \eqref{5.11} that, for all $n$ with $0 \leq n \leq \infty$ there exists a unique measure $\mu_{\alpha,n,w}$
in $\Lambda_{\msi}(\CO_\fp^\times)$ such that $\BM(i(\mu_{\alpha,n,w})) = \fJ_{\alpha,n,w}(W)$.
Recalling that we can  canonically identify the Galois group $G$ with $\CO_\fp^\times$ via the character $\rho_\fP$, we shall in what follows always view the $\mu_{\alpha,n,w}$ as $\msi$-valued measures on $G$. Moreover, we have
\begin{equation}\label{5.14}
\mu_{\alpha,\infty,w} = \lim_{n \to \infty} \mu_{\alpha,n,w}.
\end{equation}
 In the following, when there is no danger of confusion about the place $w$ of $H$ lying above $\fp$, we shall simply write $\mu_{\alpha, \infty}$ for $\mu_{\alpha, \infty, w}$.

\medskip

We recall that we have fixed an embedding of  the field $\mst$ into $\BC$ which extends our embedding of $K$ into $\BC$, so that we can then consider the complex Hecke $L$-functions
$L(\bar{\phi}^k, s)$ for all integers $k \geq 1$. Now the Hecke character $\bar{\phi}^k$ has conductor $\fq$ or $\CO_K$, according as $k$ is odd or even. For all $k \geq 1$, we shall write $L_\fq(\bar{\phi}^k, s)$ for the Euler product with the Euler factor at $\fq$ removed (so that this $L$-series is imprimitive when $k$ is even). Finally, we fix an embedding of the compositum $\mst H$
into the fraction field of $\msi$ which induces the prime $w$ of $H$ and the prime $\fP$ of $\mst$.
This is possible because $H\cap\mst = K$.

\begin{thm}\label{5.15} For all integers $k \geq 1$, the values $\Omega_\infty(A)^{-k}L_\fq(\bar{\phi}^k, k)$ belong to $\mst H$, and we have
\begin{equation}\label{5.16}
\Omega_\fp(A)^{-k}\int_{G} \rho_\fP^kd\mu_{\alpha, \infty} = b_k(\alpha)(k-1)!\Omega_\infty(A)^{-k}L_\fq(\bar{\phi}^k, k)(1-\phi^k(\fp)/N\fp),
\end{equation}
where $b_k(\alpha) = (-1)^{k-1}(\sqrt{-q})^k(\phi^k((\alpha)) - N\alpha)$.
\end{thm}

\noindent The crucial step in proving this theorem is the following classical result. If $\fb$ is any integral ideal of $K$ prime to $\fq$, we write $\gamma_\fb$ for the class of $\fb$. We then define the partial Hecke $L$-function $L_\fq(\bar{\phi}^k, \gamma_\fb, s)$ by
\begin{equation}\label{5.17}
L_\fq(\bar{\phi}^k, \gamma_\fb, s) = \sum_{ \fc \in \gamma_\fb} \frac{\bar{\phi}^k(\fc)}{(N\fc)^s},
\end{equation}
where the sum is taken over all integral ideals $\fc$ of $K$, which are prime to $\fq$, and which lie in the class $\gamma_\fb$.
\begin{prop}\label{5.18} Let $\fb$ be any integral ideal of $K$ prime to $\fq$. Then, for all $\alpha \in \msj$, we have
\begin{equation}\label{5.19}
\frac{d}{dz}\log (\fR_{\alpha, A^\fb}(\eta_{A}(\fb)(\CW(z, \CL))) = \sum_{k=1}^{\infty}b_k(\alpha)\phi(\fb)^k \Omega_\infty(A)^{-k}L_\fq(\bar{\phi}^k, \gamma_\fb, k)z^{k-1}.
\end{equation}
where $b_k(\alpha) = (-1)^{k-1}(\sqrt{-q})^k(\phi((\alpha))^k - N\alpha).$
\end{prop}
\begin{proof} The proof rests upon a miraculous product formula from the 19th century theory of elliptic functions. We recall rapidly
this product formula, and its link with our rational functions $\fR_{\alpha, A^\fb}$, without giving a fully detailed, but essentially straightforward,  proof of the Proposition.
Let $L$ be any lattice in $\BC$, say $L = u\BZ + v\BZ$, with $v/u$ having positive imaginary part.  We define
$$
A(L) = (\ov{u}v - \ov{v}u)/2\pi i, \, \, s_2(L) = \lim_{s \to 0} \sum_{\omega \neq 0 \in L}\omega^{-2}|\omega|^{-2s},
$$
where the limit is taken over real values of $s > 0$. For each integer $k \geq 1$, we have the Kronecker-Eisenstein series
$$
H_k(z, s, L) = \sum_{\omega \in L}\frac{(\ov{z} + \ov{\omega})^k}{|z + \omega|^{2s}},
$$
where we assume that $z \notin L$. This series converges in the half plane $R(s) > 1 + k/2$, but it has a holomorphic continuation to the whole $s$-plane. We then define
$$
\CE_k^*(z, L) = H_k(z, k, L).
$$
The all important classical infinite product which we use is
$$
\sigma(z, L) = z\prod_{\omega \in L\setminus 0}(1-z/\omega)\exp(z/\omega + \frac{1}{2}(z/\omega)^2),
$$
and we then define
$$
\theta(z, L) = \exp(-s_2(L)z^2/2)\sigma(z,L).
$$
Taking logarithmic derivatives, we easily obtain from the infinite product that, for all $z_0 \in \BC$ with $z_0 \notin L$, we have
\begin{equation}\label{m1}
\frac{d}{dz} \log(\theta(z+z_0, L)) = \ov{z_0}/A(L) + \sum_{k=1}^{\infty}(-1)^{(k-1)}\CE_k^*(z_0, L)z^{k-1}.
\end{equation}
We now take these functions for the lattice $\CL_\fb$, and we have the following equality (see \cite{GS}, Theorem 1.9):
\begin{equation}\label{m2}
R_{\alpha, A^{\fb}}(\CW(z, \CL_{\fb}))^2 = c_\alpha(A_\fb)^2 \theta(z, \CL_\fb)^{2N\alpha}/\theta(z, \alpha^{-1}\CL_\fb)^2.
\end{equation}
Now take $\fE$ to be any set of integral  of $K$, prime to $\fq$, whose Artin symbols in $\Gal(H(A_\fq)/K)$ give precisely $\Gal(H(A_\fq)/H)$. Recalling that $H(A_\fq) = H(A_\fq^\fb)$, and taking the $\fq$-division point on $A^\fb$ given by $Q(\fb) = \CW(\xi(\fb)\Omega_\infty(A)/\sqrt{-q}, \CL_{\fe})$, we then have (see \cite{GS}, Proposition 5.5)
\begin{equation}\label{m3}
(\sqrt{-q})^k(\phi(\fb)/\xi(\fb))^k\Omega_\infty(A)^{-k}L(\ov{\phi}^k, \gamma_\fb, k)  = \sum_{\fe \in \fE}\CE_k^*(\phi(\fe)\xi(\fb)\Omega_\infty(A)/\sqrt{-q}, \CL_\fb).
\end{equation}
On combining equations \eqref{m1}, \eqref{m2}, \eqref{m3}, recalling the definition \eqref{2.15}, and noting that
\begin{equation}\label{m4}
\eta_A(\fb)(\CW(z, \CL)) = \CW(\xi(\fb)z, \CL_\fb),
\end{equation}
the formula \eqref{5.19} follows easily. This completes the proof.
\end{proof}

\begin{cor} \label{5.20} For all integral ideals $\fb$ of $K$ prime to $\fq$ and all integers $k \geq 1$,
$$
b_k(\alpha)\phi(\fb)^k (k-1)!\Omega_\infty(A)^{-k}L_\fq(\bar{\phi}^k, \gamma_\fb, k)
$$
belongs to $H$ and is integral at $w$.
\end{cor}
\begin{proof} Since $\fR_{\alpha, A^\fb}(\eta_{A}(\fb)(P))$ is a rational function on $A/H$, the first
assertion is clear from \eqref{5.19}. Moreover, we have already seen (see Lemma \ref{u}) that this rational function has an expansion, in terms of the parameter $t_w$ of the formal group $\widehat{A_w}$, which is a unit in $\CO_{H,w}[[t_w]]$. Now we can interpret the differential operator $d/dz$ in terms of the formal group $\widehat{A_w}$ as follows. The exponential map of $\widehat{A_w}$ is given by expanding $t_w = \nu_w(z)$, where $\nu_w(z)$ is given by expanding the right hand side of
\begin{equation}\label{5.21}
t_w = -2(\wp(z, \CL) - b_2/12)/(\wp'(z, \CL) -a_1(\wp(z, \CL) -b_2/12) -a_3),
\end{equation}
as a power series in $z$. The logarithm map of $\widehat{A_w}$ is given by the inverse series under composition, say
\begin{equation}\label{5.22}
z = \lambda_w(t_w).
\end{equation}
We shall simply write $\lambda_w'(t_w)$ for the formal derivative of  the power series $\lambda_w(t_w)$ with respect to $t_w$. By one of the basic properties of such a logarithm map, $\lambda_w'(t_w)$ is in fact a unit power series in $\CO_{H,w}[[t_w]]$. But we have $d/dz = (\lambda'(t_w))^{-1}d/dt_w$. Hence, since the $t_w$-expansion of $\fR_{\alpha, A^\fb}(\eta_{A}(\fb)(P))$ is a unit in $\CO_{H,w}[[t_w]]$,
we conclude that, for all integers $k \geq 1$, the $t_w$-expansion of $(d/dz)^k \log \fR_{\alpha, A^\fb}(\eta_{A}(\fb)(P))$ will lie in $\CO_{H,w}[[t_w]]$.  In particular, its constant term will lie in $\CO_{H,w}$, proving the second assertion of the corollary.
\end{proof}

\begin{prop} For all integers $n \geq 0$ and $k \geq 1$, we have
\begin{equation}\label{5.23}
\Omega_\fp(A)^{-k}\int_{G} \rho_\fP^kd\mu_{\alpha, n} = b_k(\alpha)(k-1)!\Omega_\infty(A)^{-k} \sum_{\fc \in \fC_n}\phi^k(\fc) (L_\fq(\bar{\phi}^k, \gamma_\fc, k) - \frac{\phi(\fp)^k}{N\fp}L_\fq(\bar{\phi}^k, \gamma_{\fc\fp}, k)).
\end{equation}
\end{prop}
\begin{proof} The exponential map of the formal group $\widehat{\BG_m}$ is given by $W = e^z -1$.
Thus, by the uniqueness of the exponential map for $\widehat{A_w}$, it follows that we must have
\begin{equation}\label{5.24}
t_w = \nu_w(z) = j_w(e^{z/{\Omega_v(A)}} - 1).
\end{equation}
Now it is very well known (see, for example, \cite{CS}) that, for all integers $k \geq 1$, we have
\begin{equation}\label{5.25}
\int_{G}\rho_\fP^kd\mu_{\alpha, n}  = \int_{\BZ_2}x^k d\mu_{\alpha, n} = (d/dz)^k\fJ_{\alpha, n}(W)|_{z=0}.
\end{equation}
But obviously
$$
(d/dz)^k\fJ_{\alpha, n}(e^z - 1) =  \Omega_v(A)^k(d/dz)^k\fJ_{\alpha, n}(e^{z/{\Omega_v(A)}} - 1).
$$
In view of \eqref{5.24}, we conclude from \eqref{5.10} that
\begin{equation}\label{5.26}
\Omega_v(A)^{-k}\int_{G}\rho_\fP^kd\mu_{\alpha, n} = (d/dz)^k\fM_{\alpha, n}(t_w)|_{z=0}.
\end{equation}
But
$$
\fM_{\alpha, n}(t_w) = \frac{1}{2} \sum_{\fc \in \fC_n} \log (\fR_{\alpha, A^\fc}(\eta_{A}(\fc)(\CW(z, \CL))^2/\fR_{\alpha, A^{\fc\fp}}(\eta_{A}(\fc\fp)(\CW(z, \CL))).
$$
Hence the conclusion of the proposition follows immediately from \eqref{5.26} and applying Proposition \ref{5.18} with $\fb = \fc$ and $\fb = \fc\fp$. This completes the proof.
\end{proof}

\medskip

\noindent
 We can now prove Theorem \ref{5.15}. In view of \eqref{5.14}, we
have, for all integers $k\geq 1$,
$$
\int_{G} \rho_\fP^k d\mu_{\alpha, \infty} = \lim_{n \to \infty}\int_{G} \rho_\fP^k d\mu_{\alpha,n}.
$$
We recall that we have fixed an embedding of the compositum $H\mst$ into the fraction field of $\msi$ which induces the prime $w$ on $H$ and the prime $\fP$ on $\mst$. Recall also that for $\fc \in \fC_n$, we have $\phi(\fc) \equiv 1 \mod \fP^{n+2}$. Moreover, for each $n \geq 0$, $\fC_n$ gives a complete set of representatives of the ideal class group of $K$.
It is therefore clear that Theorem \ref{5.15} will follow from on passing to the limit as $n \to \infty$ from  \eqref{5.23}. This completes the proof.
\medskip

\section{Iwasawa theory of the local tower $K_\fp(B_{\fP^\infty})/K_\fp$}

We need to establish the local theory at the prime $\fp$ for the tower $F_\infty/K$ in order to prove the main conjecture. The local extension $K_\fp(B_{\fP^\infty})/K_\fp$ idoes not arise naturally from the points of finite order on a Lubin-Tate group. However, we take a prime $w$ of $H$ above $\fp$, and we show that, since the class number of $K$ is odd, we can fairly easily derive what is needed from the classical theory for the Lubin-Tate extension $H_w(A_{\fp^\infty})/H_w$.

\medskip

Put $r = [H_w : K_\fp]$, so that $r$ is the order of both $\delta$, and the ideal class of $\fp$, and, of course, $r$ is odd because it divides the class number of $K$. Define
\begin{equation}\label{6.1n}
H_{w, \infty} = H_w(A_{\fp^\infty}), \, \,  F_{\infty, \fp} = K_\fp(B_{\fP^\infty}),
\end{equation}
and put
\begin{equation}
\CG_w = \Gal(H_{w, \infty}/H_w ), \, \, \fG_{w, \infty} = \Gal(H_{w, \infty}/K_\fp), \, \, \Delta_w = \Gal(H_{w, \infty}/F_{\infty, \fp}),
\end{equation}
so that $\fG_{w, \infty} = \CG_w \times \Delta_w$. Of course, $H_w/K_\fp$ is unramified, $\fp$ is totally ramified in $F_{\infty, \fp}$, and $w$ is totally ramified in $H_{w, \infty}$. Let $F_{n, \fp} =K_\fp(B_{\fP^{n+2}})$ and $H_{n, w} = H_w(A_{\fp^{n+2}})$, and write $U(F_{n, \fp})$ and $U(H_{n, w})$ for their respective groups of local units.  Define
\begin{equation}\label{6.2n}
U(F_{\infty,\fp}) = \varprojlim_n U(F_{n, \fp}), \, \, U(H_{\infty,w}) = \varprojlim_n U(H_{n, w}),
\end{equation}
where the projective limits are taken with respect to the local norm maps. We recall that we have fixed a set $\{V_n\}$ of primitive $\fp^{n+2}$-division points on $A$, which are compatible in the sense that \eqref{y1.9} holds for all $n \geq 0$. We write, as before, $t_w$ for the parameter of the formal group $\widehat{A_w}$ of $A$ at $w$, and $\CO_{H, w}$ for the ring of integers of $H_w$.
The following is de Shalit's extension of Coleman's theorem \cite{deshalit}.
\begin{thm}\label{6.2n}
For each $\fu_\infty = (\fu_n)$ in $U(H_{\infty,w})$, there exists a unique power series $C_{\fu_\infty}(t_w)$ in $\CO_{H,w}[[t_w]]$ such that $\fu_n = \delta^{-n}(C_{\fu_\infty}(t_w(V_n)))$
for all $n \geq 0$.
\end{thm}
\noindent We shall also need the following corollary of this theorem.
\begin{cor}  For each $u_\infty = (u_n)$ in $U(F_{\infty,\fp})$, there exists a unique power series $C_{u_\infty}(t_w)$ in $\CO_{H,w}[[t_w]]$ such that $u_n = C_{u_\infty}(t_w(V_n))$
for all $n \geq 0$.
\end{cor}
\noindent To deduce the corollary from the theorem, we note that, for all $m \geq n$, $\Gal(\fF_{m,w}/\fF_{n,w})$ is isomorphic to $\Gal(F_{m, \fp}/F_{n,\fp})$ under restriction. Hence $u_\infty$ clearly can be viewed as an element of $U(H_{\infty,w})$, and so has its Coleman power series $C_{u_\infty}(t_w)$. This Coleman power series then has the interpolation property stated in the corollary because $\delta^n(u_n) = u_n$ for all $n \geq 0$.  The uniqueness in both results is obvious from the Weierstrass Preparation Theorem.

Let $\fu_\infty = (\fu_n)$ be any element of $U(H_{\infty,w})$,  with Coleman power series $C_{\fu_\infty}(t_w)$.  Define
\begin{equation}\label{6.3n}
\fC_{\fu_\infty}(t_w) = C_{\fu_\infty}(t_w)^2/C_{\fu_\infty}^\delta(\widehat{\eta_{A, \fp, w}}(t_w)).
\end{equation}

\begin{lem}\label{6.4n}The power series $\fA_{\fu_\infty}(t_w) = \frac{1}{2}\log \fC_{\fu_\infty}(t_w)$ lies in $\CO_{H, w}[[t_w]]$, and satisfies the identity
\begin{equation}\label{6.5n}
\sum_{V \in A_\fp}\fA_{\fu_\infty}(t_w[+]t_w(V)) = 0,
\end{equation}
where $[+]$ denotes the group law of the formal group of $\widehat{A_w}$.
\end{lem}
\begin{proof} An entirely similar argument to that given in the proof of Lemma \ref{5.4} shows that
$\fC_{\fu_\infty}(t_w)$ belongs to $1+\fm_{H, w}[[t_w]]$, and so the first assertion of the
lemma follows. Moreover, it is shown in \cite{deshalit} (see Proposition 2.1 on p. 12 and Corollary 2.3, (ii) on p. 14) that
$$
C_{\fu_\infty}^\delta(\widehat{\eta_{A, \fp, w}}(t_w)) = \prod_{V \in A_\fp}C_{\fu_\infty}(t_w[+]t_w(V)),
$$
whence the second assertion of the lemma also follows easily on taking logarithms.
\end{proof}

Recall  \eqref{5.9} that we have fixed an $\msi$-isomorphism $j_w: \widehat{\BG}_m \simeq \widehat{A_w}$, and so, writing $W$ for the parameter of the formal multiplicative group, we can then define the power series
\begin{equation}\label{6.6n}
\fB_{\fu_\infty}(W) = \fA_{\fu_\infty}(j_w(W)).
\end{equation}
In view of \eqref{6.5n}, it follows, on applying Mahler's theorem to the power series $\fB_{\fu_\infty}(W)$, that there exists an $\msi$-valued measure $\mu(\fu_\infty)$
on  $\CO_\fp^\times$ such that $\BM(i(\mu(\fu_\infty))) = \fB_{\fu_\infty}(W)$. However, we now view $\mu(\fu_\infty)$ as a measure on the Galois group $\CG_w$ via the canonical isomorphism $\rho_\fp: \CG_w \ra \CO^\times_\fp$. Recall that $\fG_{w, \infty} = \CG_w \times \Delta_w$.  Following \cite{deshalit}, we then define the measure $\tilde{\mu}(\fu_\infty)$ in $\Lambda_\msi(\fG_{w, \infty})$ by
\begin{equation}\label{6.7n}
\tilde{\mu}(\fu_\infty) = \sum_{\sigma \in \Delta_w}\sigma\mu(\sigma^{-1}(\fu_\infty)).
\end{equation}
This enables us to define the $\msi$-linear $\fG_{w, \infty}$-homomorphism
\begin{equation}\label{6.8n}
   j_{H_{w,\infty}} : U(H_{w,\infty})\widehat{\otimes}_{\CO_\fp}\msi  \to \Lambda_\msi(\fG_{w, \infty})
\end{equation}
by  $j_{H_{w,\infty}}(\fu_\infty \otimes 1) = \tilde{\mu}(\fu_\infty)$. De Shalit (see Theorem 3.6, Chap. 1, of \cite{deshalit}) then goes on to prove that $j_{H_{w,\infty}}$ is injective, and has a cokernel
of the form $\fW_w\otimes_{\CO_\fp}\msi$, where $\fW_w$ is a finite $\fG_{w, \infty}$-module. On the other hand, if we take any $u_\infty \in U(F_{\infty, \fp})$, and define the power series $\fB_{u_\infty}(W)$ attached to $u_\infty$ by exactly the same procedure as above, we again obtain by Mahler's theorem a measure $\mu(u_\infty)$ on $\CO_\fp^\times$. However, in this case, we view $\mu(u_\infty)$ as a measure on the Galois group $G$ via the canonical isomorphism
$\rho_{\fP}: G \ra \CO^\times_\fp$. This in turn enables us to define the canonical $\msi$-linear $G$-homomorphism
\begin{equation}\label{6.9n}
    j_{F_\infty} : U(F_{\infty, \fp})\widehat{\otimes}_{\CO_\fp}\msi  \to \Lambda_\msi(G)
\end{equation}
by $j_{F_\infty}(u_\infty \otimes 1) = \mu(u_\infty)$.

\begin{prop}\label{6.10n} The  map $j_{F_\infty}$ is an injective $\Lambda_\msi(G)$-homomorphism, and its cokernel is of the form $\fW \otimes_{\CO_\fp} \msi$, where $\fW$ is a finite $G$-module.
\end{prop}

\begin{proof}  We first prove that  $j_{F_\infty}$ is injective. If $\mu(u_\infty) = 0$, then we must $\fC_{u_\infty}(t_w) = 1$, and so
\begin{equation}\label{6.11n}
C_{u_\infty}(t_w)^2 = C_{u_\infty}^\delta(\widehat{\eta_{A, \fp, w}}(t_w)).
\end{equation}
Putting $t_w = t_w(V_n)$ in this equation, and recalling the compatibility relation \eqref{y1.9},
we conclude that $u_n^2 = u_{n-1}$ for all $n \geq 1$. Putting $t_w = t_w(V_0)$ in the equation, we also conclude that $u_0^2 = C_{u_\infty}(0)^\delta$. Finally, putting  $t_w = 0$ in \eqref{6.8n}, we obtain $C_{u_\infty}(0)^2 = C_{u_\infty}(0)^\delta$, and so $C_{u_\infty}(0)^{2^r-1} = 1$, where $r$ is the order of $\fp$ in the ideal class group of $K$. Thus $C_{u_\infty}(0) = 1$ because $F_\fp$ has residue class field $\BF_2$. But the group $\mu_{2^\infty}$ of all 2-power roots of unity, cannot belong to the completion of $F_\infty$ at the unique prime above $\fp$. Indeed, if it did, then the field  $H_w(A_{\fp^\infty})$ would have to also contain $A_{\fp*^\infty}$ by the Weil pairing, and this is impossible because this latter group of points would map injectively under reduction modulo $w$,
which cannot be the case because the field $H_w(A_{\fp^\infty})$  has a finite residue field. Thus we must have $u_n = 1$ for all $n \geq 0$, and the proof of injectivity is complete.

\medskip

We next show how the remaining assertion of the theorem can be derived from Theorem 3.6, Chap. 1, of \cite{deshalit}. The local norm map $N_{H_{\infty,w}/F_{\infty,\fp}}$ extends by $\msi$-linearity to a map $\theta_\infty: U(H_{\infty,w})\wh{\otimes}\msi \to U(F_{\infty, \fp})\wh{\otimes}_{\CO_\fp}\msi$. We note that $\theta_\infty$ is surjective, because
if $u_\infty = (u_n)$ is any element of $U(F_\infty)$, then, noting that raising to the $r$-th power is an automorphism of $U(F_{\infty, \fp})$ because $r$ is odd, we can simply define the element $\fu_\infty = (u_n^{1/r})$ of $U(H_{\infty,w})$, whence clearly $N_{H_{\infty,w}/F_{\infty,\fp}}(\fu_\infty) = u_\infty$. Noting that every element $\zeta$ of  $\Lambda_\msi(\fG_{w, \infty})$ can be written uniquely in the form $\zeta = \sum_{\sigma \in \Delta_w}\sigma^{-1}A(\sigma)$ with $A(\sigma)$ in $\Lambda_\msi(\CG_w)$, we can define the map $\lambda_\infty : \Lambda_\msi(\fG_{w, \infty}) \to
\Lambda_\msi(G)$ by
\begin{equation}\label{6.12n}
\lambda_\infty (\zeta) = \sum_{\sigma \in \Delta_w}\widetilde{A(\sigma)},
\end{equation}
where $\widetilde{A(\sigma)}$ is the image of $A(\sigma)$ under the isomorphism from $\Lambda_\msi(\CG_w)$ to $\Lambda_\msi(G)$ given by the restriction map. We then have the following commutative diagram with exact rows
\begin{equation}\label{6.11n}
  \xymatrix{
  0  \ar[r]^{} & U(H_{\infty,w})\wh{\otimes}_{\CO_\fp}{\msi}\ar[d]_{\theta_\infty} \ar[r]^{ j_{H_{w,\infty}}} & \Lambda_\msi(\fG_{w, \infty}) \ar[d]_{\lambda_\infty} \ar[r]^{} & \fW_w\otimes_{\CO_\fp}\msi \ar[d]_{i_\infty} \ar[r]^{} & 0  \\
  0 \ar[r]^{} & U(F_{\infty, \fp})\wh{\otimes}_{\CO_\fp}\msi \ar[r]^{j_{F_\infty}} &  \ar[r]^{}\Lambda_\msi(G) & Coker(j_{F_\infty}) \ar[r]^{} & 0   }:
\end{equation}
The commutativity of the left hand square is easily verified from the explicit description we have given of the maps $\theta_\infty$ and $\lambda_\infty$.  Since the middle vertical map
is clearly surjective because $r$ is odd,  it follows that the right hand vertical map  is also surjective. Thus  the above diagram shows that Theorem \ref{6.10n} does indeed follow from Theorem 3.6, Chap. 1, of \cite{deshalit}.
\end{proof}

\section{Proof of the main conjecture}

We establish in this section the analogue for $B/F_\infty$ of Iwasawa's \cite{IW} celebrated theorem on cyclotomic fields, which led to the discovery of the main conjectures in general. Recall that $G = \Gamma \times \Delta$, where
\begin{equation}\label{7.1}
\Gamma = \Gal(F_\infty/F), \, \, \Delta= \Gal(F_\infty/K_\infty) = \{1, j \}.
\end{equation}
For each $n \geq 0$, let $\bar{C}(F_n)$ denote the closure of $C(F_n)$ in $U(F_{n, \fp})$ in the $\fp$-adic topology, and define $\bar{C}(F_\infty) = \varprojlim_n\bar{C}(F_n)$, where the projective limit is taken with respect to the local norm maps. Recalling that  $U(F_{\infty,\fp}) = \varprojlim_n U(F_{n, \fp})$, we then define the $G$-module
\begin{equation}\label{7.2}
Z(F_\infty) = U(F_{\infty, \fp})/\bar{C}(F_\infty).
\end{equation}
We also define
\begin{equation}\label{7.3}
U'(F_\fp)  = \cap_{n \geq 0} N_n(U(F_{n, \fp})).
\end{equation}
where $N_n$ denotes the local norm map from $F_{n, \fp}$ to $F_\fp$. Thanks to \eqref{nc4},
we see that we have $\bar{C}(F) \subset U'(F_\fp)$.

\begin{prop}\label{7.4} We have $(Z(F_\infty))_\Gamma = U'(F_\fp)/\bar{C}(F).$
\end{prop}

We first establish several preliminary results needed for the proof of this Proposition.  Write $N_{F/K}$ for the global norm map from $F$ to $K$, or the local norm map from
$F_\fp$ to $K_\fp$.
\medskip
\begin{lem}\label{7.5}  An element $u$ of $U(F_\fp)$ lies in $U'(F_\fp)$ if and only if $N_{F/K}(u)=1$. \end{lem}
\begin{proof} We first note that an element $z$ of $U(K_\fp)$ is a norm from $F_{n, \fp}$ for all
$n \geq 0$ if and  only if $z = 1$. Indeed, since $\fp$ is totally ramified in $F_n$ and $\Gal(F_{n, \fp}/K_\fp) = (\CO_\fp/\fp^{n+2})^\times$, it follows easily from local class field theory that
$N_{F_n/K_\fp}(U(F_{n, \fp})) = 1 + \fp^{n+2}\CO_{\fp}$, whence the previous assertion is clear.
Now assume that $u$ is any element of $U(F_\fp)$. If $u$ lies in $U'(F_\fp)$, then $z =N_{F/K}(u)$
is an element of $U(K_\fp)$ which is a norm from $U(F_{n, \fp})$ for all $n \geq 0$, and so
$z$ must be 1 by our first remark.  Conversely, suppose $u$ is any element of $U(F_{\fp})$
with $N_{F/K}(u) = 1$. Now the restriction map from $\Gal(F_{n, \fp}/F_\fp)$ to $\Gal(K_{n, \fp}/K_\fp)$
is an isomorphism for all $n \geq 0$. Moreover, the local Artin symbol of $u$ for the extension
$F_{n, \fp}/F_\fp$, say $\sigma$ restricts to the local Artin symbol of $N_{F/K}(u)$ for the extension $K_{n, \fp}/K_\fp$, which is equal to the identity. Hence $\sigma = 1$, and so by local class field theory, $u$ must be a norm from $U(F_{n, \fp})$, and the proof is complete.
\end{proof}
\begin{lem}\label{7.5} The natural projection of $U(F_{\infty,\fp})$ onto  $U'(F_\fp)$ induces an isomorphism $(U(F_{\infty,\fp}))_\Gamma = U'(F_\fp)$.
\end{lem}
\begin{proof} This is essentially an exercise in classical Iwasawa theory, whose proof we briefly sketch. Let $Y(F_{\infty, \fp})$ be the Galois group over $F_{\infty, \fp}$ of the maximal abelian 2-extension of $F_{\infty, \fp}$, and let $Y(F_\fp)$ be the Galois group over $F_{\infty, \fp}$ of the maximal abelian 2-extension of $F_\fp$. As usual,  $\Gamma = \Gal(F_{\infty, \fp}/F_\fp)$ acts on
$Y(F_{\infty, \fp})$, and we have
\begin{equation}\label{7.6}
(Y(F_{\infty, \fp}))_\Gamma = Y(F_\fp).
\end{equation}
On the other hand, Lubin-Tate theory shows that
\begin{equation}\label{7.7}
 Y(F_{\infty, \fp}) = U(F_{\infty, \fp}) \times \BZ_2,
\end{equation}
where $\BZ_2$ denotes the Galois group of the maximal unramified 2-extension of $F_{\infty, \fp}$. Since clearly $U(F_{\infty, \fp})^\Gamma = 0$, we conclude that $ Y(F_{\infty, \fp})^\Gamma$ is equal to the Galois group of the maximal unramified 2-extension of $F_{\infty, \fp}$.
Viewing \eqref{7.7} as a short exact sequence for the $\Gamma$-module $Y(F_{\infty, \fp})$ with $U(F_{\infty, \fp})$ as submodule and $\BZ_2$ as quotient, and taking  $\Gamma$-cohomology of this exact sequence, we obtain from the snake lemma the exact sequence
\begin{equation}\label{7.8}
0 \to U(F_{\infty, \fp})_\Gamma \to Y(F_{\infty, \fp})_\Gamma \to \BZ_2 \to 0.
\end{equation}
On the other hand, we have the short exact sequence of $\BZ_2$-modules
\begin{equation}\label{7.9}
0 \to U'(F_\fp) \to Y(F_\fp) \to \BZ_2 \to 0.
\end{equation}
There is now an obvious commutative diagram with exact rows mapping the sequence
\eqref{7.8} to the sequence \eqref{7.9} in which the left hand vertical map is obviously surjective,
and the middle map is an isomorphism by \eqref{7.6}. Hence  the left vertical map must be an isomoprhism, and the proof is complete.

\end{proof}
We can now prove Proposition \ref{7.4}. We have the obvious commutative diagram with exact rows
\[\xymatrix{
  &(\bar{C}(F_\infty))_\Gamma \ar[d]_{} \ar[r]^{} & (U(F_{\infty, \fp}))_\Gamma \ar[d]_{} \ar[r]^{} & (Z(F_\infty))_\Gamma \ar[d]_{} \ar[r]^{} & 0 \\
  0 \ar[r]^{} & \bar{C}(F) \ar[r]^{} & U'(F_\fp) \ar[r]^{} & U'(F_\fp)/\bar{C}(F) \ar[r]^{} & 0.}\]
Now the left vertical map is surjective by \eqref{nc4}, and the middle vertical map is an isomorphism by Lemma \ref{7.5}, whence
the right vertical map must be an isomorphism. This completes the proof of Proposition \ref{7.4}.

\medskip

As before, let $X(F_\infty) = \Gal(M(F_\infty)/F_\infty)$, where $M(F_\infty)$ is the maximal abelian 2-extension of $F_\infty$, which is unramified outside $\fp$.
 For each $n \geq 0$, let $\CE(F_n)$ be the group of all global units of $F_n$, and let $\bar{\CE}(F_n)$ be their closure in $U(F_{n, \fp})$ in the $\fp$-adic topology.
Define $\bar{\CE}(F_\infty) = \varprojlim_n\bar{\CE}(F_n)$, where the projective limit is taken with respect to the norm maps. We have already shown in Theorem \ref{odd} that $F_\infty$ has no unramified abelian 2-extension. Hence global class field theory provides the following explicit description
of $X(F_\infty)$,
\begin{equation}
X(F_\infty) = U(F_{\infty, \fp})/\bar{\CE}(F_\infty).
\end{equation}
Finally, define $\bar{\CE}'(F) = U'(F_\fp) \cap \bar{\CE}(F).$

\begin{prop}\label{7.10} We have $X(F_\infty)^\Gamma = 0$, and $X(F_\infty)_\Gamma = U'(F_\fp)/\bar{\CE}'(F)$.
\end{prop}

\begin{proof} Let $M(F)$ be the maximal abelian 2-extension of $F$, which is unramified outside $\fp$. Then
$$
X(F_\infty)_\Gamma = \Gal(M(F)/F_\infty) =  U'(F_\fp)/ \bar{\CE}'(F).
$$
where the first equality is elementary Iwasawa theory, and the second equality is by global class field theory.
Obviously the group on the right is finite because $\CE'(F)$ has rank 1 as an abelian group. Hence, since
$X(F_\infty)$ is a free finitely generated $\BZ_2$-module by Theorem \ref{t4}, we conclude that  we must also have
$X(F_\infty)^\Gamma = 0$.
\end{proof}

\medskip

Define  $R(F_\infty) = \bar{\CE}(F_\infty)/\bar{C}(F_\infty)$,  so that we have the exact sequence of $G$-modules
\begin{equation}\label{7.11}
0 \to R(F_\infty) \to Z(F_\infty) \to X(F_\infty) \to 0.
\end{equation}

\begin{cor}\label{7.12} We have $(R(F_\infty))_\Gamma = \bar{\CE}'(F)/\bar{C}(F)$.
\end{cor}
\begin{proof} Since $(X(F_\infty))^\Gamma = 0$, we have the exact sequence
\begin{equation}\label{7.13}
0 \to (R(F_\infty))_\Gamma \to (Z(F_\infty))_\Gamma \to (X(F_\infty))_\Gamma \to 0.
\end{equation}
The corollary then follows immediately from Propositions \ref{7.4} and \ref{7.10}.
\end{proof}

\medskip

Recall that $\Delta = \{1, j \}$ is the Galois group of $F_\infty$ over $K_\infty$. In the group ring $\BZ_2[\Delta]$, we define the two elements
\begin{equation}\label{7.14}
\epsilon_+ = 1+ j, \, \, \epsilon_- = 1 - j.
\end{equation}
\begin{lem}\label{7.15} Let $Y = W/V$, with $W$ a $\BZ_2[\Delta]$-module and $V$ a submodule. Let $\epsilon$ denote either of the above elements \eqref{7.14}.  Then  $\epsilon Y = \epsilon W/\epsilon V.$ \end{lem}

\begin{proof} We have the commutative diagram with exact rows
\begin{equation}
  \xymatrix{
  0  \ar[r]^{} & V\ar[d] \ar[r]^{ } & W \ar[d]_{}\ar[r]^{} & Y \ar[d]_{} \ar[r]^{} & 0  \\
  0 \ar[r]^{} & \epsilon V \ar[r]^{} &  \epsilon W\ar[r]^{} & \epsilon W/\epsilon V \ar[r]^{} & 0   },
\end{equation}
where the three vertical maps are multiplication by $\epsilon$. Since the middle vertical map is
obviously surjective, it follows from the snake lemme that the right vertical map is also surjective,
proving the lemma.
\end{proof}
\begin{thm}\label{7.16} For all primes $q$ with $q \equiv 7 \mod 8$, we have $\epsilon_+Z(F_\infty) = 0$.
\end{thm}
\begin{proof} It suffices to show that
\begin{equation}\label{7.17}
(\epsilon_+Z(F_\infty))_\Gamma = 0.
\end{equation}
But, since $G$ is commutative, it follows from the previous lemma, Proposition \ref{7.4} and Lemma \ref{7.5} that
$$
(\epsilon_+Z(F_\infty))_\Gamma  = \epsilon_+(Z(F_\infty))_\Gamma = \epsilon_+U'(F_\fp)/\epsilon_+ \bar{C}(F) = 0.
$$
This completes the proof.
\end{proof}

\begin{thm}\label{7.18} For all primes $q$ with $q \equiv 7 \mod 8$, we have $\epsilon_-Z(F_\infty) = \epsilon _-X(F_\infty)$.
\end{thm}
\begin{proof} In view of the exact sequence \eqref{7.11} and Lemma \ref{7.15}, it suffices to show that $\epsilon_- R(F_\infty) = 0.$
Hence it suffices to show that $(\epsilon_- R(F_\infty) )_\Gamma = 0$. But, again noting that $G$ is commutative and using Lemma \ref{7.15}
and Corollary \ref{7.12}, we have
$$
(\epsilon_- R(F_\infty) )_\Gamma = \epsilon_-(R(F_\infty))_\Gamma = \epsilon_-\bar{\CE}'(F)/\epsilon_-\bar{C}(F) = \bar{\CE}'(F)^2/\bar{C}(F)^2;
$$
the last equality is valid because $j(u) = u^{-1}$ for all $u$ in $\bar{\CE}'(F)$. Note also that the group of  roots of unit of $F$ is just $\mu_2 = \{\pm 1\}$,
because $\fp^*$ is not ramified in $F$. Hence the conclusion of Theorem \ref{7.18} will be an immediate consequence of the following result.
\end{proof}

\begin{prop} \label{7.19} The index of $C(F)$ modulo torsion in $\CE'(F)$ modulo torsion is an odd integer.
\end{prop}
The proof is based on the analytic class number formula and Kronecker's limit formula. By Dirichlet's theorem, the unit group
of $F$ has has rank 1, and, as above, its torsion subgroup is just $\mu_2 = \{\pm 1\}$. Write $\eta$ for a generator of the unit group of $F$ modulo torsion.  We fix an embedding of $F$ into $\BC$.
If $z$ is any element of $F$, $|z|_\BC$ will denote the square of the ordinary complex absolute value of $z$. Let $h_F$ denote the class number of $F$, and, as before, $h$ will denote the class number of $K$. Let $\omega$ denote the non-trivial character of $\Delta$, and write $L(\omega, s)$ for its complex $L$-series. Then Dirichlet's formula for the residue of the complex zeta functions immediately gives
\begin{equation}\label{7.20}
(2\pi)^2h_F \log|\eta|_\BC/ |D_F|^{1/2} = 2\pi L(\omega, 1)h/q^{1/2}.
\end{equation}
where $D_F$ denotes the discriminant of $F$. Now, by Lemma \ref{g2.6}, the conductor $\ff$ of $F/K$ is equal to $\fq\fp^2$. and so the conductor discriminant formula tells us that
$D_F = q^2N\ff$, where $N\ff$ denotes the absolute norm of $\ff$. Moreover, the functional equation for $L(\omega, s)$ gives
that
\begin{equation}\label{7.21}
L(\omega, 1) =  2\pi W(\omega)L'(\omega, 0)/(qN\ff)^{1/2},
\end{equation}
where $W(\omega) = \pm 1$ is the sign in the functional equation for $L(\omega, s)$ Combining \eqref{7.20} and \eqref{7.21}, we obtain finally
\begin{equation}\label{7.22}
L'(\omega, 0) = W(\omega)\log|\eta |_\BC h_F/h.
\end{equation}
\noindent We recall that, for each $\alpha \in \msj$, the elliptic unit  $u_{\alpha, 0}$ is defined by $u_{\alpha, 0}  = N_{\fF_0/F}(\fR_{\alpha, A}(V_0))$. Moreover, by Lemma \ref{7.5},
we have $N_{F/K}u_{\alpha, 0} = 1$, so that $j(u_{\alpha, 0}) = u_{\alpha, 0}^{-1}.$ Write $\gamma_\alpha$ for the Artin symbol of the ideal $\alpha \CO_K$ in $G$.
\begin{lem}\label{7.23} For each $\alpha \in \msj$, we have $2\log|u_{\alpha, 0}|_\BC = (\omega(\gamma_\alpha) - N\alpha)L'(\omega, 0)$. \end{lem}
\begin{proof}The proof, which we only sketch, rests crucially on Kronecker's second limit formula. For each integral ideal $\fg$ of $K$, we write $K(\fg)$ for the ray class group of $K$ modulo $\fg$. By Lemma \ref{g2.6}, we know that $K(\fg) = H(A_\fg)$ whenever the ideal $\fg$ is divisible by the conductor $\fq$ of $\phi$. Now the conductor $\ff$ of $\omega$ is equal to $\fq\fp^2$, and so $K(\ff)$ is the compositum of its two subfields $\fF_0 = H(A_{\fp^2})$ and $H(A_\fq)$. Moreover, $\fF_0 \cap H(A_\fq) = H$ because the primes of $H$
above $\fp$ are totally ramified in $\fF_0$ and unramified in $H(A_\fq)$.  Thus it is then clear from \eqref{2.13} that $\fR_{\alpha, A}(V_0) = N_{K(\ff)/\fF_0}R_{\alpha, A}(V_0 \oplus Q)$, where $Q$ is the primitive $\fq$-division point on $A$ defined by \eqref{2.12}. Hence we obtain
\begin{equation}\label{klf1}
u_{\alpha, 0} = N_{K(\ff)/F}R_{\alpha, A}(V_0 \oplus Q).
\end{equation}
On the other hand, writing $z_Q = \Omega_\infty(A)/\sqrt{-q}$ and $V_0 = \CW(z_0, \CL)$, we have the fundamental identity (see \eqref{m2})
$$
R_{\alpha, A}(V \oplus Q)^{12} = (c_\alpha(A)\theta(z_0 + z_Q, \CL)^{N\alpha}/ \theta(z_0 + z_Q, \alpha^{-1}\CL))^{12}.
$$
In view of this equality, the classical Kronecker's formula asserts that, for our character $\omega$ of $\Gal(K_\ff/K)$, we have
\begin{equation}\label{klf}
(\omega(\gamma_\alpha) - N\alpha)L'(\omega, 0) = \sum_{\sigma \in \Gal(K(\ff)/K)} \omega(\sigma)\log |\sigma(R_{\alpha, A}(V_0 \oplus Q))|_\BC.
\end{equation}
Recalling \eqref{klf1}, we see that the right hand side of this last formula is simply equal to $\log|u_{\alpha, 0}|_\BC - \log|\tau(u_{\alpha, 0})|_\BC$, where we have written $\tau$ for the non-trivial element of $\Gal(F/K)$. But  Lemma \ref{7.5} shows that $u_{\alpha, 0}.\tau(u_{\alpha, 0}) = 1$
(because $u_{\alpha, 0}$ is a norm from every finite layer of the $\BZ_2$-extension $F_\infty/F$).  Hence $\log|u_{\alpha, 0}|_\BC - \log|\tau(u_{\alpha, 0})|_\BC = 2\log|u_{\alpha, 0}|_\BC$, and so the assertion of the lemma follows from \eqref{klf}. This completes the proof.
\end{proof}
For the next lemma, we recall that we have chosen the sign of $\sqrt{-q}$ so that $\ord_\fp((\sqrt{-q} -1)/2) > 0$.
\begin{lem}\label{7.24} Let $\beta = \sqrt{-q}$. Then $F = K(\sqrt{-\beta})$.
\end{lem}
\begin{proof}  The field $K(B_4)$ is an abelian extension of $K$, whose Galois group is a product of two cyclic groups of order 2, corresponding to the two subfields
$F = K(B_{\fP^2})$, and $F^* = K(B_{\fP^{*2}})$, where $\fP^*$ denotes the unramified degree 1 prime of $\mst$ above $\fp^*$. Moreover, by the Weil pairing, $K(B_4)$
contains the group $\mu_4$ of 4-th roots of unity. Thus we must have $K(B_4) = K(i, \sqrt{b})$, where $b$ is some non-zero element of $K$, and three quadratic extensions
of $K$ contained in $K(B_4)$ are then $K(i), K(\sqrt{b})$, and $K(\sqrt{-b})$. Now the prime $\fq = \sqrt{-q}\CO_K$ of $K$ does not ramify in $K(i)$, and so, making a choice of the sign of $b$,
we can assume that $F = K(\sqrt{b})$ and $F^* = K(\sqrt{-b})$. Now, by Theorem \ref{ge}, $\fq$ ramifies in both $F$ and $F^*$, but $\fp^*$ does not ramify in $F$, and $\fp$ does not ramify in $F^*$. It follows that we must have $b\CO_K = \fq \fb^2$ for some fractional ideal $\fb$ of $K$. Since the ideal $\fq$ is principal, it follows that $\fb^2$ is principal.  But $K$ has odd class number, and so the ideal $\fb$ itself must be principal, whence, modifying $b$ by a square in $K^\times$, we must have $b = \pm \beta$, where, as above $ \beta =
\sqrt{-q}$. To see which choice of sign to take, we note that
$$
ord_{\fp^*}(-\beta - 1) \geq 2,
$$
from which it follows easily that $\fp^*$ is unramified in the extension $K(\sqrt{-\beta})$. Hence we must have $F =  K(\sqrt{-\beta})$, and the proof of the lemma is complete.
\end{proof}

\begin{cor}\label{op} The curve $A^{(-\beta)}$ has good reduction outside the set of primes of $H$ dividing $\fp$.
\end{cor}
\begin{proof} Thanks to the equation \eqref{mgm}, we know that all bad primes of $A^{(-\beta)}$ must divide either 2 or 3. However, combining the previous lemma, and Theorem \ref{ge}, we see that every prime of $H$ where $A^{(-\beta)}$ has bad reduction must ramify in the field $\fF = FH$. But the only primes of $K$ which ramify in $F$ are $\fq$ and
$\fp$, and so the only primes of $H$ which ramify in $\fF$ must lie above $\fp$ or $\fq$. But the primes of $H$ above $\fq$ are not bad primes because of \eqref{mgm}, completing the proof.
\end{proof}
\noindent We can now complete the proof of Proposition \ref{7.19}. It suffices to show that there exists $\alpha \in \msj$ such that the index of the subgroup of $\CE'(F)$ generated by $u_{\alpha, 0}$ is odd. Take $\alpha = a^2 + b^2\sqrt{-q}$, where $a$ is an even positive rational integer which is prime to 3, and $b$ is an odd rational integer which is divisible by 3.
Plainly, we then have $(\alpha, 6\ff) = 1$, and $\alpha \equiv 1 \mod \fp^2$ because $\sqrt{-q} \equiv 1 \mod \fp^2$. Also, it is clear that $\alpha$ is the norm from $F$ to $K$ of the the element $a + b\sqrt{-\beta}$, so that the restriction of the Artin symbol of the ideal $\alpha\CO_K$ to $\Gal(F/K)$ is trivial. Noting that $N\alpha \equiv q \mod 4 \equiv 3 \mod 4$, we conclude that, for this choice of $\alpha$, we have
$$
N\alpha - \omega(\gamma_\alpha) \equiv 2 \mod 4.
$$
But now, combining this congruence with the formulae of Lemmas \eqref{7.22} and \eqref{7.23}, and recalling that both $h_F$ and $h$ are odd, and $W(\omega) = \pm 1$, we conclude
that the index of the subgroup of $C(F)$ generated by $u_{\alpha, 0}$ in $\CE'(F)$ modulo torsion is indeed odd, and the proof of Proposition \ref{7.19} is complete.

\medskip

We can now prove the ``main conjecture" for $\epsilon_- X(F_\infty)$. We recall that $\Gamma = \Gal(F_\infty/F)$, and that $\Lambda_\msi(\Gamma)$ denotes the Iwasawa algebra of $\Gamma$ with coefficients in $\msi$. Write $\rho_{\fP, \Gamma}$ for the restriction of $\rho_\fP$
to $\Gamma$. We recall that, in the results which follow concerning  the link between ``main conjectures" and complex $L$-values, we have fixed until further notice the embedding $\fii: \mst \to \BC$ given by \eqref{em}.
\begin{thm}\label{7.25}  For all primes $q \equiv 7 \mod 8$, we have the exact sequence of $\Lambda_\msi(\Gamma)$-modules
\begin{equation}\label{7.26}
0 \to  \epsilon_- X(F_\infty)\widehat{\otimes}_{\CO_\fp}\msi  \to  \Lambda_\msi(\Gamma)/ \mu_A \Lambda_\msi(\Gamma) \to \fM\widehat{\otimes}_{\CO_\fp}\msi \to 0,
\end{equation}
where $\fM$ is a finite $\Gamma$-module, and  $\mu_A$ is the unique element of  $\Lambda_\msi(\Gamma)$
such that, for all odd positive integers $k = 1, 3, 5, \cdots$, we have
\begin{equation}\label{7.27}
\Omega_\fp(A)^{-k}\int_{\Gamma} (\rho_{\fP, \Gamma})^kd\mu_A = (k-1)!\Omega_\infty(A)^{-k}L(\bar{\phi}^k, k)(1-\phi^k(\fp)/N\fp).
\end{equation}
\end{thm}

\noindent We note that, since $\Lambda_\msi(G)$ is the Iwasawa algebra of $\Gamma$ with coefficients in the group ring $\msi[\Delta]$, and $\epsilon_-\msi[\Delta] = \epsilon_-\msi$,
we have
\begin{equation}\label{7.28}
\epsilon_-\Lambda_\msi(G) = \epsilon_-\Lambda_\msi(\Gamma).
\end{equation}
If $m$ is any element of $\Lambda_\msi(G)$, we write $m(-)$ for the measure in $\Lambda_\msi(\Gamma)$ such that $\epsilon_-m = \epsilon_-m(-)$ under the equality
\eqref{7.28}. Now for a character of a $p$-adic Lie group the integral against a measure of the character coincides with evaluation of the character at the measure, it follows that, for every continuous homomorphism $\xi$ from $G$ to the multiplicative group of $\xbar{\BQ}_2$ such that $\xi(j) = -1$, we must have

\begin{equation}\label{7.29}
\int_G\xi dm = \int_\Gamma \xi_\Gamma dm(-),
\end{equation}
where $\xi_\Gamma$ denotes the restriction of $\xi$ to $\Gamma$. For each $\alpha \in \msj$, we recall that $\gamma_\alpha$ is the Artin symbol of $\alpha \CO_K$ in $G$.
We note that in fact $\gamma_\alpha$ always belongs to $\Gamma$ because of our hypothesis that $\alpha \equiv1 \mod \fp^2$ for $\alpha$ in $\msi$. Recall that
$\mu_{\alpha, \infty}$ is the measure in $\Lambda_\msi(G)$ satisfying \eqref{5.16}.

\begin{lem} \label{7.30} There exists a measure $\mu_A$ in $\Lambda_\msi(\Gamma)$ such that $\mu_{\alpha, \infty}(-) = (N\alpha - \gamma_\alpha)\mu_A$ for
all $\alpha \in \msj$.
\end{lem}

\begin{proof} Put $v_\alpha = N\alpha - \gamma_\alpha$. Since $v_\alpha \in \Lambda_\msi(\Gamma)$,  it follows easily  from \eqref{5.16} that, for all $\alpha, \beta \in \msj$, we have
\begin{equation}\label{7.31}
v _\beta  \mu_{\alpha, \infty}(-) = v_\alpha  \mu_{\beta, \infty}(-).
\end{equation}
Now we recall that $\Lambda_\msi(\Gamma)$ is a unique factorization domain, and we claim that we can choose $\alpha_0$ and $\beta_0$ in $\msj$ such that
$v_{\alpha_0}$ and $v_{\beta_0}$ are relatively prime. We first choose $\alpha_0 \in \msj$ such that
\begin{equation}\label{7.32}
\alpha_0 \equiv 1 \mod \fq, \, \, \alpha_0 \nequiv 1 \mod \fp^3, \, \, \alpha_0 \equiv 1 \mod (\fp^*)^2, \, \,  \alpha_0 \nequiv 1 \mod (\fp^*)^3.
\end{equation}
Thus $\gamma_{\alpha_0}$ is then a topological generator of $\Gamma$, and we can then identify $\Lambda_\msi(\Gamma)$ with the formal power series
ring $\msi[[T]]$ by mapping this topological generator to $1+T$, so that $v_{\alpha_0} =  N\alpha_0 - (1+T).$ Now choose $\beta_0 \in \msj$ such that
\begin{equation}\label{7.33}
\beta_0 \equiv 1 \mod \fq, \, \, \beta_0 \equiv 1 \mod \fp^3, \, \, \beta_0 \equiv 1 \mod (\fp^*)^2, \, \,\beta_0 \nequiv 1 \mod (\fp^*)^3.
\end{equation}
Then we claim that $v_{\alpha_0}$ and $v_{\beta_0}$ are indeed then relatively prime elements of $\msi[[T]]$. To justify this, we note that, provided $\alpha \equiv 1 \mod \fq$,
we have $\phi((\alpha)) = \alpha$, and so $\gamma_\alpha$ will act on $B_{\fP^\infty}$ by multiplication by $\alpha$. It follows easily that $\gamma_{\beta_0} = (\gamma_{\alpha_0})^c$, where $c = \log(\beta_0)/\log(\alpha_0)$, where the logarithm is taken in the $\fp$-adic completion  of $K$.  Then $v_{\beta_0} =  N\beta_0 - (1+T)^c $, and so to show that
$v_{\beta_0}$ and $v_{\alpha_0}$ are relatively prime, one sees immediately that it suffices to show that
\begin{equation}\label{7.34}
\log(\beta_0)/\log(\bar{\beta_0}) \neq \log(\alpha_0)/\log(\bar{\alpha_0}),
\end{equation}
where the symbol $\bar{z}$ denotes the complex conjugate of $z$. But we have $\ord_\fp(\log x) = \ord_\fp(x-1)$ whenever $\ord_\fp(x-1) \geq 2$. Thus it follows from
\eqref{7.32} that $\ord_\fp(\log(\alpha_0)/\log(\bar{\alpha_0})) = 0$, whereas \eqref{7.33} implies that $\ord_\fp(\log(\beta_0)/\log(\bar{\beta_0})) \geq 1$, proving \eqref{7.34}. Since
$v_{\beta_0}$ and $v_{\alpha_0}$ are relatively prime, it now follows easily from \eqref{7.31} that we must have $\mu_{\alpha_0, \infty}(-) =  v_{\alpha_0}\mu_A$
for some $\mu_A \in \Lambda_\msi(\Gamma)$.  But then, applying \eqref{7.31}  again,  it follows that $\mu_{\alpha, \infty}(-) =  v_{\alpha}\mu_A$ for all $\alpha \in \msj$, and the proof of the lemma is complete.
\end{proof}

\medskip

We can now prove Theorem \ref{7.25}. Since the map $j_{F_\infty}$ in \eqref{6.8n} is a $\Lambda_\msi(G)$-homomorphism, then recalling \eqref{7.28}, we see that
Theorem \ref{6.10n} shows immediately $j_{F_\infty}$ gives rise to an exact sequence of $\Lambda_\msi(\Gamma)$-modules
\begin{equation}\label{7.35}
0 \to (\epsilon_-U(F_\infty))\wh{\otimes}_{\CO_\fp}\msi \to \Lambda_\msi(\Gamma) \to W\wh{\otimes}_{\CO_\fp}\msi \to 0,
\end{equation}
where $W$ is some finite $\Gamma$-module. Recall that  $\bar{C}(F_\infty)$ is, by definition, generated as a $\Lambda(G)$-module by the
the norm compatible system of elliptic units $u_{\alpha, \infty}$ defined by \eqref{eu} for $\alpha$ running over $\msj$. It then follows from the results of \S5
that $j_{F_\infty}((\epsilon_-\bar{C}(F_\infty))\wh{\otimes}_{\CO_\fp}\msi)$ will be generated as a $\Lambda_\msi(\Gamma)$-module by the $\mu_{\alpha, \infty}(-)$
for $\alpha \in \msj$, which by Lemma \ref{7.30}  is equal to the ideal of $\Lambda_\msi(\Gamma)$ generated by the $ v_\alpha \mu_A$ for $\alpha \in \msj$,
where, as before, $v_\alpha = N\alpha - \gamma_{\alpha}$.
However, we claim that the $v_\alpha$ for $\alpha \in \msj$ generate the maximal ideal $(2, T)$ of $\Lambda(\Gamma) = \BZ_2[[T]]$. Indeed,
the ideal they generate is just the annihilator of the group of all 2-power roots of unity in $F_\infty$, and this group just consists of $\{\pm 1\}$ since the prime $\fp^*$
of $K$ is not ramified in $F_\infty$. Thus we have finally
\begin{equation}\label{7.36}
j_{F_\infty}((\epsilon_-\bar{C}(F_\infty))\wh{\otimes}_{\CO_\fp}\msi) = (2\msi, T)\mu_A.
\end{equation}
On combining \eqref{7.35} and \eqref{7.36}, noting that $\epsilon_-Z(F_\infty) = \epsilon_-U(F_\infty))/\epsilon_-\bar{C}(F_\infty)$, and recalling that $\epsilon_-X(F_\infty)$ has no non-zero finite subgroup by Lemma \ref{t7}, we conclude from Theorem \ref{7.18}  that  the exact sequence \eqref{7.26} is valid. Moreover, applying \eqref{7.29} with $\xi = (\rho_\fP)^k (k=1, 3, 5, \cdots )$, we obtain \eqref{7.27}
from \eqref{5.16}.
\begin{cor}\label{7.37} Assume $q \equiv 7 \mod 16$. Then, for all complex characters $\chi$ of finite order of $G$ with $\chi(j) = 1$, we have $L(\rho \chi, 1) \neq 0$.
\end{cor}
\begin{proof} Since $q \equiv 7 \mod16$, we have ${\epsilon_-}X(F_\infty) = 0$ by Theorem \ref{t4}, and so it follows from the exact sequence \eqref{7.26} that $\mu_A$
must be a unit in the Iwasawa algebra $\Lambda_\msi(\Gamma)$. Hence $\int_{\Gamma}(\rho_\fP \chi)_\Gamma d\mu_A$ will be a unit in $\msi$, and so, noting that $L(\ov{\rho\chi}, 1)$ is the complex conjugate of $L(\rho\chi, 1)$, the result follows from
the theorem in the Appendix and the fact that $(\rho_\fP \chi)(j) = -1$.
\end{proof}

\medskip

Similarly, we have

\begin{equation}\label{7.39}
\epsilon_+\Lambda_\msi(G) = \epsilon_+\Lambda_\msi(\Gamma).
\end{equation}
If $m$ is any element of $\Lambda_\msi(G)$, we write $m(+)$ for the measure in $\Lambda_\msi(\Gamma)$ such that $\epsilon_+m = \epsilon_+m(+)$ under the equality
\eqref{7.39}. For every continuous homomorphism $\xi$ from $G$ to the multiplicative group of $\xbar{\BQ}_2$ such that $\xi(j) = +1$, we then have
\begin{equation}\label{7.40}
\int_G\xi dm = \int_\Gamma \xi_\Gamma dm(+),
\end{equation}
where $\xi_\Gamma$ denotes the restriction of $\xi$ to $\Gamma$. Now exactly the same proof as in Lemma \ref{7.30} shows that there exists  $\nu_A$ in $\Lambda_\msi(\Gamma)$ such that
\begin{equation}\label{7.41}
\mu_{\alpha, \infty}(+) = (N\alpha - \gamma_\alpha)\nu_A
\end{equation}
for all $\alpha \in \msj$. Moreover, as before, we have the exact sequence
\begin{equation}\label{7.42}
0 \to (\epsilon_+U(F_\infty))\wh{\otimes}_{\CO_\fp}\msi \to \Lambda_\msi(\Gamma) \to W\wh{\otimes}_{\CO_\fp}\msi \to 0,
\end{equation}
where again $W$ is some finite $\Lambda(\Gamma)$-module. Again we have
\begin{equation}\label{7.43}
j_{F_\infty}((\epsilon_+\bar{C}(F_\infty))\wh{\otimes}_{\CO_\fp}\msi) = (2\msi, T)\nu_A.
\end{equation}

\begin{thm}\label{7.44} For all primes $q$ with $q \equiv 7 \mod 8$, the measure $\nu_A$ is a unit in $\Lambda_\msi(\Gamma)$.
\end{thm}
\begin{proof} This follows immediately from combining \eqref{7.42} and \eqref{7.43}, noting that $\epsilon_+Z(F_\infty) = \epsilon_+U(F_\infty)/\epsilon_+\bar{C}(F_\infty)$,
and then using Theorem \ref{7.16}.
\end{proof}
\begin{cor}\label{7.45} Assume $q \equiv 7 \mod 8$. Then, for all complex characters $\chi$ of finite order of $G$ with $\chi(j) = -1$, we have $L(\rho \chi, 1) \neq 0$.
\end{cor}
\begin{proof} We simply apply \eqref{7.40} with $\xi = \rho_\fP \chi$, and $m = \mu_{\alpha, \infty}$, noting that $(\rho_\fP \chi)(j) = +1$. Since $\nu_A$ is a unit in $\Lambda_\msi(\Gamma)$,
$\int_{\Gamma}(\rho_\fP \chi)_\Gamma d\nu_A$ will be a unit in $\msi$, and we then finally apply the theorem in the Appendix to give the value of $\int_G\xi d\mu_{\alpha, \infty}$.
\end{proof}

\begin{cor}\label{7.46} Assume $q \equiv 7 \mod 8$, and recall that $F = K(\sqrt{-\beta})$, where $\beta = \sqrt{-q}$. Let $B^{(-\beta)}$ denote the twist of the abelian variety
$B$ by the quadratic extension $F/K$, and let $\rho_{B^{(-\beta)}}$ be its Serre-Tate Grossencharacter.  Then $\rho_{B^{(-\beta)}}$ has bad reduction only at the prime $\fp$,  and  $L(\rho_{B^{(-\beta)}}\eta, 1) \neq 0$, where $\eta$ is any complex character of finite order of $\Gal(K_\infty/K)$.
\end{cor}

\begin{proof} We have $\rho_{B^{(-\beta)}} = \rho \omega$, where $\omega$ again denotes the non-trivial character of $\Delta$. It is readily verified that $\fq$ does not divide the
conductor of $\rho \omega$, whence $B^{(-\beta)}$ will have good reduction at $\fq$. The final assertion then follows from Corollary \ref{7.45}, on noting that  we have $(\omega \eta)(j) = -1$.
\end{proof}

Now, as in the Introduction, let $\fM_K$ denote the set  of all non-zero integers $M$ in $\CO_K$, which are prime to $q$, satisfy $M \equiv 1 \mod 4$, and are not squares in $K$.
For each $M \in \fM_K$, let $B^{(M)}$ be the abelian variety defined over $K$ which is the twist of $B$ by the quadratic extension $K(\sqrt{M})/K$, and let $\rho_{B^{(M)}}$ be the Serre-Tate character of $B^{(M)}$. For each $n \geq 0$, let  $F(M)_n$ be the field obtained by adjoining to $K$ the coordinates of the $\fP^{n+2}$-division points of $B^{(M)}$.
Since $B^{(M)}$ has good reduction at $\fp$ because $M \equiv 1 \mod 4$, it is easily seen that $[F(M)_n : K] = 2^{n+1}$, and that $F(M)_n$ is a quadratic extension of the field $K_n$
for all $n \geq 0$.
As usual, for each $n \geq 0$, we write $N_{F(M)_n/K}$ and $N_{K_n/K}$ for the norm maps for the extensions $F(M)_n/K$ and $K_n/K$.

\begin{thm}\label{7.47} Assume that $q \equiv 7 \mod 8$, and that $M \in \fM_K$. Then, for all $n \geq 0$, we have $\ord_{s=1}L(\rho_{B^{(M)}}\circ N_{F(M)_n/K}
, s) = \ord_{s=1}L(\rho_{B^{(M)}}\circ N_{K_n/K}, s)$. Moreover the $\BZ$-rank of $B^{(M)}(F(M)_n)$ = the $\BZ$-rank of $B^{(M)}(K_n)$ for all $n \geq 0$.
\end{thm}
\begin{proof} To lighten the notation, write $\fB = B^{(M)}$. An entirely similar argument to that given in the  proof of Theorem \ref{ge} shows that
$\fB$ has good reduction everywhere over the field $F(M)_0 = K(\fB_{\fP^2})$.  Now, since $F = K(\sqrt{-\beta})$ by Lemma \ref{7.24}, and $\fB$ is isomorphic to $B$  over the field $K(\sqrt{M})$, we see that $F(M)_0$ must be one of the three quadratic extensions of $K$ contained in the field $K(\sqrt{-\beta}, \sqrt{M})$. But, since $\fB$ has good reduction everywhere over the field $F(M_0)$, and the set of bad primes of $\fB$ over $K$ consists of $\fq$ and the primes of $K$ dividing $M$,  all of these bad primes must necesarily ramify in $F(M)_0$, whence $F(M)_0 = K(\sqrt{-\beta M}).$ In particular, writing $\fB^{(-\beta M)}$ for the twist of $\fB$ by the quadratic extension $F(M)_0$, it follows that
\begin{equation}\label{7.48}
\rho_\fB \circ N_{F(M)_0/K} = \rho_\fB \rho_{\fB^{(-\beta M)}}.
\end{equation}
But, since $\fB = B^{(M)}$, we must have $\fB^{(-\beta M)} = B^{(-\beta)}$, whence
\begin{equation}\label{7.49}
\rho_{\fB} \circ N_{F(M)_0/K} = \rho_{\fB} \rho_{B^{(-\beta)}}.
\end{equation}
Since  $L( \rho_{B^{(-\beta)}} \circ N_{K_n/K}, s) = \prod_{\eta} L(\rho_{B^{(-\beta)}}\eta, s)$, where the product is taken over all $2^n$ complex characters $\eta$ of finite order of $\Gal(K_n/K)$, the first conclusion of Theorem \ref{7.47} follows from Corollary \ref{7.46}. The proof of the second assertion of Theorem \ref{7.47} is entirely parallel to that given
for Theorem \ref{t13} in \S 3, and we omit the detailed arguments.
\end{proof}

Finally, we note that the proofs we have given  in \S 5 and above are valid for every choice of the embedding $\fii: \mst \to \BC$ made at the begining of section \S 2.
In fact, there are precisely $h$ such embeddings lying above our fixed embedding of $K$ in $\BC$, and we now denote these distinct embeddings by $i^{(r)} \, \, (r=1, \ldots, h)$.
Let $\rho^{(r)}$ denote the complex Grossencharacter of $B/K$ relative to the embedding $\i^{(r)}$, and let $\psi_{A/H}$ denote the complex Grossencharacter of the elliptic curve
$A/H$. Then we have
\begin{equation}\label{7.50}
L(\psi_{A/H}, s) = \prod_{r=1}^{h}L(\rho^{(r)}, s).
\end{equation}
Since Corollaries \ref{7.37} and \ref{7.45} holds for each of the $\rho^{(r)} \, \, (r=1, \ldots, h)$, the assertion of Theorem \ref{t1} follows easily. By an entirely similar argument, one shows that Theorem \ref{tm} follows from Corollary \ref{7.46}. Finally, recalling that Corollary \ref{7.46}
and Theorem \ref{7.47} are valid for all primes $q \equiv 7 \mod 8$,  it is also clear that Theorem \ref{t2} follows from Theorem \ref{7.47} by an analogous argument.

\medskip

Finally, for the proof of Theorem \ref{tl}, we note that, writing $L(A/J,s)$ for the complex $L$-series of a finite extension $J$ of $H$, Theorem \ref{t1} tells us that, when
$q \equiv 7 \mod 16 $, we have $L(A/J, 1) \neq 0$ whenever $J \subset \fF_\infty$. But then a standard argument, whose details we omit,  shows that the finiteness of both $A(J)$ and $\Sha(A/J)(\fp)$ follow from this non-vanishing result and the main conjecture for $A/\fF_\infty$, which is proven in \cite{CKLT}.

\section {Related results}

Combining Theorems \eqref{t4} and  \eqref{t9}, we have shown in \S3 that, for all primes $q \equiv 7 \mod 16$, we have $\Sel_{\fP^\infty}(B/F_\infty) = 0$, where we recall that
$F_\infty = K(B_{\fP^\infty})$. However, the following theorem shows that nothing like this is valid for the $\fp^\infty$-Selmer group of the elliptic curve $A$ over the field
\begin{equation}\label{8.1}
\fF_\infty = H(A_{\fp^\infty}) = HF_\infty.
\end{equation}
If $L$ is any algebraic extension of $H$, we write $\Sel_{\fp^\infty}(A/L)$ for the classical $\fp^\infty$-Selmer group of $A/L$. We also write $M(\fF_\infty)$ for the maximal abelian $2$-extension of $\fF_\infty$, which is unramified outside the primes of $\fF_\infty$ lying above $\fp$, and put $X(\fF_\infty) = \Gal(M(\fF_\infty)/\fF_\infty)$.  Since
$A$ has good reduction everywhere over $\fF_\infty$ by Theorem \eqref{ge}, we have
\begin{equation}\label{8.2}
\Sel_{\fp^\infty}(A/\fF_\infty) = \Hom(X(\fF_\infty), A_{\fp^\infty}).
\end{equation}
\begin{thm}\label{8.3} Assume $q \equiv 7 \mod 16$. Then, provided there is more than one prime of $H$ lying above $\fp$, or equivalently provided the ideal class
of $\fp$ does not have order exactly equal to $h$, we have $\Sel_{\fp^\infty}(A/\fF_\infty) = \Sha(A/\fF_\infty)(\fp) = (K_\fp/\CO_\fp)^{m_q}$, for some integer $m_q > 0$.
\end{thm}

\begin{proof} We recall that $H_\infty = HK_\infty$.  Let $M(H_\infty)$ be the maximal abelian $2$-extension of $H_\infty$ which is unramified outside the primes of $H_\infty$
above $\fp$, and write $X(H_\infty) = \Gal(M(H_\infty)/H_\infty)$. It is proven in \cite{CKL} that $X(H_\infty)$ is a free finitely generated $\BZ_2$-module. By entirely similar arguments, based on Nakayama's Lemma, to those given in section 3, one can then show that $X(\fF_\infty)$ is also a free finitely generated $\BZ_2$-module.
It follows from \eqref{8.2} that $\Sel_{\fp^\infty}(A/\fF_\infty)$ is a direct sum of a finite number of copies of $K_\fp/\CO_\fp$. Since Theorem \ref{tl} proves that $A(\fF_\infty)$
is a torsion group, to complete the proof it suffices to show that $\Sel_{\fp^\infty}(A/\fF_\infty) \neq 0$, which is equivalent to showing that
\begin{equation}\label{8.4}
\Sel_{\fp^\infty}(A/\fF_\infty)^{\Gamma} \neq 0,
\end{equation}
where we are now using the symbol $\Gamma$ for $\Gal(\fF_\infty/\fF)$. If $L$ is any algebraic extension of $H$, we define the modified Selmer group
\begin{equation}\label{8.5}
\Sel'_{\fp^\infty}(A/L)=\Ker\left(H^1(L,A_{\fp^\infty})\ra \prod_{v\nmid \fp}H^1(L_v, A)(\fp)\right),
\end{equation}
where now the product is taken over all primes $v$ of $L$ which do not lie above the prime $\fp$ of $K$. Since $A$ has good reduction everywhere over the
field $\fF$, and the extension $\fF_\infty/\fF$ is unramified outside $\fp$, entirely similar arguments to those given in the proofs of Theorem \ref{t9} and
Proposition \ref{t12} show that $\Sel_{\fp^\infty}(A/\fF_\infty) =\Sel'_{\fp^\infty}(A/\fF_\infty)$, and that
\begin{equation}\label{8.6}
\Sel_{\fp^\infty}(A/\fF_\infty)^{\Gamma} = \Sel'_{\fp^\infty}(A/\fF).
\end{equation}
Now we have the obvious exact sequence
\begin{equation}\label{8.7}
0 \to \Sel_{\fp^\infty}(A/\fF) \to \Sel'_{\fp^\infty}(A/\fF) \to \prod_{w|\fp}H^1(\fF_w, A)(\fp),
\end{equation}
where $w$ runs over the places of $\fF$ dividing $\fp$. Denoting the right hand map in this last exact sequence by $g$, the work of Cassels-Poitou-Tate (see, for example, Corollary 4 of the Appendix of \cite{PR}) gives the following exact description of the cokernel of $g$. Let $\pi$ now denote a non-zero element of $\CO_K$, such that $\pi\CO_K = \fp^r$ for some integer $r \geq 1$, and let $\pi^*$ denote the complex conjugate of $\pi$.
By Tate local duality, $H^1(\fF_w, A)(\fp)$ is dual to $\varprojlim_{n} A(\fF_w)/{\pi^*}^nA(\fF_w)$, and this latter group is easily seen to to be isomorphic to $\tilde{A}_w(k_w)(\fp^*)$, where $\tilde{A}_w$ denotes the reduction of $A$ modulo $w$, and $k_w$ is the residue field of $w$. Define
\begin{equation}\label{8.8}
\fS_{{\fp^*}^\infty} = \varprojlim_n\Sel_{{\pi^*}^n}(A/\fF),
\end{equation}
where $\Sel_{{\pi^*}^n}(A/\fF)$ denotes the classical Selmer group of $A/\fF$ relative to the endomorphism ${\pi^*}^n$ of  $A$. Now by Theorem \ref{t1}, we have
$L(A/\fF, 1) \neq 0$, and
and we have already noted in Theorem \ref{tl} that this implies that $A(\fF)$ and $\Sha(A/\fF)(\fp)$ are both finite. Then analogous arguments to those used to prove Corollary 16 in \cite{PR} for split odd primes $p$ will also enable one to show that, even for $p=2$,
the finiteness of $\Sha(A/\fF)(\fp)$ implies in turn the finiteness of $\Sha(A/\fF)(\fp^*)$. Alternatively, we could use the main conjecture for $A$ over the $\BZ_2$-extension $\fF K_\infty^*/\fF$, where $K_\infty^*$ is the unique $\BZ_2$-extension of $K$ unramified outside $\fp^*$, to show that the the non-vanishing of $L(A/\fF, 1)$ implies the finiteness of  $\Sha(A/\fF)(\fp^*)$.
It then follows easily that we have
\begin{equation}\label{8.9}
\fS_{{\fp^*}^\infty} = A(\fF)(\fp^*).
\end{equation}
But $A(\fF)(\fp^*) = A_{\fp^*}$ because the primes of $H$ above $\fp^*$ are unramified in $\fF$.  Further, we note that  $\tilde{A}_w(k_w)(\fp^*) \neq 0$ for every prime $w$ of $\fF$ above $\fp$, since reduction modulo $w$ is injective on $A_{\fp^*}$. Putting all together, the theorem of Cassels, Tate, Poitou shows finally that the cokernel of the map $g$ is dual to the image of the natural injection
\begin{equation}\label{8.10}
A_{\fp^*} \to \prod_{w|\fp}\tilde{A}_w(k_w)(\fp^*).
\end{equation}
In particular, we see that the cokernel of $g$ has order exactly 2. Finally, we note that  $\tilde{A}_w(k_w)(\fp^*) \neq 0$ for every prime $w$ of $\fF$ above $\fp$, since reduction modulo $w$ is injective on $A_{\fp^*}$. Hence $\prod_{w|\fp}H^1(\fF_w, A)(\fp)$ will always have order strictly greater than the cokernel of the map $g$ when there is more than one prime $w$ of $H$ above $\fp$. This shows that $\Sel'_{\fp^\infty}(A/\fF_\infty) \neq 0$, when there is more than one prime of $H$ above $\fp$, and  the proof of the theorem is complete.
\end{proof}
\noindent We note that Theorems \ref{t4} and \ref{t9} show that $\Sha(B/F_\infty)(\fP) = 0$ when $q \equiv 7 \mod 16$.  However, when there is more than one prime of $H$ above $\fp$, this seems in contrast with the fact that $\Sha(A/\fF_\infty)(\fp)$ is a divisible group of strictly positive $\BZ_2$-corank.  We are grateful to J. Milne for the following comment about this situation. Let $\fL/L$ be any finite Galois extension of subfields of the algebraic closure $\ov{\BQ}$ of $\BQ$, let $\fJ$ be any abelian variety defined over $\fL$, and let $J$ be the restriction of scalars of $\fJ$ from $\fL$ to $L$.    Then the $\Gal(\ov{\BQ}/L)$-module $J(\ov{\BQ})$ is the induced representation of the  $\Gal(\ov{\BQ}/\fL)$-module $\fJ(\ov{\BQ})$, and similarly at the completions at all places of $L$. Hence the local and global Galois cohomology groups of these representaions coincide. It follows, in particular, that
$\Sha(\fJ/\fL) = \Sha(J/L)$.  Thus, applying this remark to the extension $\fF_\infty/F_\infty$, we conclude that $\Sha(A/\fF_\infty) = \Sha(B/F_\infty)$. This also implies that the
endomorphism ring $\msb$ of $B/K$ operates on $\Sha(A/\fF_\infty)$, and so in particular, we have
\begin{equation}\label{s1}
\Sha(A/\fF_\infty)(\fp) = \oplus_{\fQ|\fp}\Sha(A/\fF_\infty)(\fQ),
\end{equation}
where now $\fQ$ runs through all the primes of $\mst$ lying above the prime $\fp$ of $K$. Note also that all primes of $\mst$ above $\fp$ are unramified because of result of Gross mentioned in $\S2$ and the fact that $h$ is odd.  Moreover, $\fP$ is the unique annihilator of $A_\fp = B(K)_\fp$ lying inside $\mst$, so that all the primes of $\fQ$ of $\mst$ lying above $\fp$ and distinct from $\fP$ must have residue field strictly bigger than $\BF_2$. Hence the arguments of $\S2$ cannot be applied to show that $\Sha(A/\fF_\infty)(\fQ) = 0$ when $\fQ \neq \fP$. Moreover, if we assume that there is more that one prime of $H$ above $\fp$, we certainly have $H \neq K$, and thus there is at least one prime $\fQ$ of $\mst$, distinct from $\fP$, above $\fp$. Theorem \ref{8.3} then proves that we must have  $\Sha(A/\fF_\infty)(\fQ) \neq 0$
for at least one prime $\fQ$ of $\mst$ above $\fp$ and distinct from $\fP$.

\medskip

Finally, we record without proof a special case of an old result of Wiles and the first author \cite{CW1} (cf. Theorem 11 of \cite{CW1}, but note that the proof given there assumes $p > 2$, so that one has to rework the argument slightly to handle the case $p=2$). Let $M(F)$ denote the maximal abelian $2$-extension of $F$ unramified outside the unique prime of $F$ lying above $\fp$. Obviously $M(F) \supset F_\infty$, and it is a nice exercise in global class field theory to show that $[M(F):F_\infty]$ is finite, and establish the following exact formula for this degree. One also has to use the fundamental fact proven earlier (Theorem \ref{odd}) that $F$ always has odd class number.

\begin{thm}\label{8.11} Assune that $ q \equiv 7 \mod 8$. Let $v$ denote the unique prime of $F$ lying above $\fp$, and let $\log_v$ denote the $v$-adic logarithm. Then
\begin{equation}\label{8.12}
[M(F):F_\infty] = 2^t , \, {\rm where} \,  t = (\ord_v(\log_v(\eta)) - 2)/2,
\end{equation}
and $\eta$ is any generator of the unit group of $F$ modulo torsion.
\end{thm}
\noindent Since $X(F_\infty)_\Gamma = Gal(M(F)/F_\infty)$, we obtain as an immediate corollary of \eqref{8.12} and Theorem \ref{t4}:

\begin{cor}\label{8.13} When the prime $q$ satisfies $q \equiv 7 \mod 16$, we always have $\ord_v(\log_v(\eta)) = 2$.
\end{cor}

\noindent We are extremely grateful to Zhibin Liang for having computed for us the fundamental unit $\eta$ and $\ord_v(\log_v(\eta))$ for all primes $q \equiv 15 \mod 16$
with $q < 2500$. In particular, when combined with Theorem \ref{8.11}, Liang's calculations gave strong numerical evidence that $X(F_\infty) \neq 0$ for all primes $q \equiv 15 \mod 16$. We have to confess that we cannot see how to prove this last statement using the techniques of Iwasawa theory.  However, recently Jianing Li \cite{JL} of the University of Science and Technology, Hefei, has discovered an ingenious elementary proof of both Corollary \ref{8.13} and the fact that $\ord_v(\log_v(\eta)) \geq 4$ when $q \equiv 15 \mod 16$.
Thus, for all primes $q \equiv 15 \mod 16$, Li's result combined with \eqref{8.12} shows that $M(F) \neq F_\infty$, and so, since $X(F_\infty)_\Gamma = Gal(M(F)/F_\infty)$, it follows from Nakayama's lemma that $X(F_\infty) \neq 0$, which in turn implies that $X(F_\infty)$ has positive $\BZ_2$-rank by Theorem \ref{t4}.

\medskip

\section{Zhao's Method}

In this section, we prove Theorem \ref{main} using Zhao's method \cite{Zhao1}, \cite{Zhao2}. As in the Introduction, $\msr$ will denote the set of all square free positive integers $R$ of the form $R = r_1...r_k$ , where $k\geq 0$,  and $r_1, \dots, r_k$ are distinct primes such that (i) $r_i \equiv 1 \mod 4$, and  (ii) $r_i$ is inert in $K$, for $i=1,..., k$.  For the rest of this section, $R$ will denote an arbitrary element of $\msr$. For each positive divisor $d>1$ of $R$, we let $\chi_d$ be the non-trivial character of $\Gal(K(\sqrt{d})/K)$, and we define $\phi_d = \phi\chi_d$, where, as always, $\phi$ denotes the Hecke character attached to the abelian variety $B/K$ which is the restriction of scalars of $A$ from $H$ to $K$. We also put $\phi_1 = \phi$. Recall that  $\fq = \sqrt{-q}\CO_K$ is the conductor of $\phi$. We  recall  that $i: \mst  \to \BC$ is any embedding  which extends our given embedding of $K$ into $\BC$, so that we can view all of the Hecke characters $\phi_d$ as being complex valued.

\begin{lem}\label{3.1n}
For each positive divisor $d$ of $R$, the Hecke character $\phi_d$ has conductor $d\fq$. Moreover, if $r$ is any prime dividing $R$, which does not divide $d$, then
$\phi_d(r\CO_K) = -r$.
\end{lem}

\begin{proof}
The positive integer $d$ is square free, and satisfies $d \equiv 1 \mod 4$ since each prime dividing $d$ is $\equiv 1 \mod 4$, from which it follows that $\chi_d$ has conductor $d\CO_K$. But $\phi$ has conductor $\fq$, which is prime to $d\CO_K$, and thus $\phi_d$ will have conductor $d\fq$. Let $\fr = r\CO_K$. Since $-r$ is a square modulo $q$, and $h$ is odd, the explicit formula for $\phi$ given at the beginning of \S 2 of \cite{BG} shows that $\phi(\fr) = -r$. On the other hand, since $r$ is inert in $K$, and the Galois group of
$K(\sqrt{d})/\BQ$ is not cyclic when $d > 1$, the prime $\fr$ of $K$ must split in $K(\sqrt{d})$, and so we must have $\chi_d(\fr) =1$, whence the second assertion of the lemma follows.
\end{proof}

\medskip

For each positive divisor $d$ of $R$, let $A^{(d)}$ be the twist of $A$ by the extension $H(\sqrt{d})/H$ and $B^{(d)}$ the twist of $B$ by the extension $K(\sqrt{d})/K$. Then $B^{(d)}$
is the restriction of scalars from $H$ to $K$ of $A^{(d)}$, and $B^{(d)}$ has Hecke character $\phi_d$. Recall that $\omega$ denotes the \Neron differential of our global minimal equation \eqref{2.6} of $A/H$, and that its complex period lattice is $\CL = \Omega_\infty(A)\CO_K$. Gross has shown in \cite{Gross2} (see Prop. 4.3) that the differential $\omega/\sqrt{d}$ on $A/H(\sqrt{d})$ descends to a global minimal differential on $A^{(d)}/H$ which we denote by $\omega(d)$. The following lemma is then clear from Gross' result and
Prop. 4.10, (vi) of \cite{GS}.

\begin{lem}\label{minn} For each positive divisor $d$ of $R$, the complex period lattice of $\omega(d)$ is equal to $\Omega_\infty(A)\CO_K/\sqrt{d}$. Let $E = A^{(d)}$, and, for each integral ideal $\fa$ of $K$ prime to $R\fq$, let $E^\fa$ be the curve obtained by applying the Artin symbol of $\fa$ to the coefficients of the global minimal equation for $E$, and let $\omega(d)_\fa$ be the \Neron differential of $E^\fa$. Then we have $\eta_E(\fa)^*(\omega(d)_\fa)=\xi_d(\fa)\omega(d)$, where $\xi_d(\fa) = \xi(\fa)/\chi_d(\fa)$.
\end{lem}

Let $\fa$ be any integral ideal of $K$ with $(\fa, R\fq)=1$, and, as before, let $\gamma_\fa$ denote the ideal class of $\fa$. Always assuming that $d$ is a positive divisor of $R$, we define the imprimitive partial Hecke $L$-series
\begin{equation}\label{2n}
L_{R\fq}(\ov{\phi_d},\gamma_\fa, s)=\sum_{ (\fb,R\fq)=1, \fb \in\gamma_\fa}\frac{\ov{\phi_d}(\fb)}{N(\fb)^s},
\end{equation}
where the sum on the right is taken over all integral ideals $\fb$ of $K$, which are prime to $R\fq$, and which lie in the class $\gamma_\fa$. It is classical
that the Dirichlet series on the right converges for $R(s) > 3/2$, and it has a holomorphic continuation to the whole complex plane. We recall a classical formula for $L_{R\fq}(\ov{\phi_d},\gamma_\fa, 1)$, which essentially goes back to the 19th century (see \cite{GS}). Recall that, for any lattice $\msl$,
$\CE^*_1(z,\msl)=H_1(z,1,\msl)$ is the Eisenstein series of weight 1 defined earlier.
We write $K(R\fq)$ for the ray class field of $K$ modulo $R\fq$, and let $\textrm{Tr}_{K(R\fq)/H}$ be the trace map from $K(R\fq)$ to $H$.
\begin{prop}\label{3.3n}
Assume $R \in \msr$, and let $d$ be any positive divisor of $R$. Then, for all integral ideals $\fa$ of $K$, which are prime to $R\fq$, we have
\begin{equation}\label{4n}
\frac{\phi_d(\fa)R\sqrt{-qd}}{\xi_d(\fa)}\cdot \frac{L_{R\fq}(\ov{\phi}_d,\gamma_\fa,1)}{\Omega_\infty(A)}=\textrm{Tr}_{K(R\fq)/H}\left(\CE^*_1\left(\frac{\xi_d(\fa)\Omega_\infty(A)}{R\sqrt{-qd}},\frac{1}{\sqrt{d}}\CL_\fa\right)\right).
\end{equation}
\end{prop}

\begin{proof} We apply Prop.  5.5 of \cite{GS} to the curve $E=A^{(d)}$ over  the field $H$, with $\fg = R\fq$, and  $ \rho = \Omega_\infty(A)/(R\sqrt{-qd})$. Since the conductor of the Grossencharacter $\phi_d$ is equal to $d\fq$ by  Lemma \ref{3.1n}, the obvious analogue of Lemma \ref{g2.6} shows that $H(E_\fg)$ is equal to the ray class field  $K(R\fq)$. Moreover, we have  $\eta_E(\fa)^*(\omega(d)_\fa)=\xi_d(\fa)\omega(d)$ by Lemma \ref{minn}. It then follows from Prop. 5.5 of \cite{GS} that
\begin{equation}\label{5n}
\frac{\phi_d(\fa)R\sqrt{-qd}}{\xi_d(\fa)}\cdot \frac{L_{R\fq}(\ov{\phi}_d,\gamma_\fa,1)}{\Omega_\infty(A)} = \sum_{\fb \in \fB} \CE^*_1\left(\frac{\phi_d(\fb)\xi_d(\fa)\Omega_\infty(A)}{R\sqrt{-qd}},\frac{1}{\sqrt{d}}\CL_\fa\right),
\end{equation}
where $\fB$ denotes any set of integral prime ideals of $K$, prime to $R\fq$, whose Artin symbols in $\Gal(K(R\fq)/K)$ give precisely $\Gal(K(R\fq)/H)$, and, as in \S4, $\CL_\fa = \xi(\fa)\Omega_\infty(A)\fa^{-1}$.   However, by \cite{GS},
Theorem 6.2, we note that, since the Artin symbol of every ideal $\fb \in \fB$ fixes the field of definition $H$ of $E$, the right hand side of \eqref{5n} is none other than the right hand side of \eqref{4n}. This completes the proof.
\end{proof}

\medskip

We now introduce the two fields
\begin{equation}\label{6n}
J_R = K(\sqrt{r_1}, \dots, \sqrt{r_k}), \, \, H_R = H(\sqrt{r_1}, \dots, \sqrt{r_k}),
\end{equation}
where, as always, the $r_i, \, 1\leq i \leq k$, are the distinct prime factors of $R$.
\begin{lem}\label{3.4n} We have $J_R \cap H = K$, $[H_R:J_R] = h$, and $H_R \subset K(R\fq)$. Moreover, for each positive divisor $d$ of $R$, $B^{(d)}$ is isomorphic to
$B$ over $J_R$, and $A^{(d)}$ is isomorphic to $A$ over $H_R$.

\end{lem}
\begin{proof} Since the class number $h$ is odd, and $[J_R:K] = 2^k$ by Kummer theory, it follows that $J_R \cap H = K$. Moreover, the extension $K(\sqrt{r_i})/K$ has
conductor $r_i\CO_K$ since $r_i \equiv 1 \mod 4$, and thus this extension is a subfield of $K(R\fq)$. The final assertion of the lemma is also clear since $J_R$ contains $\sqrt{d}$.
\end{proof}

It follows from \eqref{5n} that, for each divisor $d$ of $R$, the partial $L$-value $\sqrt{d}L_{R\fq}(\ov{\phi}_d,\gamma_\fa,1)/\Omega_\infty(A)$ belongs to the compositum of fields $H\mst$.  Recall that $\fP$ is our degree one prime of $\mst$ above $\fp$. Take any prime $v$ of $H_R\mst$ lying above the prime $\fP$ of $\mst$. We also assume $R$ is fixed for the limit arguments which follow. For each $n \geq 0$,  let
$\msc_n$ be a set of integral ideals of $K$, prime to $R\fq$, whose Artin symbols in $\Gal(H_R(A_{\fp^{n+2}})/K))$ give precisely $\Gal(H_R(A_{\fp^{n+2}})/J_R(B_{\fP^{n+2}}))$.
One sees easily that, for each $n \geq 0$, $\msc_n$ gives a complete set of representatives of the ideal class group of $K$. Moreover, since $J_R(B_{\fP^{n+2}}) = J_R(B^{(d)}_{\fP^{n+2}})$, we conclude that
\begin{equation}\label{7n}
\phi_d(\fa) \equiv 1 \mod \fP^{n+2} \, \, \, \, {\rm for \, all}  \, \, \fa \in \msc_n.
\end{equation}
\noindent Note that, for any lattice $\msl$ and any $\lambda\neq 0\in\BC$, we have $\CE^*_1(\lambda z,\lambda \msl)=\lambda^{-1}\CE^*_1(z,\msl).$
Hence, summing the formula \eqref{5n} over all $\fa\in \msc_n$, and taking $\lambda = 1/ (\sqrt{d}\chi_d(\fa))$, we immediately obtain the equation
\begin{equation}\label{8n}
  \sum_{\fa\in \msc_n}\phi_d(\fa)L_{R\fq}(\ov{\phi}_d,\gamma_\fa,1)/\Omega_\infty(A)=\sum_{\fa\in\msc_n} \xi(\fa)\sum_{\sigma\in\Gal(K(\fq R)/H)}\left(\sqrt{d}\right)^{\sigma-1}\frac{1}{R\sqrt{-q}}\CE^*_1\left(\frac{\xi(\fa)\Omega_\infty(A)}{R\sqrt{-q}},\CL_\fa\right)^\sigma.
\end{equation}
Now the values $L_{R\fq}(\ov{\phi}_d,\gamma_\fa,1)/\Omega_\infty(A)$ are independent of $n$ since $\msc_n$ is a complete set of representatives of the ideal class group of K. Thus, we conclude from \eqref{7n} that the left hand side of \eqref{8n} converges $v$-adically as $n \to \infty$ to $ L_{R\fq}(\ov{\phi}_d,1)/\Omega_\infty(A)$, assuming $R$ is fixed. Therefore, the right hand side of \eqref{8n}  also converges $v$-adically as $n\to \infty$, and so we have proven the following result.

\begin{lem}\label{special-val}
For every positive integer divisor $d$ of $R$, we have
\begin{equation}\label{9n}
 L_{R\fq}(\ov{\phi}_d,1)/\Omega_\infty(A) =\lim_{n\ra \infty}\sum_{\fa\in\msc_n}\xi(\fa)\sum_{\sigma\in\Gal(K(R\fq)/H)}\left(\sqrt{d}\right)^{\sigma-1}\frac{1}{R\sqrt{-q}}\CE^*_1\left(\frac{\xi(\fa)\Omega_\infty(A)}{R\sqrt{-q}},\CL_\fa\right)^\sigma.
\end{equation}.
\end{lem}
\noindent One of the key idea in Zhao's induction method is to sum the formula \eqref{9n} over all positive integer divisors $d$ of $R$, and then make use of the following well known lemma.
\begin{lem}
 Recall that $R = r_1 \dots r_k$, where the $r_i$ are distinct prime numbers. Let  $\sigma$ be any element of $\Gal(K(R\fq)/H)$. Then, letting $d$ run over all positive integer divisors of $R$, the expression  $\sum_{d\mid R}(\sqrt{d})^{\sigma-1}$ is equal to $2^k$ if $\sigma \in \Gal(K(R\fq)/H_R)$, and is equal to $0$ otherwise.
\end{lem}
\begin{proof} We quickly recall the proof. The first assertion of the lemma is clear. To prove the second assertion, suppose that $\sigma$ maps $j \geq 1$ elements of the set $\{\sqrt{r_1}, \dots, \sqrt{r_k}\}$ to minus themselves, and write $V(\sigma)$ for the subset consisting of all such elements. If $d$ is any positive integer divisor of $R$, then
$\sigma$ will fix $\sqrt{d}$ if and only if $d$ is a product of an even number of elements of $V(\sigma)$ with an arbitrary number of elements of the complement of $V(\sigma)$
in $\{\sqrt{r_1}, \dots, \sqrt{r_k}\}$. Thus the total number of $d$ such that $\sqrt{d}$ is fixed by $\sigma$ is equal to
$$
2^{k-j}((j, 0) + (j, 2) + (j, 4) + \dots) = 2^{k-1},
$$
where $(j, i)$ denotes the number of ways of choosing $i$ objects from a set of $j$ objects.  Similarly, the total number of $d$ such that $\sigma$ maps $\sqrt{d}$ to
$-\sqrt{d}$ is equal to
$$
2^{k-j}((j, 1) + (j, 3) + (j, 5) + \dots) = 2^{k-1},
$$
and the second assertion of the lemma is now clear.
\end{proof}
\noindent For each $\fa \in \msc_n$, define
\[\Psi_{\fa, R} =\textrm{Tr}_{K(R\fq)/H_R}\left(\frac{1}{R\sqrt{-q}}\cdot\CE^*_1\left(\frac{\xi(\fa)\Omega_\infty(A)}{R\sqrt{-q}},\CL_\fa\right)\right).\]
 In view of the above lemma, it follows that  Lemma \ref{9n} implies the following result.
\begin{prop}\label{key-identity}
Letting $d$ runs over all positive integer divisors of $R$, we have
\begin{equation}\label{10n}
  \sum_{d\mid R}L_{R\fq}(\ov{\phi}_d,1)/\Omega_\infty(A)=2^k \cdot \lim_{n\ra \infty}\sum_{\fa\in\msc_n} \xi(\fa)\Psi_{\fa, R}.
\end{equation}
\end{prop}

 Finally, we shall need the following key integrality result (see \cite{coates3}, \cite{CLTZ}).
\begin{prop}\label{inte-2}
For all $\fa\in\msc_n$, $\Psi_{\fa, R}$ is integral at all places of $H_R$ above $2$.
\end{prop}

\begin{proof}
We briefly recall the proof given in \cite{CLTZ}. Write $\fJ=H(A^\fa_\fq)$, which is also the ray class field $K(\fq)$.
Since $A^\fa$ is a relative Lubin-Tate formal group, in the sense of De Shalit \cite{DS2}, at each prime of $H$ lying above the set of primes of $K$ dividing $R$, it is easily seen that the action of the  Galois group $\Gal(K(R\fq)/\fJ)$ on $A^\fa_R$ gives an isomorphism
\[\tau: \Gal(K(R\fq)/\fJ)\simeq \left(\CO_K/R\CO_K\right)^\times.\]
Since $q$ is prime to $R$, we can find $\alpha,\beta$ in $\CO_K$ such that $1=\alpha R+\beta \sqrt{-q}$.  We then define
\[z_1=\frac{\xi(\fa)\alpha\Omega_\infty(A)}{\sqrt{-q}},\quad z_2=\frac{\xi(\fa)\beta\Omega_\infty(A)}{R}.\]
and write $P_1$ and $P_2$ for the corresponding points on $A^\fa$ under the Weierstrass isomorphism.
For any $b \in H$, let $b^\fa$ denote the image of $b$ under the Artin symbo of $\fa$. Define $\epsilon$ to be the inverse image of the class $-1\mod R\CO_K$ under the isomorphism $\tau$, and let $\fS$ be the fixed field of $\epsilon$, so that the extension $K(R\fq)/\fS$ has degree $2$. Of course, $\fS$ contains $H_R$ because $-1$ is a square modulo $r_j$ for $j=1,\dots k$. Defining
$\Phi_\fa=\textrm{Tr}_{K(R\fq)/\fS}\left(\Psi_\fa\right)$, we have
\[\Phi_\fa=\frac{1}{R\sqrt{-q}}\cdot\left(\CE^*_1(z_1+z_2,\CL_\fa)+\CE^*_1(z_1-z_2,\CL_\fa)\right)\]
On the other hand, by a classical identity ( see Lemma 4.3 in \cite{CLTZ}), we have
\[\CE^*_1(z_1+z_2,\CL_\fa)+\CE^*_1(z_1-z_2,\CL_\fa)=2\CE^*_1(z_1,\CL_\fa)+\frac{2y(P_1)+a^\fa_1\cdot x(P_1)+a^\fa_3}{x(P_1)-x(P_2)}.\]
However, as is explained in detail in \cite{CLTZ}, each of the two terms on the right hand side of this last equation is integral at all places of $\fS$ lying above $2$, and the assertion of the lemma follows.
\end{proof}

\medskip

As always,  $\fP$ denotes the degree 1 prime of $\mst$ lying above $\fp$ whose existence is given by Lemma \ref{t3}, and we  write $\msp$ for any of the primes of the field $H\mst$ lying above $\fP$. Now Proposition \ref{3.3n} shows  that the value $\sqrt{R}L(\ov{\phi}_R, 1)/\Omega_\infty$ belongs to $H\mst$ for all $R \in \msr$.

\begin{thm}\label{AR-nonzero}
Assume that $q\equiv 7\mod 16$, and let $R=r_1\cdots r_k$ be any element of the set $\msr$. Then, for each prime $\msp$ of $H\mst$ lying above the prime $\fP$ of $\msp$, we have
\begin{equation}\label{5.1n}
  \ord_\msp(\sqrt{R}L(\ov{\phi}_R, 1)/\Omega_\infty)=k-1.
\end{equation}
In particular, $L(\ov{\phi}_R, 1) \neq 0$.
\end{thm}
We recall that we have taken an arbitrary embedding  $i: \mst \to \BC$ in the above discussion, and each such $\iota$ then gives an embedding of $H\mst$ into $\BC$ because of our fixed embedding of $H$ into $\BC$. Since the above theorem holds for every choice of the $h$ distinct embeddings of $\mst$ into $\BC$ extending the embedding of $K$ into $\BC$, we immediately obtain the following corollary, which establishes Theorem \ref{main}.

\begin{cor} Assume that $q\equiv 7\mod 16$, and let $R$ be any element of $\msr$. Then $L(A^{(R)}/H, 1) \neq 0.$
\end{cor}

We prove Theorem \ref{AR-nonzero} by induction on the number $k$ of prime factors of $R$. We first establish some preliminary lemmas.   Recall that $2\CO_K = \fp\fp^*$, and that $F = K(B_{\fP^2})$.  Then $\fp$ and $\fq$ are the only two primes of $K$ which are ramified in the extension $F/K$. We thank Zhibin Liang for pointing out the following result to us.

\begin{lem}\label{4.1n} If $q \equiv 7 \mod 16$, then $\fp^*$ is inert in $F$, and if $q \equiv 15 \mod 16$, then $\fp^*$ splits in $F$.
 \end{lem}

\begin{proof} We recall that we have fixed the sign of  $\alpha = \sqrt{-q}$ so that $\ord_\fp((1-\alpha)/2) > 0$. Then, by Lemma \ref{7.24}, we have $F = K(\sqrt{-\alpha})$. Thus
$F$ is the splitting field over $K$ of the polynomial $g(X) = X^2 + X + (\alpha + 1)/4$. Noting that
$$
(1+\alpha)(1-\alpha)/4 = (q+1)/4, \, \, \,  (1+\alpha)/2 + (1-\alpha)/2 = 1,
$$
it follows easily that $g(X)$ modulo $\fp^*$ is equal to $X^2 + X + 1$ if $q \equiv 7 \mod 16$, and it is equal to $X^2+ X$ if $q \equiv 15 \mod 16$. The assertions of the lemma
now follow easily.
\end{proof}

\begin{cor}\label{4.2n} If $q \equiv 7 \mod 16$, then $\ord_\fP(\phi(\fp^*) - 1) = 1$, and if $q \equiv 15 \mod 16$, then $\ord_\fP(\phi(\fp^*) - 1) \geq 2$. \end{cor}

\begin{proof} The prime $\fp^*$ is unramified in the extension $F/K$, and we let $\tau$ be its Artin symbol. Since $\phi$ is the Serre-Tate homomorphism for $B/K$, we have
$\tau(Q) = \phi(\fp^*)(Q)$ for all $Q$ in $B_{\fP^2}$, whence the assertion of the corollary follows from the previous lemma.
\end{proof}

\begin{lem}\label{4.5n} For each $R \in \msr$, the extension $J_R/K$ defined by \eqref{6n} is unramified at the primes of $K$ lying above 2. \end{lem}

\begin{proof} It suffices to show that, for each prime $r$ dividing $R$, the extension $K(\sqrt{r})/K$ is unramified at the primes above 2. Put $m = (\sqrt{r}-1)/2$,
so that $V(m) = 0$, where $V(X) = X^2 + X - (r-1)/4$. But then $V'(m) = 2m-1$ is a unit at $\fp$ and $\fp^*$, and so $K(m) = K(\sqrt{r})$ is unramified at the primes of $K$ above 2.
\end{proof}

We first show that Theorem \ref{AR-nonzero} holds for $R = 1$.

\begin{prop}\label{4.4n} Assume that $q \equiv 7 \mod 16$. Then, for all primes $\msp$ of $H\mst$ above $\fP$, we have $\ord_\msp(L(\ov{\phi}, 1)/\Omega_\infty(A)) = -1$.
\end{prop}

\begin{proof} The proof makes essential use of the so called "main conjecture" for $B$ over the field $F_\infty = K(B_{\fP^\infty})$, which is given
by Theorem \ref{7.25}. Put $\Gamma = \Gal(F_\infty/F)$. Let $\msi$ be the ring of integers of the completion of the maximal unramified extension of $K_\fp$,
and write $\Lambda_\msi(\Gamma)$ for the Iwasawa algebra of $\Gamma$ with coefficients in $\msi$. Then, as is explained in the proof of Corollary \ref{7.37}
the full force of the main conjecture tells us that the measure $\mu_A$ appearing in Theorem \ref{7.25} is a unit in $\Lambda_\msi(\Gamma)$ when
$q \equiv 7 \mod 16$. Hence the integral of any continuous homomorphism from $\Gamma$ to $\msi^\times$ against this measure must be a unit in $\msi$.
Thus, recalling that the $\fp$-adic period $\Omega_\fp(A)$ is a unit in $\msi$,  it follows from equation \eqref{7.27} that
$$
\Omega_\infty(A)^{-1}L(\ov{\phi}, 1)(1- \phi(\fp)/N\fp)
$$
will be a unit in $\msi$, and thus a unit at $\fQ$. But, noting that $\phi(\fp)\phi(\fp^*) = N\fp$, the conclusion of the proposition follows immediately
from the first assertion of Corollary \ref{4.2n}.
\end{proof}

We next consider the case when $k=1$, and so $R = r$, a prime number.
 By Theorem \ref{key-identity}  for the the prime $r$, we conclude that
\begin{equation}\label{4.8n}
L(\ov{\phi}_r, 1)/\Omega_\infty(A)+ (1-\ov{\phi}((r))/r^2)L(\ov{\phi}, 1)/\Omega_\infty(A) = 2V_r,
\end{equation}
where  $V_r = \lim_{n\ra \infty}\sum_{\fa\in\fC_n} \xi(\fa)\Psi_{\fa, r}$. Let $\fW$ be any prime of $H\mst(\sqrt{r})$ lying above the prime $\fP$ of $\mst$. By Proposition \ref{inte-2}, we have  $\ord_\fW(V_r) \geq 0.$ Further, by Lemma \ref{3.1n}, we have $(1-\ov{\phi}((r))/r^2) = (1 + 1/r)$. Thus, since $r+1 \equiv 2 \mod 4$, it follows from Proposition \ref{4.4n} that $\ord_\fW((1-\ov{\phi}((r))/r^2)\msl) = 0.$ As $\ord_\fW(V_r) \geq 0$ by Proposition \ref{inte-2} , we conclude from \eqref{4.8n} that  Theorem \ref{AR-nonzero} holds when $k=1$.

\medskip

A curious new aspect of the argument now arises when we try to carry out the inductive argument for $k \geq 2$. For each positive divisor $d$ of $R$, we define
\begin{equation}\label{4.6n}
\msl(d) = \sqrt{d}L(\ov{\phi}_d, 1)/\Omega_\infty(A), \, \, \msl = \msl(1).
\end{equation}
Proposition \ref{3.3n} shows that $\msl(d)$ always belongs to the field $H\mst$, and $\msl \neq 0$ by Proposition \ref{4.4n}. However, the following stronger result is essential for our inductive argument.
\begin{prop}\label{4.7} Assume $R \in \msr$, and let $d$ be any positive divisor of $R$. Then $\msl(d)/\msl$ belongs to the field $\mst$.
\end{prop}
\begin{proof} For each integral ideal $\fa$ of $K$, which is prime to $R\fq$, we define
$$
e_\fa = \phi(\fa)/\xi(\fa).
$$
Then, as is already shown in \cite{GS} (see Corollary 4.11), $e_\fa$ only depends on the ideal class of $\fa$, and so, writing $\sigma$ for the Artin symbol of $\fa$
in $\fG = \Gal(H/K)$, we put $e_\sigma = e_\fa$. Now let $d$ be any positive divisor of $R$, so that $\msl(d)$ belongs to the field $H\mst$. Recall also that $\Gal(H\mst/\mst)$
is isomorphic $\Gal(H/K)$ under restriction, because $H \cap \mst = K$. If $\tau$ is any element of $\Gal(H\mst/\mst)$, we write $\tau_H$ for its restriction to $H$. Then it is proven in
Proposition 11.1 of  \cite{BG} that, for every $\tau$ in $\Gal(H\mst/\mst)$, we have
$$
\tau(\msl(d)) = e_{\tau_H} \msl(d).
$$
Since the factor $e_{\tau_H}$ is independent of $d$, it follows that $\msl(d)/\msl$ must belong to $\mst$, and the proof is complete.
\end{proof}

\medskip

\medskip

Assume now that $R = r_1\ldots r_k$, where $k \geq 2$. By Theorem \ref{key-identity}, we have
\begin{equation}\label{4.9n}
\msl(R)/\sqrt{R} + \sum_{d|R, d \neq 1, R}\Lambda(d, R)/\sqrt{d} + \msl \prod_{i=1}^{k}(1- \ov{\phi}((r_i))/r_i^2) = 2^kV_R,
\end{equation}
where $V_R =  \lim_{n\ra \infty}\sum_{\fa\in\fC_n} \xi(\fa)\Psi_{\fa, R}$, and
$$
\Lambda(d, R) = \msl(d) \prod_{r|R/d}(1-\ov{\phi}_d((r))/r^2).
$$
Now the terms $\Lambda(d, R)$ lie in an extension of $H\mst$ where the prime $\fP$ of $\mst$ is unramified but will usually have a large residue class field extension, and
this means one cannot carry through the inductive argument in its most naive form. The key to overcoming this difficulty is to divide both side of \eqref{4.9n} by the non-zero number $\msl$. Doing this, and defining, for each positive integer divisor $d$ of $R$, $\Phi(d, R) = \Lambda(d, R)/\msl$, we obtain the equation
\begin{equation}\label{4.10}
\Phi(R)/\sqrt{R} + \sum_{d|R, d \neq 1, R}\Phi(d, R)/\sqrt{d}  + \prod_{i=1}^{k}(1- \ov{\phi}((r_i))/r_i^2) = 2^kV_R/\msl,
\end{equation}
where $\Phi(R) = \msl(R)/\msl$. Let $H_R$ be the field defined in \eqref{6n}, and we now take $\fW$ to be any prime of the compositum $H_R\mst$ lying above $\fP$, so that $\fW/\fP$ is unramified. By Proposition \ref{inte-2}, we have  $\ord_\fW(V_R) \geq 0.$ Thus we conclude from Propositions  \ref{inte-2} and  \ref{4.4n} that $\ord_\fW(2^kV_R/\msl) \geq k+1$. Thanks to  Lemma \ref{3.1n}, we have
\begin{equation} \label{4.11}
\ord_\fW(\prod_{i=1}^{k}(1- \ov{\phi}((r_i))/r_i^2)) = k.
\end{equation}
On the other hand, our inductive hypothesis, together with Lemma \ref{3.1n} and Proposition \ref{4.4n},  shows  that, for each positive divisor $d$ of $R$, with $d \neq 1, R$, we have
\begin{equation}\label{4.12}
\ord_\fW(\Phi(d, R)/\sqrt{d}) = k.
\end{equation}
Of course, these estimates alone do not allow us to conclude from \eqref{4.10} that $\ord_\fW(\Phi(R)/\sqrt{R}) = k.$ However, the argument is saved by Proposition \ref{4.7},
which tells us that, for every positive divisor $d$ of $R$, $\Phi(d, R)$ belongs to the field $\mst$, and so it lies in the completion $\mst_\fP$ at $\fP$. Since $\fP$ has its residue
field of order 2, this means that we can write, for every positive divisor $d \neq 1, R$ of $R$,
\begin{equation}\label{4.13}
\Phi(d, R)/\sqrt{d} = \sqrt{d}\pi_\fP^k(1 + \pi_\fP b_d),
\end{equation}
where $\pi_\fP$ is a local parameter at $\fP$, and $\ord_\fP(b_d) \geq 0$. Thus
\begin{equation}\label{4.14}
\sum_{d|R, d \neq 1, R}\Phi(d, R)/\sqrt{d}  \equiv \pi_\fP^kD_R \mod \fW^{k+1},
\end{equation}
with  $D_R = \sum_{d|R. d \neq 1, R}\sqrt{d}$. But
$$
D_R^2 \equiv  \sum_{d|R. d \neq 1, R} d \mod \fW,
$$
and $\sum_{d|R. d \neq 1, R} d \equiv 2^k \mod 2$, whence $\ord_\fW(D_R) \geq 1$. Thus we have finally shown that
$$
ord_\fW(\sum_{d|R, d \neq 1, R}\Phi(d, R)/\sqrt{d} ) \geq k+1.
$$
It now follows from \eqref{4.10} and \eqref{4.11} that $\ord_\fW(\Phi(R)) = k$.  Thus, again using Proposition \ref{4.4n}, we
have finally proven Theorem \ref{AR-nonzero} by induction on the number of prime factors of $R \in \msr$.

\medskip

We end this section with a numerical example. Take $q = 23$. Then $K = \BQ(\sqrt{-23})$ has class number $h=3$. The Hilbert class field is $H = K(\alpha)$,
where $\alpha$ satisfies the equation $\alpha^3 - \alpha - 1 = 0.$  The following  global minimal Weierstrass equation for $A/H$ is given by Gross \cite{Gross0}
$$
y^2 + \alpha^3xy + (\alpha + 2)y = x^3 + 2x^2 - (12\alpha^2 + 27\alpha + 16)x - (73\alpha^2 + 99\alpha +62).
$$
Then $\msr$ will consist of all square free positive integers $R$ such that every prime factor $r$ of $R$ satisfies $r \equiv 1 \mod 4$ and
$r$ is congruent to one of $5, 7, 10, 11, 14, 15, 17, 19, 20, 21, 22 \mod 23$. Then Theorem \ref{main} shows that, for all $R \in \msr$, we have $L(A^{(R)}/H, 1) \neq 0.$ However,
we thank A. Dabrowski for pointing out the following interesting numerical example to us. Let $\beta = \sqrt{-23}$, and define $\fE$ to be the elliptic curve defined
over $H$, which is the twist of $A/H$ by the quadratic extension $H(\sqrt{-\beta})/H$. Theorem \ref{tm} in the case $q = 23$ asserts that
  $L(\fE/H, 1) \neq 0$. Now take $R= 901 = 17\times53$,  so that $R \in \fR$. Let $\fE^{(901)}/H$ be the twist of $\fE/H$ by the quadratic extension
$H(\sqrt{901})/H$. Then Dabrowski's calculations show that  $L(\fE^{(901)}/H, 1) = 0$. Thus the obvious analogue of Theorem \ref{main} does not hold for the curve $\fE/H$.

\section{Appendix}

In this Appendix, we establish a strengthening  of Theorem \ref{5.15}, and, for simplicity, we only treat the case $k=1$, which is needed for our non-vanishing results. Let $\chi$ be an arbitrary non-trivial character of finite order of $G = \Gal(F_\infty/K)$. Let $\mst_\chi$ be the field obtained by adjoining the values of $\chi$ to $\mst$, and we fix an embedding of $\mst_\chi$ into $\BC$, which extends the embedding \eqref{em}.  As in $\S4$ and $\S5$, we fix an embedding of the compositum $H\mst_\chi$ into the fraction field of $\msi$ which induces our fixed prime $w$
of $H$ lying above $\fp$, and the prime $\fP$ of $\mst$ above $\fp$. This is always possible because $H \cap \mst_\chi = K$, since $H$ does not contain any $2$-power roots
of unity of order greater than $2$. Let $r \geq 0$ be the largest integer $\geq 0$ such that  $\chi$ factor through $\Gal(F_r/K)$, so that the conductor $\fg_\chi$ of $\chi$ is equal to either $\fq\fp^{r+2}$ or $\fp^{r+2}$. Let $\zeta_r$ be the unique primitive $2^{r+2}$-th root of unity such that $j_w(\zeta_r - 1) = t_w(V_r)$, where $j_w$ is the isomorphism of formal groups given by \eqref{5.9}, and $V_r$ is the primitive $\fp^{r+2}$-division point on $A$ defined at the end of $\S2.$ Note that, for $\sigma \in \Gal(F_r/K)$, $\rho_\fP(\sigma)$ is a well defined element of $(\CO_\fp/\fp^{r+2})^\times$. We can then define the Gauss sum  $\tau(\chi)$ by
\begin{equation}\label{9.1}
\tau(\chi) = 2^{-(r+2)} \sum_{\sigma \in \Gal(F_r/K)} \zeta_r^{-\rho_\fP(\sigma)} \chi(\sigma).
\end{equation}
It is readily verified that, since $\chi$ is non-trivial, we have $|\tau(\chi)| = 2^{r/2 + 1}$. Recall that $\CW(z, \CL)$ denotes the Weierstrass isomorphism, and that $Q$ denotes the primitive $\fq$-division point on $A$ defined by \eqref{2.12}. We fix a complex number $z_\chi$ such that $\CW(z_\chi, \CL) = V_r \oplus Q.$ We will then have
$(z_\chi/\Omega_\infty(A))\CO_k = \fh_\chi/(\fq \fp^{r+2})$ for some integral ideal $\fh_\chi$ of $K$, which is prime to $\fp\fq$. For each $\alpha \in \msj$, we let $\mu_{\alpha, \infty}$ be the $\msi$-valued measure on $G$ defined by \eqref{5.14}. Write $L_\fq(\ov{\phi\chi},s)$ for the Euler product of the Hecke $L$-function of $\ov{\phi\chi}$, but with the Euler factor at the prime $\fq$-removed.

\begin{thm}\label{9.2}  For  each non-trivial character $\chi$ of finite order of $G$, the value $\Omega_\infty(A)^{-1}L(\ov{\phi\chi}, 1)$ belongs to $H\mst_\chi$, and we have
\begin{equation}\label{9.3}
\Omega_\fp(A)^{-1} \int_{G}\rho_\fP\chi d\mu_{\alpha, \infty} = \tau(\chi)(\phi\chi)(\fh_\chi)z_\chi^{-1}(N\alpha - (\phi\chi)((\alpha)))L_\fq(\ov{\phi\chi}, 1).
\end{equation}
\end{thm}
\begin{proof} We fix a set $\CB$ of integral ideals $\fb$ of $K$, prime to $\fp\fq$ such that
\begin{equation}\label{9.4}
\Gal(F_r/K) = \{\tau_\fb|F_r : \fb \in \CB \},
\end{equation}
where we recall that $\tau_\fb$ denotes the Artin symbol of $\fb$ in $\Gal(\fF_\infty/K).$ Recall that  $\rho_\fP(\Gal(F_\infty/F_r)) = 1+2^{r+2}\CO_\fp$, and note that
the characteristic function $\gamma_r(x)$ of $1+2^{r+2}\CO_\fp$ in $\CO_\fp$ is given explicitly by
\begin{equation}\label{9.5}
\gamma_r(x) = 2^{-(r+2)}\sum_{j \in \CO_\fp/\fp^{r+2}} \zeta_r^{(x-1)j}.
\end{equation}
For $r \leq n \leq \infty$, the measure $\mu_{\alpha, n, w}$ in $\Lambda_\msi(G)$ arises, via the isomorphism \eqref{5.7} and Mahler's theorem, from the power series
$J_{\alpha, n, w}(t_w)$ defined in Lemma \ref{5.6}. For every $\sigma \in G$, we write $\mu_{\alpha, n, w}^{(\sigma)}$ for the measure $\sigma \mu_{\alpha, n, w}$,
and note that $\mu_{\alpha, n, w}^{(\sigma)}$ arises, again via the isomorphism \eqref{5.7} and Mahler's theorem, from the power series $J_{\alpha, n, w}^{(\sigma)}(t_w)$
defined by
\begin{equation}\label{9.6}
J_{\alpha, n, w}^{(\sigma)}(t_w) = J_{\alpha, n, w}(\widehat{\rho_\fp(\sigma)}_w(t_w)).
\end{equation}
In what follows, for simplicity, we shall drop the fixed place $w$ of $H$ from the notation, and also we will just write $\mu_{\alpha, n}^{\fb}$ for the measure $(\tau_\fb|F_\infty)\mu_{\alpha, n, w}$.  As usual, we write $\chi(\fb^{-1})$ for $\chi(\tau_\fb^{-1})$ when $\fb \in \CB$, and we note that, thanks to our fixed embedding of $\mst$ into the fraction field of
$\msi$, we have $\rho_\fP(\tau_\fb) = \phi(\fb)$. It then follows easily from Mahler's theorem that
$$
\int_{G}\rho_\fP\chi d\mu_{\alpha, n} = \sum_{\fb \in \CB}(\phi\chi)(\fb^{-1})\int_{\CO_\fp}x\gamma_r(x)d\mu_{\alpha, n}^{\fb}.
$$
Substituting the formula \eqref{9.5} for $\gamma_r(x)$ in the integral on the right hand side,  we conclude that
\begin{equation}\label{9.7}
\int_{G}\rho_\fP\chi d\mu_{\alpha, n} = 2^{-(r+2)}\sum_{\fb \in \CB}(\phi\chi)(\fb^{-1})\sum_{j \in \CO_\fp/\fp^{r+2}}\int_{\CO_\fp}\zeta_r^{(x-1)j}xd\mu_{\alpha, n}^\fb.
\end{equation}
We omit the proof of the following classical lemma about the Mahler map $\BM$ which is defined earlier immediately below \eqref{5.12}. Let $\Delta$ denote the differential operator on $\msi[[W]]$ defined by $\Delta f(W) = (1+W)d/dW f(W).$
\begin{lem}\label{9.8} Let $\mu$ be any element of $\Lambda_\msi(\CO_\fp)$, and let $\BM(\mu)$ be its Mahler power series in $\msi[[W]]$. Then, if $\zeta$ denotes any 2-power root of unity, we have
\begin{equation}\label{9.9}
\Delta(\BM(\mu))(\zeta - 1) = \int_{\CO_\fp}x\zeta^x d\mu.
\end{equation}
\end{lem}
Now $\BM(\mu_{\alpha, n}^\fb) = \fK_{\alpha, n, w, \fb}(W)$, where
\begin{equation}\label{9.10}
\fK_{\alpha, n, w, \fb}(W)= J_{\alpha, n, w}(j_w((1+W)^{\phi(\fb)} - 1)).
\end{equation}
It then follows from the previous lemma that
\begin{equation}\label{9.11}
\int_{G}\rho_\fP\chi d\mu_{\alpha, n} = 2^{-(r+2)}\sum_{j \in \CO_\fp/\fp^{r+2}}\zeta_r^{-j}\fS_r(j),
\end{equation}
where
\begin{equation}\label{9.12}
\fS_r(j) = \sum_{\fb \in \CB}(\phi\chi)(\fb^{-1})(\Delta \fK_{\alpha, n, w, \fb})(\zeta_r^j-1).
\end{equation}
\begin{lem}\label{9.13} If 2 divides $j$, then $\fS_r(j) = 0$.
\end{lem}
\begin{proof}  In view of the definition \ref{9.10}, we see immediately that
\begin{equation}\label{9.14}
\fS_r(j) = \sum_{\fb \in \CB}\chi(\fb^{-1}) \frac{d}{dW}(J_{\alpha, n, w}\circ j_w)(\zeta_r^{\phi(\fb)j}-1)\zeta_r^{\phi(\fb)j}.
\end{equation}
Note that $\zeta_r^j$ is a $2^{r+1}$-th root of unity if 2 divides $j$. Now the restriction of the character $\chi$ to $\Gal(F_r/F_{r-1})$ is not the trivial character by the definition of $r$; here
we have put $F_{-1} = K$ if $r=0$. Thus we have $\sum_{\sigma \in \Gal(F_r/F_{r-1})}\chi(\sigma) = 0$. On the other hand, we have $\phi(\fb_1) \equiv \phi(\fb_2) \mod \fP^{r+1}$ if the restrictions of $\tau_{\fb_1}$ and $\tau_{\fb_2}$ to $F_{r-1}$ are equal. The assertion of the lemma is now clear from the explicit expression \eqref{9.14}
for $\fS_r(j)$.
\end{proof}
\begin{lem}\label{9.15} We have
\begin{equation}\label{9.16}
\int_{G}\rho_\fP\chi d\mu_{\alpha, n} = \tau(\chi) \sum_{\fb \in \CB}(\phi\chi)(\fb^{-1})(\Delta \fK_{\alpha, n, w, \fb})(\zeta_r - 1).
\end{equation}
\end{lem}
\begin{proof} Define
$$
W(\CB) = \sum_{\fb \in \CB}(\phi\chi)(\fb^{-1})(\Delta \fK_{\alpha, n, w, \fb})(\zeta_r - 1).
$$
We claim that $W(\CB)$ is independent of the choice of the set $\CB$ of integral ideals of $K$, prime to $\fp\fq$ such that $\Gal(F_r/K) = \{\tau_\fb|F_r : \fb \in \CB\}$. Indeed, we have
$$
W(\CB) = \sum_{\fb \in \CB}\chi(\fb^{-1}) \frac{d}{dW}(J_{\alpha, n, w}\circ j_w)(\zeta_r^{\phi(\fb)}-1)\zeta_r^{\phi(\fb)}.
$$
Now, if $\CB'$ is another set of integral ideals of $K$ with the above properties, then, for each $\fb \in \CB$ there exists a unique $\fb' \in \CB'$ such that $\tau_{\fb}|F_r = \tau_{\fb'}|F_r$. Then we have $\phi(\fb) \equiv \phi(\fb') \mod \fP^{r+2}$, so that $\zeta_r^{\phi(\fb)} = \zeta_r^{\phi({\fb'})}$, whence it is clear that $W(\CB) = W(\CB')$. Now choose a set $\CE$ of integral ideals of $K$, which are prime to $\fp\fq$, which are indexed by the odd integers $j$, with $1 \leq j < 2^{r+2}$, in such a way
that $\phi(\fe_j) \equiv j \mod \fP^{r+2}$. Then we clearly have
\begin{equation}\label{9.17}
\zeta_r^{-j}\fS_r(j) = \chi(\fe_j)\zeta_r^{-\phi(\fe_j)} W(\CB_j),
\end{equation}
where $\CB_j = \{\fb\fe_j: \fb \in \CB\}$. But, from our previous remark, we have $W(\CB_j) = W(\CB)$ for all odd integers $j$ with $1 \leq j < 2^{r+2}$, and \eqref{9.16}
follows.
\end{proof}

We recall that $j_w(\zeta_r -1) = t_w(V_r)$, where $V_r \in A_{\fp^{r+2}}$ is defined at the end of $\S2$, and for simplicity, put $Q_r = V_r \oplus Q$. Moreover, we have fixed a complex number $z_\chi$ such that $\CW(z_\chi, \CL) = Q_r$, and then the integral ideal $\fh_\chi$ satisfies  $(z_\chi/\Omega_\infty(A))\CO_k = \fh_\chi/(\fq \fp^{r+2})$. If $\fb$ is any integral ideal of $K$ prime to $\fp\fq$, we shall write $\tau_r(\fb)$ for the restriction of the Artin symbol $\tau_\fb$ of $\fb$ to the subfield $\fF_r$ of $\fF_\infty$. Moreover, for such an ideal $\fb$, we define now the partial $L$-series for the extension $F_r/K$ by
\begin{equation}\label{9.18}
L_\fq(\ov{\phi}^k, \tau_r(\fb), s) = \sum_{\tau_r(\fa) = \tau_r(\fb)} \frac{\ov{\phi}^k(\fa)}{(N(\fa))^s},
\end{equation}
where the sum is taken over all integral ideals $\fa$ of $K$, which are prime to $\fp\fq$, and which satisfy $\tau_r(\fa) = \tau_r(\fb)$. Note that this is a a different partial
$L$-series from that defined by \eqref{5.17}. We then have the following analogue of Proposition \ref{5.18}. Let $z_r \in \BC$ be such that $\CW(z_r, \CL) = V_r.$

\begin{prop} \label{9.19} For all integral ideals $\fd$ of $K$, which are prime to $\fp\fq$, we have
\begin{equation}\label{9.20}
\frac{d}{dz}\log \fR_{\alpha, A^\fd}(\eta_{A}(\fd)(\CW(z + z_r, \CL)) = \sum_{k \geq 1}(-1)^{k-1}(\phi(\fd\fh_\chi)/z_\chi)^k(N\alpha L_\fq(\ov{\phi}^k, \tau_r(\fd \fh_\chi), k) - \phi((\alpha))L_\fq(\ov{\phi}^k, \tau_r(\fd \fh_\chi(\alpha)), k)) z^{k-1}.
\end{equation}
\end{prop}
\begin{proof} It follows immediately from \eqref{m1}, \eqref{m2}, \eqref{m4} that, for all complex numbers $\zeta$, and all integral ideals $\fd$ of $K$ prime to $\fp\fq$, we have
\begin{equation}\label{9.21}
\frac{d}{dz}\log R_{\alpha, A^\fd}(\eta_A(\fd)(\CW(z + \zeta, \CL))) = \sum_{k \geq 1}(-1)^{k-1}\xi(\fd)^k(N\alpha\CE_k^*(\xi(\fd)\zeta, \CL_\fd) - \CE_k^*(\xi(\fd)\zeta, \alpha^{-1}\CL_\fd)z^{k-1}.
\end{equation}
We now fix a set  $\fS$ of integral ideals of K prime to $\fp\fq$ such that $\Gal(H(A_{\fq\fp^{r+2}})/\fF_r) = \{\tau'_\fs|H(A_{\fq\fp^{r+2}}) : \fs \in \fS \}$; here we have written  $\tau'_\fs$ for the Artin symbol of $\fs$ in $\Gal(H(A_{\fq\fp^\infty})/K$. All of the ideals
$\fs \in \fS$ are norms of ideals of $H$ because their Artin symbols fix the Hilbert class field $H$ of $K$. Also, the Artin symbols of the ideals $\fs \in \fS$ fix the
point $V_r$, and $\Gal(H(A_{\fq\fp^{r+2}})/\fF_r)$ is isomorphic to $\Gal(H(A_\fq)/H)$ under restriction. It follows easily that
\begin{equation}\label{9.22}
\frac{d}{dz}\log (\fR_{\alpha, A^\fd}(\eta_{A}(\fd)(\CW(z + z_r, \CL))) = \sum_{\fs \in \fS} \frac{d}{dz}(\log R_{\alpha, A^\fd}(\eta_A(\fd)(\CW(z + \phi(\fs)z_\chi, \CL))).
\end{equation}
Hence we conclude from \eqref{9.21} that the right hand side of \eqref{9.22} is equal to
\begin{equation}\label{9.23}
\sum_{k \geq 1}(-1)^{k-1} \xi(\fd)^k\sum_{\fs \in \fS}(N\alpha \CE_k^*(\xi(\fd)\phi(\fs)z_\chi, \CL_\fd) - \CE_k^*(\xi(\fd)\phi(\fs)z_\chi, \alpha^{-1}\CL_\fd))z^{k-1}.
\end{equation}
On the other hand, noting that
$$
z_\chi\Omega_\infty(A)^{-1}\CO_K = \fh_\chi/(\fq\fp^{r+2}), \, \, \alpha z_\chi\Omega_\infty(A)^{-1}\CO_K = (\alpha)\fh_\chi/(\fq\fp^{r+2}),
$$
we deduce from Proposition 5.5 of \cite{GS} that
$$
\sum_{\fs \in \fS} \CE_k^*(\xi(\fd)\phi(\fs)z_\chi, \CL_\fd) = (\phi(\fd\fh_\chi)/\xi(\fd)z_\chi)^k L_\fq(\ov{\phi}^k, \tau_r(\fd \fh_\chi), k),
$$
$$
\sum_{\fs \in \fS} \CE_k^*(\xi(\fd)\phi(\fs)z_\chi, \alpha^{-1}\CL_\fd) = (\phi((\alpha)\fd\fh_\chi)/\xi(\fd)z_\chi)^k L_\fq(\ov{\phi}^k, \tau_r((\alpha)\fd \fh_\chi), k).
$$
Substituting these last two expressions into \eqref{9.23}, the equation \eqref{9.20} follows, and the proof of Proposition \ref{9.19} is complete. \end{proof}
\noindent Moreover, putting $z=0$ in \eqref{9.16}, we obtain the following corollary.

\begin{cor}\label{9.24}  For all integral ideals $\fd$ of $K$, which are prime to $\fp\fq$, we have
$$
\frac{d}{dz}\log \fR_{\alpha, A^\fd}(\eta_{A}(\fd)\CW(z, \CL))|_{z = z_r} = \phi(\fd\fh_\chi){z_\chi}^{-1}(N\alpha L_\fq(\ov{\phi}, \tau_r(\fd \fh_\chi), 1) - \phi((\alpha))L_\fq(\ov{\phi}, \tau_r(\fd \fh_\chi(\alpha)), 1)).
$$
\end{cor}

\medskip

Recalling that $\Gal(\fF_r/H)$ is isomorphic under restriction to $\Gal(F_r/K)$, we assume from now on that that we have chosen the set $\CB$ of integral ideals of $K$, prime to $\fp\fq$, such that
\begin{equation}\label{9.26}
\Gal(\fF_r/H) = \{\tau_r(\fb) : \fb \in \CB\}.
\end{equation}
Now take $n$ to be any finite integer $\geq r$, and let $\fC_n$ be the set of integral ideals of $K$, prime to $\fp\fq$, satisfying \eqref{y1.7}. Since $\Gal(\fF_n/K_n)$ is isomorphic to
$\Gal(\fF_r/F_r)$ under restriction, it follows that
\begin{equation}\label{9.27}
\Gal(\fF_r/K) = \{\tau_r(\fb\fc) : \fb \in \CB, \fc \in \fC_n\}.
\end{equation}
Now, for $\fb \in \CB$,  we have $\Delta \fK_{\alpha, n, w, \fb}(W) = \Phi_{1, n, \fb}(W) -  \Phi_{2, n, \fb}(W)$, where
\begin{equation}\label{9.28}
 \Phi_{1, n, \fb}(W) = \Delta \log (d_{\alpha, n}(j_w((1+W)^{\phi(\fb)}-1))), \,  \Phi_{2, n, \fb}(W) = \frac{1}{2} \Delta \log (d_{\alpha, n}^\delta(\widehat{\eta_{A,\fp,w}}(j_w((1+W)^{\phi(\fb)}-1)))).
\end{equation}
\begin{lem}\label{9.29}  For every integer $n \geq r$, we have
$$
 \sum_{\fb \in \CB} (\phi \chi)(\fb^{-1})\Phi_{1, n, \fb} (\zeta_r - 1) = \Omega_\fp(A)\phi(\fh_\chi)z_\chi^{-1} \sum_{\fb \in \CB} \chi(\fb^{-1})\sum_{\fc \in \fC_n}\phi(\fc)(N\alpha L_\fq(\ov{\phi}, \tau_r(\fb\fc\fh_\chi), 1) - \phi((\alpha))L_\fq(\ov{\phi}, \tau_r((\alpha)\fb\fc\fh_\chi), 1)).
 $$
\end{lem}
\begin{proof} Let $\fb$ be any element of $\CB$.  Since $\tau_r(\fb)$ fixes $H$, we see that $A^\fb = A^{\fb\fc}$, whence it follows easily from \eqref{nc5} that  $\Phi_{1, n, \fb}(W)$ is obtained by first substituting
$t_w = j_w(W)$ in the $t_w$-expansion of the rational function on $A/H$ given by
$$
\prod_{\fc \in \fC_n}\fR_{\alpha, A^{\fc\fb}}(\eta_A(\fc\fb)(P)),
$$
and then applying the operator $\Delta$ to the logarithm of this series in $W$. Hence, recalling \eqref{5.24} and the fact that $V_r = j_w(\zeta_r - 1)$, we conclude that
$$
\Phi_{1, n, \fb}(W)(\zeta_r - 1) = \Omega_\fp(A) \sum_{\fc \in \fC_n}\frac{d}{dz}\log\fR_{\alpha, A^{\fc\fb}}(\eta_A(\fc\fb)\CW(z, \CL))|_{z = z_r},
$$
and so the assertion of the lemma is now clear from Corollary \ref{9.24}.
\end{proof}

 \begin{lem}\label{9.30} For every integer $n \geq r$, we have
 \begin{equation}\label{9.31}
 \sum_{\fb \in \CB} (\phi \chi)(\fb^{-1})\Phi_{2, n, \fb}(\zeta_r - 1) = 0.
 \end{equation}
 \end{lem}
 \begin{proof} For each $\fb \in \CB$, we have
 $$
 2\Phi_{2, n, \fb}(\zeta_r - 1)  = \Omega_\fp(A)\frac{d}{dz}(\log D_{\alpha, n}^\delta(\eta_A^\fp(\fb)\eta_A(\fp)(P)))|_{P = V_r}.
 $$
 Recalling that, by \eqref{y1.9}, we have $\eta_A(\fp)(V_r) = V_{r-1}^\delta$, and noting that $A^{\fp\fb} = A^\fp$ because $\tau_r(\fb)$ acts trivially on $H$, it follows that
 \begin{equation}\label{9.32}
  2\Phi_{2, n, \fb}(\zeta_r - 1) = \Omega_\fp(A) \phi(\fb\fp) \frac{d}{dz}(\log D_{\alpha, n}^\delta(P))|_{P = \eta_{A^\fp}(\fb)(V_{r-1}^\delta)}.
\end{equation}
Hence we have
\begin{equation}\label{9.33}
2 \sum_{\fb \in \CB} (\phi \chi)(\fb^{-1})\Phi_{2, n, \fb}(\zeta_r - 1) = \phi(\fp)\sum_{\fb \in \CB}\chi(\fb^{-1})\frac{d}{dz}(\log D_{\alpha, n}^\delta(P))|_{P = \eta_{A^\fp}(\fb)(V_{r-1}^\delta)}.
\end{equation}
Let $\CB'$ be the subset of all those $\fb \in \CB$ such that $\tau_r(\fb)$ fixes $\fF_{r-1}$.  Then $ \eta_{A^\fp}(\fb)(V_{r-1}^\delta) = V_{r-1}^\delta$ for $\fb \in \CB'$, and  $\sum_{\fb \in \CB'} \chi(\fb^{-1}) = 0$ because the restriction of $\chi$ to $\Gal(F_r/F_{r-1})$ is non-trivial. Using these facts, the assertion \eqref{9.31} follows easily from \eqref{9.33}.
\end{proof}

We can now at last complete the proof of Theorem \ref{9.2}. Combining Lemmas \ref{9.15}, \ref{9.29}, \ref{9.30}, we see that, for each finite integer $n \geq r$, we have
\begin{equation}\label{9.34}
\Omega_\fp(A)^{-1}\int_{G}\rho_\fP\chi d\mu_{\alpha, n} = \tau(\chi)\phi(\fh_\chi)z_\chi^{-1} \sum_{\fb \in \CB} \chi(\fb^{-1})\sum_{\fc \in \fC_n}\phi(\fc)(N\alpha L_\fq(\ov{\phi}, \tau_r(\fb\fc\fh_\chi), 1) - \phi((\alpha))L_\fq(\ov{\phi}, \tau_r((\alpha)\fb\fc\fh_\chi), 1)).
\end{equation}
Note also that Corollary \ref{9.24} shows that , for all $\fc \in \fC_n$ and $\fb \in \CB$, the expression $z_\chi^{-1}(N\alpha L_\fq(\ov{\phi}, \tau_r(\fb\fc\fh_\chi), 1) - \phi((\alpha))L_\fq(\ov{\phi}, \tau_r((\alpha)\fb\fc\fh_\chi), 1))$ is integral at our fixed embedding of $H\mst$ in the fraction field of $\msi$. Now, letting $n \to \infty$, the left hand side of \eqref{9.34} converges to $\Omega_\fp(A)^{-1}\int_{G}\rho_\fP\chi d\mu_{\alpha, \infty}$. Recalling that
$\phi(\fc) \equiv 1 \mod \fP^{n+2}$ for all $\fc \in \fC_n$, it is clear that the right hand side of \eqref{9.34} converges to $\tau(\chi)(\phi\chi)(\fh_\chi)z_\chi^{-1}(N\alpha - (\phi\chi)((\alpha)))L_\fq(\ov{\phi\chi}, 1)$, and the proof of Theorem \ref{9.2} is complete.
\end{proof}

\medskip

\noindent John Coates,\\
Emmanuel College, Cambridge,\\
England.\\
{\it jhc13@dpmms.cam.ac.uk }

\medskip

\noindent Yongxiong Li,\\
Yau Mathematical Sciences Center,\\
Tsinghua University, \\
Beijing, China.\\
{\it liyx\_1029@mail.tsinghua.edu.cn}

\end{document}